\documentclass[a4paper]{amsart}
\usepackage[utf8]{inputenc}
\usepackage{amsmath,amssymb, comment}
\usepackage{graphicx,multirow,array,enumerate,wasysym}
\usepackage[left=3cm,right=3cm,bottom=3cm,top=3cm]{geometry}
\usepackage{amsthm}
\usepackage{natbib}
\usepackage{tikz}
\usepackage{microtype}
\usepackage{mathtools}
\usepackage{hyperref}

\bibpunct{[}{]}{,}{n}{}{;}

\newtheorem{theorem}{Theorem}[section]
\newtheorem{lemma}[theorem]{Lemma}
\newtheorem{proposition}[theorem]{Proposition}
\newtheorem{conjecture}[theorem]{Conjecture}
\newtheorem{corollary}[theorem]{Corollary}
\theoremstyle{definition}
\newtheorem{definition}[theorem]{Definition}
\newtheorem{example}[theorem]{Example}
\newtheorem{remark}[theorem]{Remark}

\newcommand{\SN}{\mathbb{N}}                    
\newcommand{\SZ}{\mathbb{Z}}                    
\newcommand{\SC}{\mathbb{C}}                    
\newcommand{\SP}{\mathbb{P}}                    %
\newcommand{\SA}{\mathbb{A}}                    %
\newcommand{\CZ}{\mathcal{Z}}                    %
\newcommand{\CC}{\mathcal{C}}                    %
\newcommand{\CF}{\mathcal{F}}                    
\newcommand{\Gr}{{\rm Gr}}
\newcommand{\frakg}{\mathfrak{g}}
\newcommand{\frakheis}{\mathfrak{heis}}
\newcommand{\ra}[1]{\kern-1.5ex\xrightarrow{\ \ #1\ \ }\phantom{}\kern-1.5ex}
\newcommand{\ras}[1]{\kern-1.5ex\xrightarrow{\ \ \smash{#1}\ \ }\phantom{}\kern-1.5ex}
\newcommand{\da}[1]{\bigg\downarrow\rlap{$\scriptstyle{#1}$}}

\newcommand{\hda}{\mathrel{\rotatebox[origin=c]{-90}{$\hookrightarrow$}}}
\newcommand{\hra}{\mathrel{\rotatebox[origin=c]{0}{$\hookrightarrow$}}}

\title[Euler characteristics of Hilbert schemes of points on surface singularities]{Euler characteristics of Hilbert schemes of points on simple surface singularities}

\author{Ádám Gyenge}
\address{Alfréd Rényi Institute of Mathematics, Hungarian Academy of Sciences, Reáltanoda utca 13-15, Budapest, H-1053,  Hungary}
\email{gyenge.adam@renyi.mta.hu}

\author{András Némethi}
\address{Alfréd Rényi Institute of Mathematics, Hungarian Academy of Sciences, Reáltanoda utca 13-15, Budapest, H-1053,  Hungary}
\email{nemethi.andras@renyi.mta.hu}

\author{Balázs Szendrői}
\address{Mathematical Institute, University of Oxford, Andrew Wiles Building, Woodstock Road, Oxford, OX2 6GG, United Kingdom}
\email{szendroi@maths.ox.ac.uk}



\begin{document}

\begin{abstract}
We study the geometry and topology of Hilbert schemes of points on the orbifold surface $[\SC^2/G]$, respectively the singular quotient surface $\SC^2/G$, where $G<\mathrm{SL}(2,\SC)$ is a finite subgroup of type $A$ or $D$. We give a decomposition of the (equivariant) Hilbert scheme of the orbifold into affine space strata indexed by a certain combinatorial set, the set of Young walls. The generating series of Euler characteristics of Hilbert schemes of points of the singular surface of type $A$ or $D$ is computed in terms of an explicit formula involving a specialized character of the basic representation of the corresponding affine Lie algebra; we conjecture that the same result holds also in type $E$. Our results are consistent with known results in type $A$, and are new for type $D$. 
\end{abstract}

\maketitle


\tableofcontents

\section{Orbifold singularities and their Hilbert schemes}

\subsection{Quotient surface singularities, Hilbert schemes and generating series}
\label{sec_notation}

Let $G<{\mathrm{GL}}(2,\SC)$ be a finite subgroup and denote by $\SC^2/G$ the corresponding quotient variety.
There are two different types of Hilbert scheme attached to this data. 
First, there is the classical Hilbert scheme $\mathrm{Hilb}(\SC^2/G)$ of the quotient space. This is the moduli space of ideals in $\mathcal{O}_{\SC^2/G}(\SC^2/G)=\SC[x,y]^G$ of finite colength. We call this the \textit{coarse Hilbert scheme of points}. It decomposes as
\[ \mathrm{Hilb}(\SC^2/G)=\bigsqcup_{m \in \SN}\mathrm{Hilb}^{m}(\SC^2/G)\]
into components that are quasiprojective but singular varieties indexed by ``the number of points'', the codimension $m$ of the ideal. Second, there is the moduli space of $G$-invariant finite colength subschemes of $\SC^2$, the invariant part of $\mathrm{Hilb}(\SC^2)$ under the lifted action of~$G$. This Hilbert scheme is also well known and is variously called the \textit{orbifold Hilbert scheme} \cite{young2010generating} or \textit{equivariant Hilbert scheme} \cite{gusein2010generating}.
We denote it by $\mathrm{Hilb}([\SC^2/G])$. This space also decomposes as
\[ \mathrm{Hilb}([\SC^2/G])=\bigsqcup_{\rho \in {\mathop{\rm Rep}}(G)}\mathrm{Hilb}^{\rho}([\SC^2/G]),\]
where
\[\mathrm{Hilb}^{\rho}([\SC^2/G])=  \{ I \in \mathrm{Hilb}(\SC^2)^G \colon H^0(\mathcal{O}_{\SC^2}/I) \simeq_G \rho \}\]
for any finite-dimensional representation $\rho\in {\mathop{\rm Rep}}(G)$ of $G$;
here $\mathrm{Hilb}(\SC^2)^G$ is the set of $G$-invariant ideals of $\SC[x,y]$, and $\simeq_G$ means $G$-equivariant isomorphism. Being components of fixed point sets of a finite group acting on smooth quasiprojective varieties, the orbifold Hilbert schemes themselves are smooth and quasiprojective \cite{cartan1957quotient}.

There is a natural pushforward map between the two kinds of Hilbert scheme: each $J \in \mathrm{Hilb}([\SC^2/G])$ can be mapped to its $G$-invariant part, giving a morphism~\cite[3.4]{brion2011invariant} 
\[ \begin{array}{rccl} p_\ast:& \mathrm{Hilb}([\SC^2/G])&\rightarrow &\mathrm{Hilb}(\SC^2/G) \\ & J&  \mapsto & J^G=J\cap \SC[x,y]^G\end{array}\]
called the \textit{quotient-scheme map}. There is also a set-theoretic pullback map, which however
does {\em not} preserve flatness in families, so it is not a morphism between the Hilbert schemes: 
the inclusion $i: \SC[x,y]^G \subset \SC[x,y]$ induces a pullback map on the ideals, 
and its image is contained in the set of $G$-equivariant ideals, leading to a map of sets
\[\begin{array}{rccl}  i^\ast  : &\mathrm{Hilb}(\SC^2/G)(\SC) & \rightarrow & \mathrm{Hilb}([\SC^2/G])(\SC)\\ & I & \mapsto & i^\ast I= \SC[x,y].I\end{array}\]
Since for $I\lhd \SC[x,y]^G$, we clearly have $(\SC[x,y].I)^G = I$, the composite $p_\ast \circ i^\ast$ is the identity on the set of ideals of the invariant ring. 

We collect the topological Euler characteristics of the two versions of the Hilbert scheme into two generating functions. Let $\rho_0,\ldots,\rho_n\in\mathop{\rm Rep}(G)$ denote the (isomorphism classes of) irreducible representations of $G$, with $\rho_0$ the trivial representation.
\begin{definition}
\begin{enumerate}[(a)]
\item The \textit{orbifold generating series} of the orbifold $[\SC^2/G]$ is 
\[Z_{[\SC^2/G]}(q_0,\ldots, q_n)= \sum_{m_0,\dots,m_n=0}^\infty \chi\left(\mathrm{Hilb}^{m_0 \rho_0 + \ldots +m_n \rho_n}([\SC^2/G]) \right)   q_0^{m_0}\cdot \ldots \cdot q_n^{m_n}.\]
\item The \textit{coarse generating series} of the singularity $\SC^2/G$ is 
\[Z_{\SC^2/G}(q)=\sum_{m=0}^\infty \chi\left(\mathrm{Hilb}^m(\SC^2/G)\right)q^m.\]
\end{enumerate}
\end{definition}

\begin{remark} For a smooth variety $X$, the generating series 
\[ Z_X(q)=\sum_{m=0}^\infty \chi\left(\mathrm{Hilb}^m(X)\right)q^m\] 
of the Euler characteristics of Hilbert 
schemes of points of $X$, as well as various refinements of this series, have been extensively studied. 
In particular, for a nonsingular curve $C$, we have MacDonald's result~\cite{macdonald1962poincare}
\[ Z_C(q)= (1-q)^{-\chi(C)},
\]
whereas for a nonsingular surface $S$ we have (a specialization of) G\"ottsche's formula~\cite{gottsche1990betti}
\begin{equation}\label{eq:goettsche} Z_S(q)=\left(\prod_{m=1}^{\infty}(1-q^m)^{-1}\right)^{\chi(S)}. 
\end{equation}
There are also results for higher-dimensional varieties~\cite{cheah1996cohomology}. 

For singular varieties~$X$, the series $Z_X(q)$ is much less studied. For a singular curve $C$ with a finite set $\{P_1, \ldots, P_k\}$ of planar singularities however, we have the beautiful conjecture of Oblomkov and Shende~\cite{oblomkov2012hilbert}, proved by Maulik~\cite{maulik2012stable}, which takes the form
\begin{equation}\label{formula:singcurve}
Z_C(q)= (1-q)^{-\chi(C)}\prod_{j=1}^k Z^{(P_i, C)}(q).
\end{equation}
Here $Z^{(P_i, C)}(q)$ are highly nontrivial local terms that depend only on the embedded topological type of the link of the singularity $P_i\in C$. 
\end{remark}

\subsection{Simple surface singularities}

In this paper we are only concerned with finite subgroups $G<\mathrm{SL}(2,\SC)$. See \cite{gyenge2016hilbert} for some partial results for some other finite groups. As it is well known, finite subgroups of $\mathrm{SL}(2,\SC)$ are classified into three types: type $A_n$ for $n\geq 1$, type $D_n$ for $n\geq 4$ and type $E_n$ for $n=6,7,8$. The type
of the singularity can be parametrized by a simply-laced irreducible Dynkin diagram with $n$ nodes, arising from an irreducible simply laced root system $\Delta$. 
We denote the corresponding group by $G_\Delta<\mathrm{SL}(2,\SC)$; all other data 
corresponding to the chosen type will also be labelled by the subscript $\Delta$. 
Irreducible representations $\rho_0,\ldots,\rho_n$ of $G_\Delta$ 
are then labelled by vertices of the affine Dynkin diagram associated with $\Delta$. The singularity $\SC^2/G_\Delta$ is known as a simple (Kleinian, surface) singularity; we will refer to the corresponding orbifold $[\SC^2/G_\Delta]$ as the simple singularity orbifold.

As we recall in Appendix~\ref{sect:aff_lie_hilb}, the following result is known. 
\begin{theorem}[\cite{nakajima2002geometric}]  
\label{thmgenfunct}
Let $[\SC^2/G_\Delta]$ be a simple singularity orbifold. Then its orbifold generating series can be expressed as
\begin{equation} Z_{[\SC^2/G_\Delta]}(q_0,\dots,q_n)=\left(\prod_{m=1}^{\infty}(1-q^m)^{-1}\right)^{n+1}\cdot\sum_{ \overline{m}=(m_1,\dots,m_n) \in \SZ^n } q_1^{m_1}\cdot\dots\cdot q_n^{m_n}(q^{1/2})^{\overline{m}^\top \cdot C_\Delta \cdot \overline{m}},\label{eq:orbi_main_formula}\end{equation}
where $q=\prod_{i=0}^n q_i^{d_i}$ with $d_i=\dim\rho_i$, 
and $C_\Delta$ is the finite type Cartan matrix corresponding to $\Delta$.
\end{theorem}

Our first main result is a strengthening of this theorem. 
Given a Dynkin diagram $\Delta$ of type $A$ or $D$, we will recall below in~\ref{subsectypeA}, respectively~\ref{subsec:PYW}, the definition of a certain combinatorial set, the {\em set of Young walls $\CZ_\Delta$ of type $\Delta$}. 

\begin{theorem}
\label{thmorb}
Let $[\SC^2/G_\Delta]$ be a simple singularity orbifold, where $\Delta$ is of type $A_n$ for $n \geq 1$ or $D_n$ for $n \geq 4$. 
Then there exists a decomposition 
\[\mathrm{Hilb}([\SC^2/G_\Delta]) = \displaystyle\bigsqcup_{Y\in\CZ_\Delta} \mathrm{Hilb}([\SC^2/G_\Delta])_Y \]
into locally closed strata indexed by the set of Young walls $\CZ_\Delta$ of the appropriate type. 
Each stratum is isomorphic to an 
affine space 
of a certain dimension, and in particular has Euler characteristic $\chi(\mathrm{Hilb}([\SC^2/G_\Delta])_Y)=1$. 
\end{theorem}

For type $A$, the set of Young walls is simply the set of finite partitions, represented as Young diagrams, equipped with a diagonal labelling. In this case, Theorem~\ref{thmorb} is well known; the decomposition in type $A$ is not unique, but depends on a choice of a one-dimensional 
subtorus of the full torus $(\SC^*)^2$ acting on the affine plane $\SC^2$. For completeness, we summarize the details in~\ref{subsectypeA}. On the other hand, the type $D$ case appears to be new; in this case, our decomposition is unique, there is no further choice to make.

\begin{remark} The orbifold Hilbert schemes of points for $G<\mathrm{SL}(2,\SC)$ are well known to be Nakajima quiver varieties for the corresponding affine quiver. As it was shown in~\cite{savage2011quiver}, certain Lagrangian subvarieties in Nakajima quiver varieties are isomorphic to quiver Grassmannians for the preprojective algebra of the same type, parametrizing submodules of certain fixed modules. On the other hand, results of the recent papers~\cite{lorscheid2015quivera,lorscheid2015quiverb} imply that every quiver Grassmannian of a representation of a quiver of affine type $D$ has a decomposition into affine spaces. The relation between this decomposition and ours deserves further investigation.
\end{remark} 

As we will explain combinatorially in~\ref{subsec:Anabacus}, respectively~\ref{subsect:coreabacus}, and via representation theory in~\ref{extbasic}-\ref{sect:affinecryst}, the right hand side of~\eqref{eq:orbi_main_formula} enumerates the set of Young walls $\CZ_\Delta$ of the appropriate type. Thus Theorem~\ref{thmorb} implies Theorem~\ref{thmgenfunct}.

\begin{remark} In type $A$, it is easy to refine formula~\eqref{eq:orbi_main_formula} to a formula involving the Betti numbers~\cite{fujii2005combinatorial}, or the motives~\cite{gusein2010generating}, of the orbifold Hilbert schemes. We leave the study of such a refinement in type $D$ to future work; compare Remark~\ref{rem:Dn:motivic}.
\end{remark}

The second main result of our paper is the following formula, which says that the coarse generating series is a very particular specialization of the orbifold one.

\begin{theorem}
\label{thmsing}
Let $\SC^2/G_\Delta$ be a simple singularity, where $\Delta$ is of type $A_n$ for $n \geq 1$ or $D_n$ for $n \geq 4$. Let $h^\vee$ be the (dual) Coxeter number of the corresponding finite root system (one less than the dimension of the corresponding simple Lie algebra divided by $n$). Then
\[ Z_{\SC^2/G_\Delta}(q)=\left(\prod_{m=1}^{\infty}(1-q^m)^{-1}\right)^{n+1}\cdot\sum_{ \overline{m}=(m_1,\dots,m_n) \in \SZ^n } \zeta^{m_1+m_2+ \dots +m_n}(q^{1/2})^{\overline{m}^\top \cdot C_\Delta \cdot \overline{m}},\]
where $\zeta=\exp\left({\frac{2 \pi i}{1+h^\vee}}\right)$ and $C_\Delta$ is the finite type Cartan matrix corresponding to $\Delta$.
\end{theorem}

Thus $Z_{\SC^2/G_\Delta}(q)$ is obtained from $Z_{[\SC^2/G_\Delta]}(q_0,\dots,q_n)$ by the substitutions 
\[q_1=\dots=q_n=\exp\left({\frac{2 \pi i}{1+h^\vee}}\right), \ \ q_0=q\exp\left({-\frac{2 \pi i}{1+h^\vee}}\sum_{i\neq 0} \mathrm{dim}\rho_i\right).\] 
In type $A$, the formula in Theorem~\ref{thmsing} is not new: it was proved directly (in a slight disguise) by Dijkgraaf and Sulkowski in \cite{dijkgraaf2008instantons} and also recently, using completely different methods, by Toda in \cite{toda2013s}. Our main contribution is the general Lie-theoretic formulation, as well as a proof in type $D$; we also provide a direct combinatorial proof in type $A$, which appears to be new.

One can check directly that the generating series in Theorem $\ref{thmsing}$ has also integer coefficients for $E_6$, $E_7$ and $E_8$ to a high power in $q$. This motivates the following. 

\begin{conjecture} Let $\SC^2/G_\Delta$ be a simple singularity of type $E_n$ for $n=6,7,8$. Let $h^\vee$ be the (dual) Coxeter number of the corresponding finite root system. Then, as for other types, 
\[ Z_{\SC^2/G_\Delta}(q)=\left(\prod_{m=1}^{\infty}(1-q^m)^{-1}\right)^{n+1}\cdot\sum_{ \overline{m}=(m_1,\dots,m_n) \in \SZ^n } \zeta^{m_1+m_2+ \dots +m_n}(q^{1/2})^{\overline{m}^\top \cdot C_\Delta \cdot \overline{m}},\]
where $\zeta=\exp\left({\frac{2 \pi i}{1+h^\vee}}\right)$ and $C_\Delta$ is the finite type Cartan matrix corresponding to $\Delta$.
\label{conj:typeE}
\end{conjecture}

The key tool in our proof of Theorem~\ref{thmsing} for types $A$ and $D$ is the combinatorics
of Young walls, in particular their abacus representation. 
We are not aware of such explicit combinatorics in type $E$. We hope to return 
to this question in later work. 

\begin{remark} We are dealing here with Hilbert schemes, parametrizing rank $r=1$ 
sheaves on the orbifold or singular surface. 
In the relationship between the instantons on algebraic surfaces and affine Lie algebras, 
level equals rank \cite{grojnowski1996instantons}.
Indeed the (extended) basic represenatation underlying the Young wall combinatorics (see Appendix) has level $l=1$. 
Thus the substitution above is by the root of unity $\zeta=\exp\left({\frac{2 \pi i}{l+h^\vee}}\right)$, with $l=1$ and $h^\vee$ the (dual) Coxeter number. 
There is an intriguing analogy here with the Verlinde formula, which uses a similar substitution, into
characters of Lie algebras, by a root of unity $\zeta=\exp({\frac{2 \pi i}{l+h^\vee}})$, where $l$ again is the level, and 
$h^\vee$ the (dual) Coxeter number of the root system of the Lie algebra of the gauge group. The geometric significance 
of this observation, if any, is left for future research.
\end{remark}

\begin{remark} Given the results above, it is easy to write down a global formula analogous 
to~\eqref{formula:singcurve} for a singular surface with canonical singularities. This formula, as well as 
its modularity, are discussed in the announcement~\cite{gyenge2015announcement}.
\end{remark}

\subsection{Some terminology and structure of the paper} 
\label{sec:terminology}

We work over the field $\SC$ of complex numbers.
We call a regular map $f\colon X\to Y$ a {\em trivial affine fibration with fibre $\SA^k$}, if there is an isomorphism $X\cong Y \times \SA^k$ with $f$ being the first projection. %

The structure of the rest of the paper is as follows. In Section~\ref{Ansect}, we give a new proof of Theorem~\ref{thmsing} in type $A$, which has the advantage that it generelizes away from that case. The rest of the paper treats the case of type~$D$. In Section~\ref{Dnidealsect}, we introduce  Schubert-style cell decompositions of Grassmannians of homogeneous summands of $\SC[x,y]$. In Section~\ref{sect:dnorbidec} we give a cell decomposition of the orbifold Hilbert scheme, proving Theorem \ref{thmorb}. In Section~\ref{Dnspecialsect}, we discuss some special subsets of the strata and their geometry. A decomposition of the coarse Hilbert scheme is given in Section~\ref{sect:coarse}. In Section~\ref{sect:dnabacus}, the proof of Theorem~\ref{thmsing} is completed using combinatorial enumeration. The representation theoretic background is briefly summarized in Appendix~\ref{sect:apprepr}. Some relevant facts on joins of projective varieties are discussed in Appendix~\ref{sec:joins}.

\subsection*{Acknowledgements} The authors would like to thank Gwyn Bellamy, Alastair Craw, Eugene Gorsky, Ian Grojnowski, Kevin McGerty, Iain Gordon, Tomas Nevins and Tam\'as Szamuely for helpful comments and discussions. \'A.Gy.~was partially supported by the \emph{Lend\"ulet program} (Momentum Programme) of the Hungarian Academy of Sciences and by ERC Advanced Grant LDTBud (awarded to Andr\'as Stipsicz). A.N.~was partially supported by OTKA Grants 100796 and K112735. B.Sz.~was partially supported by EPSRC Programme Grant EP/I033343/1.

\section{Type  \texorpdfstring{$A_n$}{An}}
\label{Ansect}

\subsection{Type  \texorpdfstring{$A$}{A} basics}

Let $\Delta$ be the root system of type $A_n$. Choosing a primitive $(n+1)$-st root of unity $\omega$, 
the corresponding subgroup $G_\Delta$ of $SL(2,\SC)$, a cyclic subgroup of order $n+1$, 
is generated by the matrix
\[ \sigma= 
\begin{pmatrix}
\omega & 0\\
0 & \omega^{-1} \\
\end{pmatrix}.
\]
All irreducible representations of $G_\Delta$ are one dimensional, and they are simply given by $\rho_j\colon \sigma\mapsto \omega^j$, for $j\in\{0, \ldots, n\}$. The corresponding McKay quiver is the cyclic Dynkin diagram of type $\widetilde A_n^{(1)}$. 

The group $G_\Delta$ acts on $\SC^2$; the 
quotient variety $\SC^2/G_\Delta$ has an $A_{n}$ singularity at the origin. 
The matrix $\sigma$ clearly commutes with the diagonal two-torus $T=(\SC^*)^2$, and so $T$ acts on the
quotient $\SC^2/G_\Delta$ and the orbifold $[\SC^2/G_\Delta]$. Consequently $T$ also acts on the orbifold
Hilbert scheme $\mathrm{Hilb}([\SC^2/G_\Delta])$ and the (reduced) coarse Hilbert scheme
$\mathrm{Hilb}(\SC^2/G_\Delta)$ as well. 

\subsection{Partitions, torus-fixed points and decompositions}
\label{subsectypeA}

Consider the set ${\mathbb N}\times {\mathbb N}$ of pairs of non-negative 
integers; we will draw this set as a set of blocks on the plane, occupying the non-negative quadrant.
Label blocks diagonally with $(n+1)$ labels $0, \ldots, n$ as in the picture; the block with coordinates
$(i,j)$ is labelled with $(i-j) \mod (n+1)$. We will call this the \emph{pattern of type~$A_n$}.

\begin{center}
\begin{tikzpicture}[scale=0.6, font=\footnotesize, fill=black!20]
  \draw (0, 0) -- (0,7);
  \foreach \x in {1,2,4,5,6,7,8}
    {
      \draw (\x, 0) -- (\x,6.2);
    }
    \draw (0,0) -- (9,0);
   \foreach \y in {1,2,4,5,6}
    {
         \draw (0,\y) -- (8.2,\y);
    }
    \draw (0.5,0.5) node {0};
    \draw (1.5,0.5) node {1};
    \draw (4.5,0.5) node {$n$$-$$1$};
    \draw (5.5,0.5) node {$n$};
    \draw (0.5,1.5) node {$n$};
    \draw (1.5,1.5) node {0};
    \draw (4.5,1.5) node {$n$$-$$2$};
    \draw (5.5,1.5) node {$n$$-$$1$};
    \draw (6.5,0.5) node {0};
    \draw (7.5,0.5) node {1};
    \draw (6.5,1.5) node {$n$};
    \draw (7.5,1.5) node {0};
    
    \draw (0.5,4.5) node {1};
    \draw (1.5,4.5) node {2};
    \draw (0.5,5.5) node {0};
    \draw (1.5,5.5) node {1};
    \draw(0.5,3) node {\vdots};
    \draw(1.5,3) node {\vdots};
    \draw(3,0.5) node {\dots};
    \draw(8.75,0.5) node {\dots};
    \draw(0.5,6.75) node {\vdots};
\end{tikzpicture}
\end{center}

Let ${\mathcal P}$ denote the set of partitions.
Given a partition $\lambda=(\lambda_1,\dots,\lambda_k)\in{\mathcal P}$, 
with $\lambda_1\geq \ldots \geq \lambda_k$ 
positive integers, we consider its Young (or Ferrers) diagram, the subset 
of ${\mathbb N}\times {\mathbb N}$ which consists of the $\lambda_i$ lowest blocks
in column $i-1$. The blocks in $\lambda$ also get labelled by the $n+1$ labels. 
Let $\CZ_\Delta$ denote the resulting set of $(n+1)$-labelled partitions, including the empty partition. 
For a labelled partition $\lambda\in \CZ_\Delta$, let $\mathrm{wt}_j(\lambda)$ denote the number of blocks 
in $\lambda$ labelled $j$, and define the multiweight of $\lambda$ to be 
$\underline{\mathrm{wt}}(\lambda)=(\mathrm{wt}_0(\lambda), \ldots, \mathrm{wt}_n(\lambda))$. 

\begin{proposition}
The torus $T$ acts with isolated fixed points on $\mathrm{Hilb}([\SC^2/G_\Delta])$, parametrized by the 
set $\CZ_\Delta$ of $(n+1)$-labelled partitions. 
More precisely, for $k_0, \ldots, k_n$ non-negative integers and $\rho=\oplus_{j=0}^n \rho_i^{\oplus k_i}$,
the $T$-fixed points on $\mathrm{Hilb}^\rho([\SC^2/G_\Delta])$ are parametrized by $(n+1)$-labelled partitions
of multiweight $(k_0, \ldots, k_n)$. 
\end{proposition}
\begin{proof}
We just sketch the proof, which is well known~\cite{ellingsrud1987homology, fujii2005combinatorial}. 
It is clear that the $T$-fixed points on $\mathrm{Hilb}([\SC^2/G_\Delta])$, which coincide with the $T$-fixed points 
on $\mathrm{Hilb}(\SC^2)$, are the monomial ideals in $\SC[x,y]$ of finite colength. 
The monomial ideals are enumerated in turn by Young diagrams of partitions. 
The labelling of each block gives the weight of 
the $G_\Delta$-action on the corresponding monomial, proving the refined statement. 
\end{proof}

\begin{corollary} There exist a locally closed decomposition, depending on a choice specified below,
of $\mathrm{Hilb}([\SC^2/G_\Delta])$ into strata indexed by the set of $(n+1)$-labelled partitions.
Each stratum is isomorphic to an affine space.
\end{corollary}
\begin{proof} Again, this is well known \cite{nakajima2005instanton}. Fixing a representation $\rho$, choose a sufficiently general
one-dimensional subtorus $T_0\subset T$ which has positive weight on both
$x$ and $y$. For general $T_0\subset T$, the fixed point set on $\mathrm{Hilb}^\rho([\SC^2/G_\Delta])$ is unchanged and 
in particular consists of a finite number of isolated points. Choosing positive weights on $x,y$ ensures that all limits of $T_0$-orbits at $t=0$
in $\mathrm{Hilb}([\SC^2/G_\Delta])$ exist, even though $\mathrm{Hilb}^\rho([\SC^2/G_\Delta])$ is non-compact. Since
$\mathrm{Hilb}^\rho([\SC^2/G_\Delta])$ is smooth, the result follows by taking the Bia\l ynicki-Birula decomposition
of $\mathrm{Hilb}^\rho([\SC^2/G_\Delta])$ given by the $T_0$-action.
\end{proof}

Denote by 
\[Z_{\Delta}(q_0,\dots,q_n) = \sum_{\lambda\in \CZ_\Delta} \underline{q}^{\underline{\mathrm{wt}}(\lambda)}\]
the generating series of $(n+1)$-labelled partitions, where we used multi-index notation 
\[\underline{q}^{\underline{\mathrm{wt}}(\lambda)}=\prod_{i=0}^nq_i^{\mathrm{wt}_i(\lambda)}.\] 
From either of the previous two statements, we immediately deduce the following. 
\begin{corollary} \label{cor:Acombinatorial} Let $[\SC^2/G_\Delta]$ be a simple singularity orbifold of type $A$. 
Then its orbifold generating series can be expressed as
\begin{equation} \label{eq:orbigeneq} Z_{[\SC^2/G_\Delta]}(q_0,\dots,q_n)=Z_{\Delta}(q_0,\dots,q_n).
\end{equation}
\end{corollary}

According to \cite{fujii2005combinatorial} the generating series of $(n+1)$-labelled partitions has the following form:
\begin{equation} Z_{\Delta}(q_0,\dots,q_n)=\frac{\sum_{ \overline{m}=(m_1,\dots,m_n) \in \SZ^k }^\infty q_1^{m_1}\cdot\dots\cdot q_n^{m_n}(q^{1/2})^{\overline{m}^\top \cdot C \cdot \overline{m}}}{\prod_{m=1}^{\infty}(1-q^m)^{n+1}},
\label{eq:orbiseran}
\end{equation}
where $q=q_0\cdot \dots \cdot q_n$ and $C$ is the (finite) 
Cartan matrix of type $A_n$; for a sketch proof, see the end 
of~\ref{subsec:Anabacus} below. In particular, (\ref{eq:orbigeneq}) and (\ref{eq:orbiseran}) imply Theorem \ref{thmgenfunct} for type $A$.

We now turn to the coarse Hilbert scheme. Let us define a subset $\CZ^0_\Delta$
of the set of $(n+1)$-labelled partitions $\CZ_\Delta$ as follows. 
An $(n+1)$-labelled partition $\lambda\in \CZ_\Delta$ will be called
{\it $0$-generated} (a slight misnomer, this should be really be ``complement-$0$-generated'') 
if the complement of $\lambda$ inside ${\mathbb N}\times {\mathbb N}$
can be completely covered by translates of  ${\mathbb N}\times
{\mathbb N}$ to blocks labelled $0$ contained in this complement. 
Equivalently, an $(n+1)$-labelled partition $\lambda$ is
$0$-generated, if all its addable blocks (blocks whose addition gives
another partition) are labelled~$0$. It is immediately seen that this
condition is equivalent to the corresponding monomial ideal $I\lhd\SC[x,y]$ being
generated by its invariant part $I\cap\SC[x,y]^{G_\Delta}$. Indeed, we have the following.

\begin{proposition}
The torus $T$ acts with isolated fixed points on $\mathrm{Hilb}(\SC^2/G_\Delta)$, which are in bijection 
with the set $\CZ^0_\Delta$ of $0$-generated $(n+1)$-labelled partitions. 
More precisely, for a non-negative integer $k$,
the $T$-fixed points on $\mathrm{Hilb}^k(\SC^2/G_\Delta)$ are parametrized by $0$-generated 
$(n+1)$-labelled partitions $\lambda$ with $0$-weight $\mathrm{wt}_0(\lambda)=k$. 
\end{proposition}
\begin{proof} This is immediate from the above discussion. The $T$-fixed points of 
$\mathrm{Hilb}(\SC^2/G_\Delta)$ are the monomial ideals $I$ of $\SC[x,y]^{G_\Delta}$ of finite colength. 
Inside $\SC[x,y]$, the ideals they generate correspond to partitions which are $0$-generated. The ring 
$\SC[x,y]^{G_\Delta}$ has a basis consisting of monomials with corresponding blocks labelled $0$ inside $\SC[x,y]$; 
thus the codimension of a monomial ideal $I$ inside $\SC[x,y]^{G_\Delta}$ is simply the number of blocks
denoted $0$.
\end{proof}
Denoting by
\[Z^0_{\Delta}(q) = \sum_{\lambda\in \CZ^0_\Delta} q^{\mathrm{wt}_0(\lambda)}\]
the corresponding specialization of the generating series of $0$-generated $(n+1)$-labelled partitions,
we deduce the following. 
\begin{corollary} Let $[\SC^2/G_\Delta]$ be a simple singularity orbifold of type $A$. 
Then the coarse generating series can be expressed as
\begin{equation} Z_{\SC^2/G_\Delta}(q)=Z^0_{\Delta}(q).
\end{equation}
\label{cor_part_An_coarse}
\end{corollary}

\begin{proof}[Proof of Theorem \ref{thmsing} for the $A_n$ case]
The (dual) Coxeter number of the type $A_n$ root system is $h^\vee=n+1$. Thus
Theorem \ref{thmsing} for this case follows from Corollary~\ref{cor_part_An_coarse},
formula~\eqref{eq:orbiseran}, and the combinatorial
Proposition~\ref{prop:ansubst} below, which computes the series $Z^0_{\Delta}(q)$. 
\end{proof}

\begin{remark} The single variable generating series $Z_{\SC^2/G_\Delta}$ in type $A$ was  
calculated by Toda in~\cite{toda2013s} using threefold machinery including a flop formula for 
Donaldson--Thomas invariants of certain Calabi--Yau threefolds. He does not mention any connection to Lie theory. 
The combinatorics, and the one-variable formula for $Z^0_{\Delta}(q)$, were
already known to Dijkgraaf and Sulkowski~\cite{dijkgraaf2008instantons}. They do not give the interpretation of the 
combinatorial formula in terms of Hilbert schemes, though they are clearly motivated by
closely related ideas. Their proof is different, using the method of Andrews~\cite{andrews1984generalized}
in place of the abacus combinatorics we use below. We believe that already in type $A$, our new proof is preferable
since it directly exhibits the clear connection between the orbifold and coarse generating series. 
Also, as we show later, this method generalizes away from type $A$.
\end{remark}

\subsection{Abacus of type \texorpdfstring{$A_n$}{An}}
\label{subsec:Anabacus}

We now introduce some standard combinatorics related to the type $A$ root system, which will allow us to relate the 
generating series $Z_{\Delta}$  
of $(n+1)$-labelled partitions to the specialized series $Z^0_{\Delta}$ of $0$-generated partitions.
We follow the notations of \cite{leclerc2002some}. 

The {\em abacus of type $A_n$} is the arrangement of the set of
integers in $(n+1)$ columns according to the following pattern.
\begin{center}
\begin{tabular}{c c c c c}
\vdots & \vdots & & \vdots & \vdots \\
$-2n-1$ & $-2n$ & \dots & $-n-2$  & $-n-1$ \\
$-n$ & $-n+1 $& \dots & $-1$  & 0 \\
1 & 2 & \dots & $n$ & $n+1$\\
$n+2$ & $n+3$ & \dots & $2n+1$ & $2n+2$\\
\vdots & \vdots & & \vdots & \vdots
\end{tabular}
\end{center}
Each integer in this pattern is called a {\em position}. For any integer $1 \leq k \leq n+1$ the 
set of positions in the $k$-th column of the abacus is called 
the $k$-th {\em runner}. An {\em abacus configuration} is a set of {\em beads}, 
denoted by $\bigcirc$, placed on the positions, with each position occupied by at most one bead. 

To an $(n+1)$-labelled partition $\lambda=(\lambda_1,\dots,\lambda_k)\in\mathcal{Z}_\Delta$ 
we associate its \emph{abacus representation} (sometimes also called \emph{Maya diagram})
as follows: place a bead in position $\lambda_i-i+1$ for all $i$, interpreting $\lambda_i$ as 0 for $i>k$. 
Alternatively, the abacus representation can be described by tracing the outer
profile of the Young diagram of a partition: the occupied positions
occur where the profile moves ``down'', whereas the empty
positions are where the profile moves ``right''.
In the abacus representation of a partition, the number of occupied positive positions is always equal 
to the number of absent nonpositive positions; we call such abacus configurations {\em balanced}. Conversely, 
it is easy to see that any balanced configuration represents a unique $(n+1)$-labelled partition, 
an element of $\mathcal{Z}_\Delta$.

For $n=0$, we obtain a representation of partitions on a single runner; this is sometimes called
the {\em Dirac sea} representation of partitions. 

The \emph{$(n+1)$-core} of a labelled partition $\lambda \in \mathcal{Z}_\Delta$
is the partition obtained from $\lambda$ by successively removing border strips of length $n+1$, 
leaving a partition at each step, until this is no longer possible. Here a {\em border strip}
is a skew Young diagram which does not contain $2 \times 2$ blocks and 
which contains exactly one $j$-labelled block for all labels $j$. The
removal of a border strip corresponds in the abacus representation to shifting one of the 
beads up on its runner, if there is an empty space on the runner above
it. In this way, the core of a partition corresponds to the bead configuration in which all the 
beads are shifted up as much as possible; this in particular shows that the $(n+1)$-core of 
a partition is well-defined. We denote by ${\mathcal C}_\Delta$ the set of $(n+1)$-core partitions, and
\[c\colon {\mathcal Z}_\Delta\to {\mathcal C}_\Delta\] the map which takes an $(n+1)$-labelled 
partition to its $(n+1)$-core. 


Given an $(n+1)$-core $\lambda$, we can read the $(n+1)$ runners of its abacus rrepresentation 
separately. These will not necessarily be balanced. The $i$-th one will be shifted from the 
balanced position by a certain integer number $a_i$ steps, which is negative if the shift is 
toward the negative positions (upwards), and positive otherwise. These numbers satisfy 
$\sum_{i=0}^n a_i=0$, since the original abacus configuration was balanced. 
The set $\{a_1,\dots,a_n\}$ completely determines the partition, so we get a bijection
\begin{equation}\label{typeA-cores} 
{\mathcal C}_\Delta \longleftrightarrow \left\{\sum_{i=0}^n a_i=0\right\}\subset\SZ^{n+1}.\end{equation}
We will represent an $(n+1)$-core partition by the corresponding $(n+1)$-tuple
$\underline{a}=(a_0,\dots,a_{n})$. 

On the other hand, for an arbitrary partition, on each runner we have a partition up to shift, 
so we get a bijection
\[ {\mathcal Z}_\Delta \longleftrightarrow {\mathcal C}_\Delta \times {\mathcal P}^{n+1}. \]
This corresponds to the structure of
formula~\eqref{eq:orbiseran} above; its denominator is the generating series of $(n+1)$-tuples of 
(unlabelled) partitions, whereas its numerator (after eliminating a variable) is exactly a sum over 
$\underline{a}\in {\mathcal C}_\Delta$. The multiweight of a core partition corresponding to an 
element $\underline{a}$ is given by the quadratic expression $Q(\underline{a})$ in the
exponent of the numerator of~\eqref{eq:orbiseran}. 
For more details, see Bijections 1-2 in~\cite[\S2]{garvan1990cranks}. 

\subsection{Relating partitions to 0-generated partitions}

The purpose of this section is to prove the following, completely combinatorial statement. 

\begin{proposition} Let $\Delta$ be of type $A_n$, and let $\xi$ be a
primitive $(n+2)$-nd root of unity. Then the generating series of 
$0$-generated partitions can be computed from that of all $(n+1)$-labelled ones
by the following substitution: 
\[Z^0_{\Delta}(q) = Z_{\Delta}(q_0,\dots,q_n)\Big|_{q_0=\xi^{-n}q , q_1=\dots=q_n=\xi}.\]
\label{prop:ansubst}
\end{proposition}

We start by combinatorially relating partitions to $0$-generated partitions. $\CZ^0_\Delta$ is clearly 
a subset of $\CZ_\Delta$, but there is also a map \[p\colon \CZ_\Delta\to\CZ^0_\Delta\]  
defined as 
follows: for an arbitrary partition $\lambda$, let $p(\lambda)$ be the smallest $0$-generated partition 
containing it.
Since the set of $0$-generated partitions is closed under intersection, $p(\lambda)$ is well-defined,
and it can be constructed as follows: $p(\lambda)$ is the complement of the unions of the translates of
${\mathbb N}\times {\mathbb N}$ to $0$-labelled blocks in the
complement of $\lambda$. It is clear that  $p(\lambda)$ can
equivalently be obtained by adding all possible addable blocks to
$\lambda$ of labels different from $0$. 

\begin{remark} The map $p$ can also be described in the language of ideals. 
If the monomial ideal $I\lhd\SC[x,y]$ corresponds to the partition $\lambda$,
then the monomial ideal $i^*p_*I = (I\cap \SC[x,y]^{G_\Delta}).\SC[x,y]\lhd\SC[x,y]$ corresponds to the 
partition~$p(\lambda)$.
\end{remark} 

\begin{lemma} The bead configurations corresponding to $0$-generated partitions are exactly those 
which have all beads right-justified on each row, with no empty position to the right of 
a filled position. The map $p\colon \CZ_\Delta\to\CZ^0_\Delta$ can 
be described in the abacus representation by the process of pushing
all beads of an abacus configuration as far right as possible. 
\label{lem:rows}
\end{lemma}
\begin{proof}
This follows from the description of the map from a partition to its
abacus representation using the profile of the partition. Indeed, a $0$-generated partition
has a profile which only turns from ``down'' to ``right'' at
$0$-labelled blocks. In other words, the only time when a string of
filled positions can be followed by an empty position is when the last
filled position is on the rightmost runner. In other words, there
cannot be empty positions to the right of filled positions in a row. The proof
of the second statement is similar.
\end{proof}

\begin{remark} As explained above, the maps $c\colon \CZ_\Delta\to{\mathcal C}_\Delta$ and 
$p\colon \CZ_\Delta\to\CZ^0_\Delta$ have natural descriptions on abacus configurations: $c$ corresponds 
to pushing beads all the way up within their column, whereas $p$ corresponds to pushing beads all the way to the right within their row. It is then clear that there is also a third map $\CZ_\Delta\to{}^0\!\CZ_\Delta\subset\CZ_\Delta$, dual to $p$, defined on the abacus by pushing beads all the way to the left. On labelled partitions this corresponds to the operation of removing all possible blocks with labels different from $0$. This dual constuction occured in the literature earlier in \cite{gordon2008quiver}.
\end{remark}

\begin{proof}[Proof of Proposition \ref{prop:ansubst}] 
We will prove the substitution formula 
on the fibres of the map $p\colon \CZ_\Delta\to\CZ^0_\Delta$. In other words, we need
to show that for any given $\lambda_0 \in \CZ^0_\Delta$, we have
\begin{equation} \sum_{\mu \in p^{-1}(\lambda_0)} \underline{q}^{\underline{\mathrm{wt}}(\mu)}\Big|_{q_1=\dots=q_n=\xi,q_0=\xi^{-n}q}=q^{\mathrm{wt}_0(\lambda_0)}.\label{form:oneatatime}\end{equation}

As a first step, we reduce the computation to $0$-generated cores. Given an arbitrary 
$0$-generated partition $\lambda$, by the 
first part of Lemma~\ref{lem:rows} its core $\nu=c(\lambda)$ is also $0$-generated, and 
the corresponding abacus configuration can be obtained by permuting the rows of the configuration of $\lambda$. 
Fix one such permutation $\sigma$ of the rows. Then, using the second part of 
Lemma~\ref{lem:rows}, we can use the row permutation $\sigma$ to define a bijection 
\[\tilde \sigma\colon  p^{-1}(\lambda)\to p^{-1}(\nu)\] 
between (abacus representations of) partitions in the fibres, mapping $\lambda$ itself to $\nu$. 

The difference between the partitions $\lambda$ and $\nu$ is a certain number of border strips, 
each removal represented by pushing up one bead on some runner by one step. Each border 
strip contains one block of each label, so the total number of times
we need to push up a bead by one step on the different runners is $N=\mathrm{wt_0}(\lambda)-\mathrm{wt_0}(\nu)$.
Thus, with $q=q_0\cdot\ldots\cdot q_n$ as in the substitution above, we can write
\[ \underline{q}^{\underline{\mathrm{wt}}(\lambda)}=q^{\mathrm{wt_0}(\lambda)-\mathrm{wt_0}(\nu)}\underline{q}^{\underline{\mathrm{wt}}(\nu)}.
\]
On the other hand, it is easy to see that in fact for any $\mu\in p^{-1}(\lambda)$, the corresponding $\tilde \sigma(\mu)$
can also be obtained by pushing up beads exactly $N$ times, one step at a time, 
the difference being just in the runners on which these shifts are performed. This means that each 
$\mu$ differs from $\tilde \sigma(\mu)$ by the same number 
$N=\mathrm{wt}(\lambda)-\mathrm{wt}(\nu)$ of border strips. 
Therefore, we have 
\[ \sum_{\mu \in p^{-1}(\lambda)} \underline{q}^{\underline{\mathrm{wt}}(\mu)}=q^{\mathrm{wt_0}(\lambda)-\mathrm{wt_0}(\nu)}\sum_{\mu \in p^{-1}(\nu)} \underline{q}^{\underline{\mathrm{wt}}(\mu)}.\]
This is clearly compatible with \eqref{form:oneatatime} and reduces the argument to $0$-generated core 
partitions. 

Fix a $0$-generated core $\lambda\in {\mathcal Z}_\Delta^0\cap {\mathcal C}_\Delta$;
using Lemma~\ref{lem:rows} again, 
the corresponding $(n+1)$-tuple is a set of {\em nondecreasing} 
integers $\underline{a}=(a_0,\dots,a_{n})$ summing to $0$. 
The fibre $p^{-1}(\lambda)$ consists of partitions whose abacus representation contains the same number of 
beads in each row as $\lambda$. The shift of one bead to the left results in the removal in the partition 
of a block labelled $i$, with $1 \leq i \leq n$. After substitution, this multiplies the contribution of 
the diagram on the right hand side of \eqref{form:oneatatime} by $\xi^{-1}$. 
If we fix all but one row, which contains $k$ beads, then these contributions add up to
 \[\sum_{n_1=0}^{n-k+1} \sum_{n_2=0}^{n_1}\dots\sum_{n_k=0}^{n_{k-1}}(\xi^{-1})^{n_1+\dots+n_k}=\binom{n+1}{k}_{\xi^{-1}}, \]
 where $\binom{m}{r}_{z}=\frac{[m]_z!}{[r]_z![m-r]_z!}$ is the Gaussian binomial coefficient, 
with $[m]_z=\frac{1-z^m}{1-z}$.

The number of rows containing exactly $k$ beads in the configuration corresponding to $\lambda$ is $a_{n+1-k}-a_{n-k}$. 
Therefore, the total contribution of the preimages, the left hand side of \eqref{form:oneatatime}, is 
 \begin{align*}  \sum_{\mu \in p^{-1}(\lambda)} \underline{q}^{\underline{\mathrm{wt}}(\mu)}\Big|_{q_1=\dots=q_n=\xi,q_0=\xi^{-n}q} &
 =\prod_{k=1}^n \binom{n+1}{k}_{\xi^{-1}}^{a_{n+1-k}-a_{n-k}}\underline{q}^{\underline{\mathrm{wt}}(\lambda)}\Big|_{q_1=\dots=q_n=\xi,q_0=\xi^{-n}q}\\
 & = \prod_{l=0}^n \left(\frac{\binom{n+1}{n+1-l}_{\xi^{-1}}}{ \binom{n+1}{n-l}_{\xi^{-1}}}\right)^{a_{l}}\underline{q}^{\underline{\mathrm{wt}}(\lambda)}\Big|_{q_1=\dots=q_n=\xi,q_0=\xi^{-n}q}\\
 & =\prod_{l=0}^n \left(\frac{1-\xi^{-l-1} }{1-\xi^{l-n-1}}\right)^{a_{l}}\underline{q}^{\underline{\mathrm{wt}}(\lambda)}\Big|_{q_1=\dots=q_n=\xi,q_0=\xi^{-n}q}\\
 & =\prod_{l=1}^n \left(\frac{1-\xi^{-n-1} }{1-\xi^{-1}}\frac{1-\xi^{-l-1} }{1-\xi^{l-n-1}}\right)^{a_{l}}\underline{q}^{\underline{\mathrm{wt}}(\lambda)}\Big|_{q_1=\dots=q_n=\xi,q_0=\xi^{-n}q}
\\
 & =\xi^{-\sum_{l=1}^n la_{l}}\underline{q}^{\underline{\mathrm{wt}}(\lambda)}\Big|_{q_1=\dots=q_n=\xi,q_0=\xi^{-n}q},
 \end{align*}
 where in the second equality we used $\binom{n+1}{0}_z=\binom{n+1}{n+1}_z=1$, 
in the penultimate equality we used $a_0=-a_1-\dots-a_n$, and in the last equality we used 
 \[ \frac{1-\xi^{-n-1} }{1-\xi^{-1}}\frac{1-\xi^{-l-1} }{1-\xi^{l-n-1}}= \xi^{-l},\]
which can be checked to hold for $\xi$ a primitive $(n+2)$-nd root of unity. 
Incidentally, as the multiplicative order of $\xi$ is exactly $n+2$, 
all the denominators appearing above are non-vanishing.
Finally, according to \cite[\S2]{garvan1990cranks}, we have
 \[ \underline{q}^{\underline{\mathrm{wt}}(\lambda)}=q^{\frac{Q(\underline{a})}{2}}q_1^{a_1+\dots+a_n} \cdot \ldots \cdot  q_n^{a_n}, \]
where again $q=q_0\cdot \ldots \cdot q_n$ and  $Q: \SZ^n \rightarrow \SZ$ is the quadratic form associated to $C_{\Delta}$. 
Since $q_0$ appears only in $q$ on the right hand side, it is clear that $\frac{Q(\underline{a})}{2}=\mathrm{wt}_0(\lambda)$. Hence,
 \[ q^{\frac{Q(\underline{a})}{2}}q_1^{a_1+\dots+a_n} \dots   q_n^{a_n}\Big|_{q_1=\dots=q_n=\xi}=q^{\mathrm{wt}_0(\lambda)}\xi^{\sum_{l=1}^n la_{l}}. \]
This concludes the proof.
\end{proof}

\section{Type \texorpdfstring{$D_n$}{Dn}: ideals and Young walls}
\label{Dnidealsect}

\subsection{The binary dihedral group}

Fix an integer $n\geq 4$, and let~$\Delta$ be the root system of type $D_n$. 
For $\varepsilon$ a fixed primitive $(2n-4)$-th root of unity, the corresponding subgroup
$G_\Delta$ of $SL(2,\SC)$ can be generated by the following two elements $\sigma$ and $\tau$:
\[ \sigma= 
\begin{pmatrix}
\varepsilon & 0\\
0 & \varepsilon^{-1} \\
\end{pmatrix}, \qquad
\tau=\begin{pmatrix}
0 & 1\\
-1 & 0 \\
\end{pmatrix}.
\]
The group $G_\Delta$ has order $4n-8$, and is often called the binary dihedral group.
We label its irreducible representations as shown in Table \ref{dnchartable}. There is a distinguished
2-dimensional representation, the defining representation $\rho_{\mathrm{nat}}=\rho_2$.
See \cite{ito1999hilbert, decelis2008dihedral} for more detailed information. 
 
We will often meet the involution on the set of representations of $G_\Delta$ 
which is given by tensor product with the sign representation
$\rho_1$: on the set of indices $\{0,\dots,n\}$, this is the involution $j\mapsto \kappa(j)$ which swaps $0$ and $1$ and $n-1$ 
and $n$, fixing other values $\{2,\dots, n-2\}$. 
Given $j \in \{0,\dots,n\}$, we denote $\kappa(j,k)=\kappa^{kn}(j)$;
this is an involution which is nontrivial when $k$ and $n$ are odd, and trivial otherwise. 
The special case $k=1$ will also be denoted as $\underline{j}=\kappa^n(j)$.

The following identities will be useful:
\begin{equation}
\label{eq:repstensor}
\rho_{n-1}^{\otimes 2} \cong \rho_{n}^{\otimes 2} \cong \rho_{\underline{0}}, \ \ \rho_{n-1} \otimes \rho_{n} \cong \rho_{\underline{1}}, \ \ \rho_1 \otimes \rho_{n-1} \cong \rho_{n}, \ \ \rho_1 \otimes \rho_{n} \cong \rho_{n-1}, \ \ \rho_1^{\otimes 2}\cong \rho_0.
\end{equation}

\begin{table}
\begin{center}
\begin{tabular}{|c|c c c|}
\hline
$\rho$ & $\mathrm{Tr}(1)$ & $\mathrm{Tr}(\sigma)$ & $\mathrm{Tr}(\tau)$ \\
\hline
$\rho_0$ & 1 & 1 & 1\\
$\rho_1$ & 1 & 1 & $-1$ \\
$\rho_2$ & 2 & $\varepsilon+\varepsilon^{-1} $ & 0\\
\vdots & & \vdots & \\
$\rho_{n-2}$ & 2 & $\varepsilon^{n-3}+\varepsilon^{-(n-3)} $ & 0\\
$ \rho_{n-1}$ & 1 & $-1$ & $-i^n$\\
$ \rho_{n}$ & 1 & $-1$ & $i^n$\\
\hline
\end{tabular}
\vspace{0.2in}
\caption{Labelling the representations of the group $G_\Delta$}
\label{dnchartable}
\end{center}
\end{table}

\subsection{Young wall pattern and Young walls}
\label{subsec:PYW}

We describe here the type $D$ analogue of the set of labelled partitions used in type $A$, following \cite{kang2004crystal,kwon2006affine}. 
In this section, we only describe the combinatorics; see Appendix \ref{sect:apprepr} for the representation-theoretic significance of 
this set. 

First we define the {\em Young wall pattern of type\footnote{The
    combinatorics introduced in this section should really be called 
type $\tilde{D}_n^{(1)}$, but we do not wish to overburden the notation. Also we have reflected the pattern in a vertical axis compared to the pictures of \cite{kang2004crystal,kwon2006affine}.}  $D_n$}, 
the analogue of the $(n+1)$-labelled positive quadrant 
lattice of type $A_{n}$ used above. This is the following infinite
pattern, consisting of two types of blocks: half-blocks carrying possible labels
$j\in\{0,1,n-1, n\}$, and full blocks carrying possible labels $1<j<n-1$:

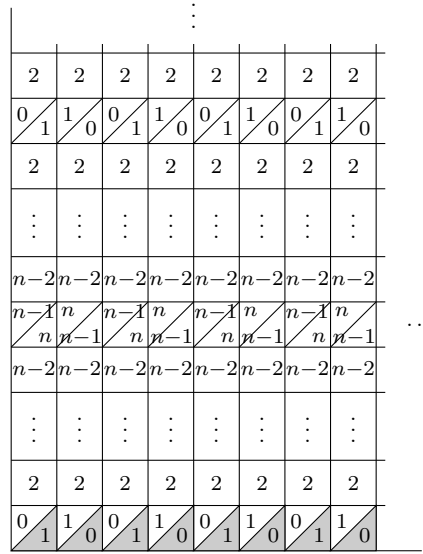
\begin{center}
\begin{tikzpicture}[scale=0.6, font=\footnotesize, fill=black!20]
  \foreach \x in {1,2,3,4,5,6,7,8}
    {
      \draw (\x, 0) -- (\x,11+0.2);
    }
     \draw (0, 0) -- (0,12);
   \foreach \y in {1,2,3.5,4.5,5.5,6.5,8,9,10,11}
    {
         \draw (0,\y) -- (8.2,\y);
    }
    \draw (0,0) -- (9,0);
    \foreach \x in {0,1,2,3,4,5,6,7}
    {
    	\draw (\x,4.5) -- (\x+1,5.5);
    	\draw (\x,9) -- (\x+1,10);
    	\draw (\x+0.5,1.5) node {2};
    	\draw (\x+0.5,4) node {$n$$-$$2$};
    	\draw (\x+0.5,6) node {$n$$-$$2$};
    	\draw (\x+0.5,8.5) node {2};
    	\draw (\x+0.5,10.5) node {2};
    	\draw(\x+0.5,2.85) node {\vdots};
    	\draw(\x+0.5,7.35) node {\vdots};
    	\filldraw (\x,0) -- (\x+1,1) -- (\x+1,0) -- cycle ;
    }
     \foreach \x in {0,2,4,6}
        {
        	\draw (\x+0.25,0.65) node {0};
        	\draw (\x+0.75,0.35) node {1};
        	\draw (\x+0.49,5.28) node  {$n$$-$$1$};
        	\draw (\x+0.75,4.72) node {$n$};
        	\draw (\x+0.25,9.65) node {0};
        	\draw (\x+0.75,9.35) node {1};
        }
       \foreach \x in {1,3,5,7}
             {
             	\draw (\x+0.25,0.65) node {1};
             	\draw (\x+0.75,0.35) node {0};
             	\draw (\x+0.25,5.28) node {$n$};
             	\draw (\x+0.52,4.72) node {$n$$-$$1$};
             	\draw (\x+0.25,9.65) node {1};
             	\draw (\x+0.75,9.35) node {0};
             }
             
       \draw (9,5) node {\dots};
       \draw (4,12) node {\vdots};
\end{tikzpicture}
\end{center}

Next, we define the set of {\em Young walls\footnote{In \cite{kang2004crystal,kwon2006affine}, these arrangements are called {\em proper Young walls}. Since we will not meet any other Young wall, we will drop the adjective {\em proper} for brevity.} of type $D_n$}. A Young wall of type $D_n$ is a 
subset~$Y$ of the infinite Young wall of type $D_n$, satisfying the following rules. 
\begin{enumerate}
\item[(YW1)] $Y$ contains all grey half-blocks, and a finite number of the white blocks and half-blocks. 
\item[(YW2)] $Y$ consists of continuous columns of blocks, with no block placed on top of a missing block or half-block. 
\item[(YW3)] Except for the leftmost column, there are no free
  positions to the left of any block or half-block. Here the rows of
  half-blocks are thought of as two parallel rows; only half-blocks of the same orientation have to be present.
\item[(YW4)] A full column is a column with a full block or both half-blocks present at its top; then no two full columns 
have the same height\footnote{This is the properness condition of \cite{kang2004crystal}.}.
\end{enumerate}

Let $\CZ_\Delta$ denote the set of all Young walls of type $D_n$. For any $Y\in \CZ_\Delta$ and label $j\in \{0, \ldots, n\}$ let $wt_j(Y)$ be the number of white half-blocks, respectively blocks, of label $j$. These are collected into the multi-weight vector $\underline{wt}(Y)=(wt_0(Y), \ldots, wt_n(Y))$. The total weight of $Y$ is the sum
\[|Y|=\sum_{j=0}^{n} wt_j(Y),
\]
and for the formal variables $q_0,\dots,q_n$,
\[\underline{q}^{\underline{wt}(Y)}=\prod_{j=0}^{n}q_j^{wt_j(Y)}.\]

\subsection{Decomposition of \texorpdfstring{$\SC[x,y]$}{C[x,y]} and the transformed Young wall pattern}

The group $G_\Delta$ acts on the affine plane $\SC^2$ via the defining representation $\rho_{\mathrm{nat}}=\rho_2$. Let $S=\SC[x,y]$ be the coordinate ring of the plane, then $S=\oplus_{m\geq 0} S_m$  where $S_m$ is the $m$th symmetric power of $\rho_{\mathrm{nat}}$, the space of homogeneous polynomials of degree $m$ of the coordinates $x,y$. 

We further decompose
\[  S_m = \bigoplus_{j=0}^{n} S_m[\rho_j]
\]
into subrepresentations indexed by irreducible representations. We will also use this notation for linear 
subspaces: for $U\subset S_m$ a linear subspace, $U[\rho_j]=U \cap  S_m[\rho_j]$.
We will call an element $f\in S$ {\em degree homogeneous}, 
if $f\in S_m$ for some $m$; we call it {\em degree and weight homogeneous}, if $f\in S_m[\rho_j]$ 
for some $m, j$. 

The decomposition of $S$ into $G_\Delta$-summands can be read off very conveniently from the
\textit{transformed Young wall pattern}. The transformation is an affine one, involving a shear: 
reflect the original Young wall pattern in the line $x = y$ in the plane, translate the $n$th row 
by $n$ to the right, and remove the grey triangles of the
original pattern. In this way, we get the following picture: 

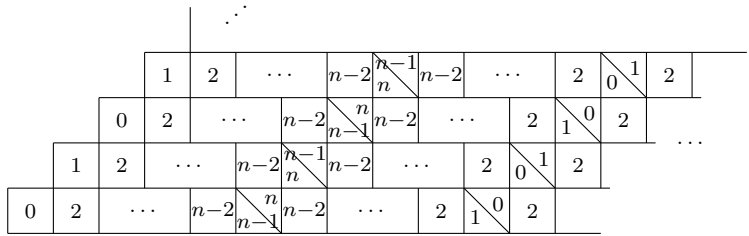
\begin{center}
\begin{tikzpicture}[scale=0.6, font=\footnotesize, fill=black!20]
  \draw (5,5) node {\reflectbox{$\ddots$}};
  \draw (15,2) node {\dots};
  \draw (4,4) -- (4,5);
  \draw (0,0 ) -- (13,0);
    \foreach \y in {1,2,3,4}
        {
          \draw (\y -1,\y ) -- (\y +12+0.2,\y);
        }
    \foreach \y in {0,1,2,3}
    {
    \foreach \x in {0,1,2,4,5,6,7,9,10,11,12}
                  {
                       \draw (\y+\x ,\y) -- (\y+\x ,\y+1);
                  }
    	\draw (\y+6,\y) -- (\y+5,\y+1);
    	\draw (\y+11,\y) -- (\y+10,\y+1);
    	\draw (\y+1.5,\y+0.5) node {2};
    	\draw (\y+4.5,\y+0.5) node {$n$$-$$2$};
    	\draw (\y+6.5,\y+0.5) node {$n$$-$$2$};
    	\draw (\y+9.5,\y+0.5) node {2};
    	\draw (\y+11.5,\y+0.5) node {2};
    	\draw(\y+3,\y+0.5) node {\dots};
    	\draw(\y+8,\y+0.5) node {\dots};
    }
     \foreach \y in {0,2}
        {
      	    \draw (\y+0.5,\y+0.5) node {0};
        	\draw (\y+5.49,\y+0.25) node  {$n$$-$$1$};
        	\draw (\y+5.78,\y+0.75) node {$n$};
        	\draw (\y+10.75,\y+0.65) node {0};
        	\draw (\y+10.25,\y+0.35) node {1};
        }
        \foreach \y in {1,3}
          {
              	    \draw (\y+0.5,\y+0.5) node {1};
                	\draw (\y+5.25,\y+0.25) node  {$n$};
                	\draw (\y+5.52,\y+0.75) node {$n$$-$$1$};
                	\draw (\y+10.75,\y+0.65) node {1};
                	\draw (\y+10.25,\y+0.35) node {0};
          }
      
\end{tikzpicture}
\end{center}

As it can be checked readily, this is a representation of $S$ and its
decomposition into $G_\Delta$-representations. The homogeneous components
$S_m$ are along the antidiagonals. For $1<i<n-1$, a full block
labelled $j$ below the diagonal, together with its mirror image,
correspond to a 2-dimensional representation $\rho_j$. For $j\in\{0,1,n-1,
n\}$, a full block labelled $j$ on the diagonal, as well as a half-block
labelled $j$ below the diagonal with its mirror
image, corresponds to a one-dimensional representation. 
The dimension of $S_m[\rho_j]$ is the same as the total number of
full blocks labelled $j$ on the $m$th diagonal in the transformed
Young wall pattern, counting mirror images also.

It is easy to translate the conditions (YW1)-(YW4) into the combinatorics 
of the transformed pattern; see Proposition~\ref{wallprop} and Remark~\ref{rem:corr_young_rules} below. Pictures of some small Young walls in the transformed pattern can be found below in Examples~\ref{ex:3sidepyr}-\ref{ex:GHilb} below.

\subsection{Subspaces and operators}
\label{subsec:subops}

For each non-negative integer $m$ and irreducible representation~$\rho_j$, consider the space $P_{m,j}$ of nontrivial 
$G_\Delta$-invariant subspaces of minimal dimension in $S_m[\rho_j]$. Specifically, if $\rho_j$ is one-dimensional, 
then these will be lines, and $P_{m,j}$ is simply the projectivization $\SP S_m[\rho_j]$.
If $\rho_j$ is two-dimensional, then $P_{m,j}$ is a closed subvariety of $\Gr(2,S_m[\rho_j])$. It is easy to 
see that in this case also, $P_{m,j}$ is isomorphic to a projective space. 

More generally, let $G_{m,j}^r$ be the space of $(r-1)$-dimensional projective subspaces of $P_{m,j}$. If $\rho_j$ is one-dimensional, then this is the Grassmannian $\Gr(r, S_m[\rho_j])$. When $\rho_j$ is two-dimensional, then $G_{m,j}^r$ is a closed subvariety of $\Gr(2 r,S_m[\rho_j])$ isomorphic to a Grassmannian of rank $r$. Clearly $G_{m,j}^1=P_{m,j}$.

For $0\leq j\leq n$, we introduce operators $L_j\colon \Gr(S)\to \Gr(S)$ 
on the Grassmannian $\Gr(S)$ of all linear subspaces of the vector space $S$ 
as follows: for $v\in\Gr(S)$, we set
\begin{enumerate}
\item $L_0 v = v$;
\item $L_1 v = xy\cdot v$;
\item for $1<j<n-1$, $L_j v = \langle x^{j-1} \cdot v, y^{j-1} \cdot v\rangle$;
\item $L_{n-1} v =(x^{n-2}-i^ny^{n-2})\cdot v$;
\item $L_{n} v =(x^{n-2}+i^ny^{n-2})\cdot v$.
\end{enumerate}
Sometimes we will use the notation $L_2=L_{2,x}+L_{2,y}$ for the $x$- and $y$-component of the operator $L_2$, i.e.~multiplication with $x$, respectively $y$. 
The operators above restrict to operators $L_0\colon \Gr(S_m)\to \Gr(S_m)$, $L_1\colon \Gr(S_m)\to\Gr(S_{m+2})$, 
$L_j\colon \Gr(S_m)\to \Gr(S_{m+j-1})$ for $1<j<n-1$, and $L_{n-1}, L_n\colon \Gr(S_m)\to\Gr(S_{m+n-2})$
on the Grassmannians of the graded pieces $S_m$.
To simplify notation, if we do not write the space to which these
operators are applied, then application to $\langle 1\rangle$ is
meant. So, for example, the symbol $L_1^2$ standing alone denotes the vector subspace $\langle
x^2y^2 \rangle$ of $S_4$, while $L_2$ alone denotes the two-dimensional
vector subspace $\langle x,y \rangle$ of $S_1$. For a linear subspace
$v$ of $S$, the sum $\sum_{j \in I}L_j v$ denotes the subspace of $S$
generated by the images $L_j v$. We use the operator notation also for
a set of subspaces; the meaning should be clear from the context. 

\subsection{Cell decompositions of equivariant Grassmannians}
\label{subsec:schubertcells}

We start this section by defining decompositions of the
Grassmannians $P_{m,j}$ of nontrivial 
$G_\Delta$-invariant subspaces of minimal dimension in $S_m[\rho_j]$. 
Given $(m,j)$, let $B_{m,j}$ denote the set of pairs of non-negative
integers $(k,l)$ such that $k+l=m$, $l \geq k$, and the block
position $(k,l)$ on the $m$-th antidiagonal on or below the main diagonal contains a block or 
half block of color $j$. Here $k$ is the row index, $l$ is the column index, and both of them in a nonnegative integer. It clearly follows from our setup that 
\[ \dim P_{m,j} = |B_{m,j}| -1.\]

\begin{proposition} Given $(m,j)$, there exists a locally closed stratification
\[  P_{m,j} = \displaystyle\bigsqcup_{(k,l)\in B_{m,j}} V_{k,l,j},\]
which is a standard stratification of the projective space $P_{m,j}$ into affine spaces $V_{k,l,j}$ of 
decreasing dimension.
\label{prop:decomp of P}
\end{proposition}

We will call $V_{k,l,j}$ the {\em cells} of $P_{m,j}$. 
The decomposition will be defined inductively, based on the following Lemma. 
Recall that $j\mapsto \kappa(j)$ denotes the involution on $\{0, \ldots, n\}$ which swaps $0$ and $1$ and 
$n-1$ and $n$. 

\begin{lemma} For any $l\geq 0$ and any $j\in[0,n]$, we have an injection 
 \[ L_1 \colon P_{l-2, j} \to P_{l, \kappa(j)}.\]
This map is an isomorphism except in the case when the block or half-block in the bottom row of
the transformed Young wall pattern on the $l$-th antidiagonal has label $j$, in which case the image has
codimension one. 
\label{lem:codim1_2}
\end{lemma}
\begin{proof} It is clear that multiplication by $L_1$ induces an injection, so we simply need to 
check the dimensions. The statement then clearly follows by looking at the transformed Young wall pattern: 
multiplication by $L_1$ corresponds to shifting the $(l-2)$-nd diagonal up by one diagonal step to the $l$-th diagonal;
the number of blocks or half-blocks labelled $j$ is identical, unless the new (half-)block has 
label $j$, and then the codimension is exactly one. 
\end{proof}
\begin{remark} The half-block in the bottom row of
the transformed Young wall pattern in the $l$-th antidiagonal has label $j=0,1$ for 
$l \equiv 0 \; \mathrm{mod} \; (2n-4)$ except at $(0,0)$ where only $0$ occurs. 
Half-blocks labelled $j=(n-1), n$ occur for $l \equiv n-2 \; \mathrm{mod} \; (2n-4)$.
For $j\in [2,n-2]$, there are full blocks labelled $j$ in the bottom row on antidiagonals for $l \equiv j-1 \mbox{ or } 2n-3-j  \; \mathrm{mod} \; (2n-4)$.
\end{remark}

\begin{proof}[Proof of Proposition~\ref{prop:decomp of P}]
Nontrivial cells $V_{0, l, j}$ need to be defined exactly when the block or half-block in the bottom 
row of the transformed Young wall pattern in the $l$-th antidiagonal has label~$j$. 
In these cases, we set the cells along the bottom row to be          
\[ V_{0, l, j} = P_{l, j} \setminus L_1 P_{l-2, \kappa(j)}.
\]
Once the cells $V_{0,k,j}$ along the bottom row are defined, we define
the general cells for all $0\leq j \leq n$, all $l$ and $k$ by 
\[V_{k,k+l,j}=L_1^k V_{0,l,\kappa^k(j)}. \]
What this says is that the cells are shifted up diagonally by $L_1$, taking into account that $L_1$ 
multiplies by the sign representation, so shifts the indices by the appropriate power of the 
involution~$\kappa$. By induction, we obtain a decomposition of $P_{m,j}$ with the stated properties.
\end{proof}

As it is well known, a decomposition of a projectivization of a vector
space into affine cells is equivalent to giving a flag in the space itself. This induces a natural decomposition of all higher rank Grassmannians into \emph{Schubert cells}, which are known to be affine. Thus our cell decomposition of $P_{m,j}$ induces cell decompositions of all $G_{m,j}^r$. Since the cells in the first decomposition are indexed by the set $B_{m,j}$, the cells in the second will be indexed by subsets of $B_{m,j}$ of size $r$. 
A Schubert cell of $G_{m,j}^r$  corresponding to a subset $S=\{(k_1,l_1),\dots (k_r,l_r)\} \subset B_{m,j} $ will consist
of those $(r-1)$-dimensional projective subspaces of $P_{m,j}$ that intersect $V_{k_i,l_i,j}$ nontrivially for 
all $1\leq i \leq r$. We will denote the cell corresponding to $S$ in $G_{m,j}^r$ by $V_{S,j}$. We obtain a locally 
closed decomposition
\[ G_{m,j}^r = \bigsqcup_{\substack{S \subseteq B_{m,j} \\ |S|=r}} V_{S,j}. \]
Occasionally, when it is clear from the context that $S$ is a subset
of $B_{m,j}$, we will supress the index $j$ and write just $V_S$ for
the Schubert cells of $G_{m,j}^r$.

We will call a Schubert cell \emph{maximal} if it intersects the
maximal dimensional cell of $P_{m,j}$ nontrivially. Such a cell corresponds to
subsets $S\subset B_{m,j}$ which contain $(k_{min},l)$ where $k_{min}$ is minimal among the first components of the elements of $B_{m,j}$. The intersection with $V_{k_{min},l,j}$ of a subspace corresponding to a point in a maximal Schubert cell is an affine subspace of $V_{k_{min},l,j}$. Conversely, to any affine subspace of $V_{k_{min},l,j}$, there corresponds a point in a maximal Schubert cell given by the completion of the subspace in~$P_{m,j}$.

For a maximal subset $S$, denote by $\overline{S} \subset B_{m,j}$ the set of indices which we get by deleting $(k_{min},l)$ from $S$. $\overline{S}$ is empty, if $|S|=1$. Define the codimension one projective subspace 
\[ \overline{P}_{m,j}=\bigsqcup_{ \{(k,l)\in B_{m,j}\;:\; k>k_{min}\} } V_{k,l,j} \subset P_{m,j}=P_{m,j}\setminus V_{k_{min},l,j}. \]
For each $(r-1)$-dimensional subspace $U\subset P_{m,j}$ intersecting the affine space
$V_{{k_{min}},l,j}$ nontrivially, let $\overline{U}= U\cap \overline{P}_{m,j}$. 

\begin{lemma} The map $\omega\colon V_{S,j} \rightarrow V_{\overline{S},j}$ defined by $\omega(U)= \overline{U}$
is a trivial affine fibration with fibre $\SA^{|B_{m,j}|-|S|}$. 
\label{lemma_omega}
\end{lemma}
We can think of this map as associating to an {\em affine} subspace of $V_{{k_{min}},l,j}$
its set of ``ideal points at infinity''. 
We define the mapping 
$\omega \colon V_{S,j} \rightarrow V_{\overline{S},j}$ 
as the identity for those index sets and the corresponding cells
which are not maximal. In such cases, $\overline{S}=S$ considered as a subset of $B_{m,j} \setminus \{(k_{min},l)\}$.

Consider the fibre $\omega^{-1}(\overline{U})$ over a point $\overline{U} \in V_{\overline{S}}$, which we will also denote by $V_S|_{\overline{U}}$ below. This fibre consists of those subspaces $U\subseteq P_{m,j}$ which intersect $\overline{P}_{m,j}$ in $\overline{U}$, i.e.~when considered as an affine subspace of $V_{k_{min},l,j}$, they have $\overline{U}$ as their set of ``points at infinity''. We will denote the set of such subspaces also by $V_{k_{min},l,j}/\overline{U}$. This notation means that we take the cosets in $V_{k_{min},l,j}$ of an arbitrary affine subspace $U\subset V_{k_{min},l,j}$ with $\overline{U}=\overline{P}_{m,j}\cap U$. The affine structure on $V_{k_{min},l,j}$ descends to an affine structure on $V_{k_{min},l,j}/\overline{U}$ which does not depend on the particular affine subspace $U$ whose cosets were taken.

We also need a description of the affine subspaces of the cells
$V_{k,l,j}$ for $k>k_{min}$. The relevant Schubert cells in this case
are indexed by those subsets $S$ of $B_{m,j}$ which contain $(k,l)$
but do not contain any $(k',l')$ for $k' < k$. Hence, the index set
$B_{m,j}$ is first truncated by deleting the pairs $(k',l')$ with $k'
< k$. We will denote the result as $B_{m,j}(k)$. Then, the maximal
Schubert cells for $V_{k,l,j}$ correspond to those subsets $S
\subseteq B_{m,j}(k) $ of the truncated index set which contain
$(k,l)$. For these, $\overline{S}$ is defined by removing $(k,l)$ from
$S$. There is still a morphism $\omega \colon V_{S,j} \rightarrow
V_{\overline{S},j}$ which is defined in the same way as above; its
global structure and the description of its fibres is analogous to the
previous special case. 

Note that the notation $\overline{S}$ is ambiguous at this point: any
maximal subset $S \subseteq B_{m,j}(k)$ can also be considered as a
nonmaximal subset of $B_{m,j}(k')$ for $k'<k$. If we view $S$ as a
subset of $B_{m,j}(k)$, then $\overline{S}=S\setminus\{ (k,m-k)
\}$. On the other hand, if we view it as a nonmaximal subset of
$B_{m,j}(k')$ for $k' < k$, then $\overline{S}=S$. We have decided not
to introduce extra notation; when this notation gets used below, we
will always specify the reference point $k$ explicitly.

\subsection{The Young wall associated to a homogeneous ideal}
\label{sec:Ywall_to_ideal}

In this section, we study ideals generated by degree- and weight-homogeneous 
polynomials; we will call such ideals simply {\em homogeneous}
ideals. Here is the main definition of this section. 

\begin{definition} Consider a homogeneous $G_\Delta$-invariant ideal 
$I\lhd \SC[x,y]$. Let $Y_I$ denote the following subset of
the transformed Young wall pattern of type $D_n$: 
for each block or half-block $(k,l)$ of label $j$, with $k+l=m$, include this block or
half-block in $Y_I$ if and only if $I\cap S_m[\rho_j]$ does not intersect the preimage in $S_m[\rho_j]$ of the stratum $V_{k,l,j}\subset P_{m,j}$ from the stratification of Proposition~\ref{prop:decomp of P}. $Y_I$ will be called the \emph{profile} of $I$.
\label{def_wall_to_ideal}
\end{definition}

It will be useful to introduce a little bit of extra notation, and to
reformulate this definition in the new notation. Given a homogeneous $G_\Delta$-invariant ideal $I$,
let $I_{m,j}$ be the set of $G_\Delta$-invariant subspaces of minimal dimension in $I\cap
S_m[\rho_j]$. Then as $I\cap S_m[\rho_j] \subset S_m[\rho_j]$ is a linear subspace, $I_{m,j}\subset
P_{m,j}$ is a projective linear subspace. Then the definition simply
says that a block or half-block $(k,l)$ labelled $j$ is included in $Y_I$ 
if and only if $I_{m,j}\cap V_{k,l,j} = \emptyset$, for $m=k+l$ as before. 
Since $\{V_{k,l,j}\}$ is a standard stratification of the projective
space $P_{m,j}$ into affine spaces, $\{I_{m,j}\cap V_{k,l,j}\}$ is also a
standard stratification of its projective linear subspace $I_{m,j}$ into affine
spaces, and so has the same number of strata as its affine dimension. We
conclude 
\begin{lemma} \label{lem:intdim}
For all $m,j$, the number of absent blocks or half-blocks of label $j$ on the $m$-th
diagonal equals $\mathrm{dim}( I \cap S_m[\rho_j])$. 
\end{lemma}

\begin{proposition} Given a homogeneous $G_\Delta$-invariant ideal $I\lhd \SC[x,y]$, 
the associated subset $Y_I$ of the transformed Young wall pattern of type $D_n$ has the following 
properties.
\begin{enumerate}
\item If a full or half block is missing, then all the blocks above-right from it on the diagonal are missing.\footnote{Again, for a missing half block only the half blocks of the same orientation have to be missing.}
\item If a full block is missing, then all full or half blocks to the right of it are missing, 
and {\em at least} one (full or half) block immediately above it is missing. 
\item If a half block is missing, then the full block to the right of it is missing. 
\item If both half-blocks sharing the same block position are missing, then the full block immediately above
this position is missing.
\end{enumerate}
In particular, if $I$ is of finite codimension, then $Y_I$ is a  Young
wall of type $D_n$, an element of the set ${\mathcal Z}_\Delta$.
\label{wallprop}
\end{proposition}
\begin{remark} 
As it can be checked from the definitions, the relationship between the directions in the original 
and transformed Young wall patterns is the following: (right, up, diagonally right and down) in the original 
correspond after transformation to (diagonally right and up, right, up) respectively. This way, it is easy to 
check the correspondence between the rules for Young walls from~\ref{subsec:PYW} and this 
proposition. \label{rem:corr_young_rules}
\end{remark} 

\begin{proof}[Proof of Proposition \ref{wallprop}] Fix a homogeneous invariant ideal $I\lhd \SC[x,y]$ and let $Y_I$
be the corresponding subset of the Young wall pattern. 
Property (1) of $Y_I$ follows by applying $L_1$, recalling the inductive nature of the 
stratification of $P_{m,j}$ using $L_1$. The inductice construction also implies that it suffices to check 
properties (2)-(4) for blocks missing on the bottom row. 

Let us next prove (2) in the general case, when a full block in position $(0,l)$ in representation $j\in [3, n-3]$ 
is missing from $Y_I$; by the choice of $j$, both above and to the right of this block 
there are also full blocks. Since the block at $(0,l)$ is missing, there is an invariant $2$-dimensional subspace  
$u\in I_{l,j}\cap V_{0,l,j}$ contained in $I$. Since $u$ is in the lowest stratum $V_{0,l,j}$ of $P_{l,j}$, it has a basis 
one of whose members at least is not divisible by $xy$; without loss of generality, we may assume that this polynomial is  
$x^ap$ where $a$ is a non-negative integer and $p$ a polynomial in $x,y$ not divisible by
$x,y$. Now we can write 
\[L_2u = u_+\oplus u_-\]
with $u_+\in  I_{l+1, j+1}$ and $u_-\in  I_{l+1, j-1}$. Then $u_+$ must contain a polynomial with $x^{a+1}p$ as nonzero summand, 
so it cannot be in the image of $L_1$; so we have $u_+\in V_{0,l+1,j+1}$. Similarly, $u_-$ must contain a polynomial 
with $x^ayp$ as nonzero summand, so it cannot be in the image of $L_1^2$ and so $u_-\in V_{1,l,j-1}$. 
Thus indeed both the blocks in positions $(0,l+1)$ and $(1,l)$ are missing as claimed. 

Let us now consider what changes if $j$ is chosen such that there are half-blocks around. Suppose first that the 
half-blocks happen to lie to the right of our block labelled $j$. Then we have a decomposition
\[L_2u = u_+^1\oplus u_+^2\oplus u_-,\]
with $u_+^i$ both one-dimensional. In this case, it is easy to check that the polynomial $x^{a+1}p$ cannot itself 
generate a one-dimensional eigenspace, so both $u_+^i$ will contain a polynomial with $x^{a+1}p$ as nonzero summand. 
Thus neither of these subspaces can be in the image of $L_1$, and so must lie in the large stratum. Hence both these
blocks are missing.

Suppose now that the half-blocks happen to lie above our block
labelled $j$. Then $L_2u$ is either three- or four-dimensional. In the
general case, it has dimension four and there is a decomposition
\[L_2u = u_-^1\oplus u_-^2\oplus u_+\]
with $u_-^i$ both one-dimensional. In special situations $L_2u$ is
only three dimensional, and one of the $u_-^i$'s is missing
(see Lemma~\ref{l0} below for a detailed analysis).
In any case, $x^{a+1}p$ will lie in $u_+$, forcing that subspace to be in the large stratum.
The other relevant polynomial $x^ayp$ may or may not generate a one-dimensional invariant subspace, 
depending on the values of $a, p$; so at least 
one, possibly both, of $u_-^1, u_-^2$ lies in the image of $L_1$ but not $L_1^2$, forcing them to lie in the corresponding
stratum. So at least one, but possibly both, of the corresponding half-blocks must be missing. 
We remark here that the other $u_-^i$, if present, can be divisible by
a higher power of $L_1$. This implies that in this case the ideal generated by $u$ may have nontrivial intersection with the cells $V_{k,l,j}$ even with $l > 0$. 

The proofs of (3)-(4) follow the same pattern; we omit the details. Finally if $I$ is of finite codimension, 
then it contains $S_m$ for $m$ large enough, and so $Y_I$ contains only finitely many blocks and half-blocks. 
\end{proof}


\section{Type \texorpdfstring{$D_n$}{Dn}: decomposition of the orbifold Hilbert scheme}
\label{sect:dnorbidec}

\subsection{The decomposition}

The aim of this section is to prove the following result, which gives a constructive proof of Theorem~\ref{thmorb} for type $D_n$.
\begin{theorem}
\label{thm:dnorbcells}
Let $G_\Delta$ be the subgroup of $SL(2,\SC)$ of type $D_n$. 
Then there is a locally closed decomposition
\[\mathrm{Hilb}([\SC^2/G_\Delta]) = \displaystyle\bigsqcup_{Y \in{\mathcal Z}_\Delta } \mathrm{Hilb}([\SC^2/G_\Delta])_Y\]
of the equivariant Hilbert scheme $\mathrm{Hilb}([\SC^2/G_\Delta])$ into strata indexed bijectively by the
set ${\mathcal Z}_\Delta$ of Young walls of type $D_n$, with each stratum $\mathrm{Hilb}([\SC^2/G_\Delta])_Y$ a non-empty affine space. 
\end{theorem}
\begin{proof} The affine plane $\SC^2$ carries the diagonal $T=\SC^\ast$-action, which 
commutes with the $G_\Delta$-action. The action of $T$ lifts to all the equivariant Hilbert schemes
$\mathrm{Hilb}^{\rho}([\SC^2/G_\Delta])$ which are themselves nonsingular. Thus the fixed point set
\[\mathrm{Hilb}([\SC^2/G_\Delta])^T=\sqcup_\rho\mathrm{Hilb}^{\rho}([\SC^2/G_\Delta])^T\] 
is also a union of nonsingular varieties, and it consists of points representing homogeneous invariant ideals. 
The construction of~\ref{sec:Ywall_to_ideal} associates a Young wall $Y$ to each homogeneous invariant ideal $I\lhd\SC[x,y]$. Since the construction uses a locally closed decomposition of the projective spaces $P_{m,j}$, the Young wall $Y$ also depends in a locally closed way on the ideal $I$, and thus we obtain a decomposition 
\[\mathrm{Hilb}([\SC^2/G_\Delta])^T = \displaystyle\bigsqcup_{Y \in{\mathcal Z}_\Delta } Z_Y\]
into reduced locally closed subvarieties, where $Z_Y\subset \mathrm{Hilb}([\SC^2/G_\Delta])^T$ is the locus
of homogeneous invariant ideals $I$ with associated Young wall $Y$.

Let $\mathrm{Hilb}([\SC^2/G])_Y\subset\mathrm{Hilb}([\SC^2/G]) $ denote the locus of ideals which flow to $Z_Y$ under the action of the torus $T$. 
Then by the Bia\l{}ynicki-Birula theorem~\cite{bialynicki1973some}, there is a regular map
$\mathrm{Hilb}([\SC^2/G])_Y \to Z_Y$ which is a Zariski locally
trivial fibration with affine space fibres, and a compatible
$T$-action on the fibres. By Theorem~\ref{thm:Zstrata} below, the base
is an affine space as well. Hence, by~\cite[Sect.3, Remarks]{bialynicki1973some},
$\mathrm{Hilb}([\SC^2/G])_Y$ is an algebraic vector bundle over this base, and
hence trivial (Serre--Quillen--Suslin) \cite{lang2002algebra}. Theorem~\ref{thm:dnorbcells} 
follows. 
\end{proof}

\begin{remark} 
\label{rem:Zyuniv}
\begin{enumerate}
\item As $\mathrm{Hilb}([\SC^2/G_\Delta])=\mathrm{Hilb}(\SC^2)^{G_{\Delta}} \subset \mathrm{Hilb}(\SC^2)$ is a smooth subvariety, the universal family over $\mathrm{Hilb}(\SC^2)$ restricts to a universal family over the equivariant Hilbert scheme $\mathrm{Hilb}([\SC^2/G_\Delta])$.
This restricts to a
universal family of homogeneous invariant ideals ${\mathcal U} \lhd {\mathcal O}_{\mathrm{Hilb}([\SC^2/G_\Delta])^T}\otimes\SC[x,y]$ over the $T$-fixed point set.
Restricting this universal family ${\mathcal U}$ to each of the strata
constructed above gives flat families of homogeneous invariant ideals
${\mathcal U}_Y\lhd {\mathcal O}_{Z_Y}\otimes\SC[x,y]$ over each
stratum $Z_Y$. It follows from the construction that the families ${\mathcal U}_Y$ are
universal for flat families of homogeneous invariant ideals with
associated Young wall~$Y$. We will have occasion to use the universal
property of the strata $Z_Y$ below.
\item The diagonal $T=\SC^\ast$-action on $\SC^2$ induces the usual
  monomial grading on $\SC[x,y]$, and when $I$ is homogeneous, also on
  the quotient $\SC[x,y]/I$. For a $G_{\Delta}$-invariant ideal $I$, this
  quotient is a $G_{\Delta}$-representation. In the
  homogeneous $G_{\Delta}$-invariant case, the quotient has a
  multigraded Hilbert function: a $\SZ$-grading induced by the
  $T=\SC^\ast$-action and a grading (or rather labelling) by the set
  $\{\rho_0,\dots,\rho_n\}$. By Lemma \ref{lem:intdim}, the
  multigraded Hilbert function of a homogeneous invariant ideal
  $I\lhd\SC[x,y]$ is determined by its associated Young wall $Y$.
\end{enumerate}
\end{remark}



The following is the main technical result of this section.

\begin{theorem} For each $Y\in{\mathcal Z}_\Delta$, the stratum $Z_Y$ constructed above is isomorphic to a nonempty affine space. 
\label{thm:Zstrata}\end{theorem}

\begin{remark} We note that our proof of Theorem~\ref{thm:Zstrata} below certainly provides some information about the dimension of the affine space $Z_Y$, and thus of the affine space $\mathrm{Hilb}([\SC^2/G_\Delta])_Y$. We leave the study of these quantities, which could lead to a refinement of Theorem~\ref{thmgenfunct} in the Grothendieck ring of varieties for type $D_n$, for further study.
\label{rem:Dn:motivic} 
\end{remark}

Our proof of Theorem~\ref{thm:Zstrata}, discussed below following some preparation, is a direct inductive proof. We start with a series of examples which
exhibit the range of issues our proof will have to tackle; the discussions use results to be proved further below. Throughout we take the simplest example $n=4$, which exhibits all the nontrivial features. 

\begin{example}
\label{ex:3sidepyr}
Let $Y_1$ be the triangle of size $3$.
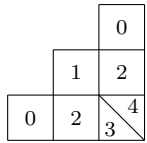
\begin{center}
\begin{tikzpicture}[scale=0.6, font=\footnotesize, fill=black!20]
 \draw (1, 0) -- (4,0);
  \foreach \x in {1,2,3}
    {
      \draw (\x, 0) -- (\x,\x);
    }
   \foreach \y in {1,2,3}
       {
            \draw (\y,\y) -- (4,\y);
       }
\draw (4, 0) -- (4,3);

\foreach \x in {1}
        {
        	\draw (\x+1.5, 0.5) node {2};
        }
  \foreach \x in {0}
          {
           \draw (\x+3.5, \x+1.5) node {2};
          }   
     
        \foreach \x in {0}
         {
                	\draw (\x+1.5, \x+0.5) node {0};
                	\draw (\x+2.5, \x+1.5) node {1};
                }
        \draw (3.5, 2.5) node {0};
 \foreach \x in {3}
           {
           \draw (\x+0.75,0.75) node {4};
                        	\draw (\x+0.25,0.25) node {3};
                        	\draw (\x,1) -- (\x+1,0);
           }
\end{tikzpicture}
\end{center}
An invariant homogeneous ideal $I$ corresponding to this Young wall
necessarily has a generator in $V_{0,3,2}\subset P_{3,2}$. The latter
is a projective line whose points can be represented by expressions
$\alpha_0 L_2 L_3+\alpha_1 L_2 L_4$. The affine line $V_{0,3,2}$ is given by $\alpha_0 + \alpha_1\neq 0$. 
It is straightforward to check that a general point in $V_{0,3,2}$, that is, when $[\alpha_0:\alpha_1] \not\in \{[1:0],[0:1]\}$, indeed generates an ideal $I$ with Young wall $Y_1$. However, when $\alpha_0$ or $\alpha_1$
become zero, then $I$ does not intersect $V_{1,3,4}$, respectively $V_{1,3,3}$, even though it should, so we have to add another generator to $I$ 
from within the corresponding cell (see Proposition~\ref{prop:1}(5) below). Both these cells are points, so there is no further choice to make and thus the space $Z_{Y_1}$ is isomorphic to an affine line. This example already illustrates the fact that even within a single stratum $Z_Y$, the minimal number of generators of an ideal $I$ with Young wall $Y$ can vary.
\end{example}

\begin{example} \label{ex:4sidepyr} Let $Y_2$ be the triangle of size $4$. In this case, we
  get $Z_{Y_2}\cong Z_{Y_1}\cong\SA^1$, the affine line of Example
  \ref{ex:3sidepyr}, due to the isomorphism $V_{0,3,2}\cong V_{0,4,0}
  \times V_{0,4,1}$, see Proposition~\ref{prop:1}(2) below, or repeat the same argument as above.
\end{example} 

\begin{example}
\label{ex:5sidepyr}
Let $Y_3$ be the triangle of size $5$.
\begin{center}
\begin{tikzpicture}[scale=0.6, font=\footnotesize, fill=black!20]
 \draw (1, 0) -- (6,0);
  \foreach \x in {1,2,3,4,5}
    {
      \draw (\x, 0) -- (\x,\x);
    }
   \foreach \y in {1,2,3,4,5}
       {
            \draw (\y,\y) -- (6,\y);
       }
\draw (6, 0) -- (6,5);

\foreach \x in {1,3}
        {
        	\draw (\x+1.5, 0.5) node {2};
        }
  \foreach \x in {0,1,2}
          {
           \draw (\x+3.5, \x+1.5) node {2};
          }   
   \foreach \x in {0}
             {
              \draw (\x+5.5, \x+1.5) node {2};
             }   
     
        \foreach \x in {0,2}
         {
                	\draw (\x+1.5, \x+0.5) node {0};
                	\draw (\x+2.5, \x+1.5) node {1};
                }
        \draw (5.5, 4.5) node {0};
 \foreach \x in {0,2}
           {
           \draw (\x+3.75,\x+0.75) node {4};
                        	\draw (\x+3.25,\x+0.25) node {3};
                        	\draw (\x+4,\x) -- (\x+3,\x+1);
           }
  \foreach \x in {1}
             {
             \draw (\x+3.75,\x+0.75) node {3};
                          	\draw (\x+3.25,\x+0.25) node {4};
                          	\draw (\x+4,\x) -- (\x+3,\x+1);
             }
             
     \foreach \x in {0}
               {
               \draw (\x+5.75,\x+0.75) node {0};
                            	\draw (\x+5.25,\x+0.25) node {1};
                            	\draw (\x+6,\x) -- (\x+5,\x+1);
               }
    
\end{tikzpicture}
\end{center}
For each fixed ideal $I$ with associated Young wall $Y_3$ must necessarily have
an generator $f\in I$ such that $[f]\in V_{0,5,2}$. By construction, $V_{0,5,2}$ is an affine
plane. This generator is, up to scalar, unique, since the block of
label $2$ in position $(0,5)$ is the only one with this label missing
from the degree $5$ antidiagonal. In other words, $I_{5,2}\cap V_{0,5,2}$ should be a point, otherwise the points at infinity of $I_{5,2}\cap V_{0,5,2}$ would intersect the other cells of degree 5 which is not allowed because of the shape of the Young wall. $I$ must also intersect both
$V_{1,5,0}$ and $V_{1,5,1}$, and again in essentially unique points,
the corresponding blocks being the only $0/1$ blocks missing from the
degree $6$ antidiagonal. This puts the following constraint on the
allowed $[f]\in V_{0,5,2}$. Use the isomorphism $L_1^{-1}$ which maps
$V_{1,5,0}\sqcup V_{1,5,1}$, a disjoint union of an affine line and a point, to $V_{0,4,1}\sqcup V_{0,4,0}$. 
Map this locus into $V_{0,5,2}$, an affine plane, to obtain $M_1^0\sqcup M_0^0\subset
V_{0,5,2}$ by taking the ideals generated by them (the curious notation $M_1^0\sqcup M_0^0$ for this locus is used here to be consistent with the general setup later, see the definitions after Lemma~\ref{lem:imgdescr}). Then by
Proposition~\ref{prop:1}(2) below, the ideal $\langle f\rangle$ intersects $P_{6,0}$ and
$P_{6,1}$ at most in the correct cells $V_{1,5,0}$ and $V_{1,5,1}$ if
and only if $[f]\in V_{0,5,2}$ lies on the linear join of the affine line $M_0^0$ and the
point $M_1^0$ inside the affine plane $V_{0,5,2}$ but not on $M_1^0\sqcup M_0^0$. This join is the plane $V_{0,5,2}$
minus a punctured affine line $M_1\setminus M_1^0$, where 
$M_1$ is the line parallel to $M_0^0$, going through $M_1^0$. On this join,
but away from the locus $M_0^0\sqcup M_1^0$ itself, we can set
$I=\langle f\rangle$ to indeed get an ideal with Young wall $Y_3$. On
the locus $M_0^0\sqcup M_1^0$ however, the ideal $\langle f\rangle$ itself will not
actually meet both cells, so we have to add an arbitrary generator of
the missed cell to $f$ to obtain an ideal of the correct Young
wall. Over the affine line $M_0^0$, there is no choice, since
$V_{1,5,0}$ is a point. But over the point $M_1^0$, we still have
$V_{1,5,0}$, in other words an affine line, worth of choices. This
sofar tells us that $Z_{Y_3}$ is the disjoint union of
$V_{0,5,2}\setminus M_1 \cong \SA^2\setminus\SA^1$ and
$V_{1,5,0}\cong\SA^1$. 

\begin{center}
\begin{tikzpicture}[scale=0.8,font=\footnotesize]
\draw (-3,0,0) --(3,0,0); 
\draw (0,3,0) --(0,0.8,0);
\draw[->] (0,0.6,0) -- (0,0.2,0);
\draw (0,0,-3)--(0,0,3);
\draw (-1,0,-3)--(3,0,1);
\draw[dashed] (3,0,3)--(-3,0,-3);
\draw (-3,0,-3)--(-3,0,3)--(3,0,3)--(3,0,-3)--(-3,0,-3);
\draw (-0.5,0,0.6) node {$M_1^0$};
\draw (1.4,0,-1.4) node {$M_0^0$};
\draw (-1.9,0,-1.1) node {$M_1$};
\end{tikzpicture}
\end{center}

To fully work out the geometry of $Z_{Y_3}$, note that $P_{5,2}\cong\SP^2$ can be parametrized by expersions $\alpha_0 L_2 L_3^2 + \alpha_1 L_2 L_3 L_4 + \alpha_2 L_2 L_4^2$. The locus $V_{0,5,2}\subset P_{5,2}$ is given by $\alpha_0 + \alpha_1 + \alpha_2\neq 0$. The image of $V_{0,4,0}$ in these coordinates is 
$M_0^0=\{ (\alpha_0,0,\alpha_1): \alpha_0+\alpha_1 \neq 0\}$, while the image of $V_{0,4,1}$ is $M_1^0=\{(0, 1, 0)\}$. The linear combinations of the points in $M_0^0$ and $M_1^0$ cover the whole affine plane $V_{0,5,2}$, except a punctured line. 
For a general linear combination $a \cdot (\alpha_0,0,\alpha_1)+ b \cdot (0, 1, 0)$, the ideal generated by the corresponding $f$ intersects $V_{1,5,0}\times V_{1,5,1}$ in $(L_1 L_3 L_4,  \alpha_0 L_1 L_3^2+\alpha_1 L_1 L_4^2)$. For $(a,b)=(1,0)$ we have to have an extra generator in $V_{1,5,0}$, while for $(a,b)=(0,1)$ we need an extra generator in $V_{1,5,1}$.

Consider a family of ideals which approaches the point $M_1^0$ from the direction $(\alpha_0,0,\alpha_1)$. Then it can be
shown by explicit calculation that the limit ideal contains the subspace 
generated by $\alpha_0 L_1 L_3^2+\alpha_1 L_1 L_4^2 \in V_{1,5,1}$. This shows that $Z_{Y_3}$ can be obtained by blowing up the affine plane $V_{0,5,2}$ in its point $M_1^0$, and removing the proper transform of the punctured line $M_1\setminus M_1^0$ from this blowup. Thus $Z_{Y_3} \cong \SA^2$. From the blowup construction, we also obtain a canonical morphism $Z_{Y_3} \to \SA^1$, the restriction of the morphism ${\rm Bl}_0 \SA^2 \to {\mathbb P}^1$ to the exceptional curve of the blowup. 
\end{example}

\begin{example}
\label{ex:5sidedalt}
 Let $Y_3'$ be the Young wall
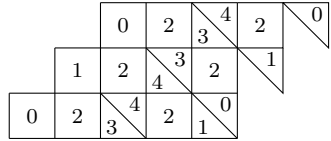
\begin{center}
\begin{tikzpicture}[scale=0.6, font=\footnotesize, fill=black!20]
 \draw (1, 0) -- (6,0);
  \foreach \x in {1,2,3}
    {
      \draw (\x, 0) -- (\x,\x);
    }
  
   \foreach \y in {1,2,3}
       {
            \draw (\y,\y) -- (6,\y);
       }
 \draw (4, 0) -- (4,3);
\draw (5, 0) -- (5,3);
 \draw (6, 0) -- (6,3);
 \draw (6, 3) -- (7,3) -- (7,2) --(6,2);

\foreach \x in {0,1,2}
        {
        	\draw (\x+2.5, \x+0.5) node {2};
        }
  \foreach \x in {0,1,2}
       {
           \draw (\x+4.5, \x+0.5) node {2};
          }   
     
     \foreach \x in {0}
         {
                	\draw (\x+1.5, \x+0.5) node {0};
                	\draw (\x+2.5, \x+1.5) node {1};
                }
       \draw (3.5, 2.5) node {0};
 \foreach \x in {0,2}
           {
           \draw (\x+3.75,\x+0.75) node {4};
                        	\draw (\x+3.25,\x+0.25) node {3};
                        	\draw (\x+3,\x+1) -- (\x+4,\x+0);
           }
  \foreach \x in {1}
             {
             \draw (\x+3.75,\x+0.75) node {3};
                                     	\draw (\x+3.25,\x+0.25) node {4};
                                     	\draw (\x+3,\x+1) -- (\x+4,\x+0);
             }
      \foreach \x in {0}
                    {
                    \draw (\x+5.75,\x+0.75) node {0};
                                 	\draw (\x+5.25,\x+0.25) node {1};
                                 	\draw (\x+6,\x) -- (\x+5,\x+1);
                    }
     \draw (6,2)--(7,1)--(7,2);
     \draw (6.75,1.75) node {1};
     \draw (7,3)--(8,2)--(8,3)--(7,3);
     \draw (7.75,2.75) node {0};

\end{tikzpicture}
\end{center}
There is necessarily still a unique generator in $V_{0,5,2}$. The difference compared to $Y_3$ is that there is now no intersection with $V_{1,5,1}$ but an intersection with $V_{3,3,1}$. The cells of the 0/1-blocks missing from the degree six diagonal are $V_{1,5,0}$ and $V_{3,3,1}$, both of which are zero dimensional. As before, we pull back these using $L_1^{-1}$, take the linear combinations of their images in $P_{5,2}$, and intersect this line with $V_{0,5,2}$. This is exactly the line $M_1$ from Example \ref{ex:5sidepyr}. This has one special point, $M_1^0$. If the new generator is (in the subspace represented by) this point, then it will not generate an ideal with shape $Y$, except if we keep the unique element of $V_{3,3,0}$ as a generator. In any case, $Z_{Y_3'}\cong M_1\cong\SA^1$.
\end{example}

\begin{example} \label{ex:GHilb} As a final example, it is well known that the minimal resolution of the singularity $\SC^2/G_\Delta$ is given by the component~$\mathrm{Hilb^{\rho_{reg}}([\SC^2/G])}$ of $\mathrm{Hilb([\SC^2/G])}$ corresponding to the regular representation~\cite{kapranov2000kleinian}. The $\SC^*$-fixed set on the minimal resolution consists of the $\SP^1$ within the exceptional locus corresponding to the central node in the Dynkin diagram, as well as three isolated points on the other three $\SP^1$'s. 
The following five Young walls contribute for the regular representation.
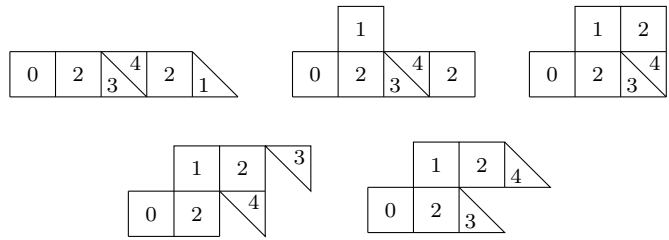
\begin{center}
\begin{tikzpicture}[scale=0.6, font=\footnotesize, fill=black!20]
 \draw (1, 0) -- (6,0);
 \draw (1, 1) -- (5,1);
  \foreach \x in {1,2,3,4,5}
    {
      \draw (\x, 0) -- (\x,1);
    }
  \draw (1.5, 0.5) node {0};    
  \draw (2.5, 0.5) node {2};
  
  \draw (3.25, 0.25) node {3};
  \draw (3,1)--(4,0);
  \draw (3.75, 0.75) node {4};

\draw (4.5, 0.5) node {2};

\draw (5,1)--(6,0);
  \draw (5.25, 0.25) node {1};
\end{tikzpicture}\qquad
\begin{tikzpicture}[scale=0.6, font=\footnotesize, fill=black!20]
 \draw (1, 0) -- (5,0);
 \draw (1, 1) -- (5,1);
  \foreach \x in {1,2,3,4,5}
    {
      \draw (\x, 0) -- (\x,1);
    }
  \draw (1.5, 0.5) node {0};    
  \draw (2.5, 0.5) node {2};
  
  \draw (3.25, 0.25) node {3};
  \draw (3,1)--(4,0);
  \draw (3.75, 0.75) node {4};

\draw (4.5, 0.5) node {2};

\draw (2,1)--(2,2)--(3,2)--(3,1);
  \draw (2.5, 1.5) node {1};
\end{tikzpicture}\qquad
\begin{tikzpicture}[scale=0.6, font=\footnotesize, fill=black!20]
 \draw (1, 0) -- (4,0);
 \draw (1, 1) -- (4,1);
  \foreach \x in {1,2,3,4}
    {
      \draw (\x, 0) -- (\x,1);
    }
  \draw (1.5, 0.5) node {0};    
  \draw (2.5, 0.5) node {2};
  
  \draw (3.25, 0.25) node {3};
  \draw (3,1)--(4,0);
  \draw (3.75, 0.75) node {4};

\draw (2,1)--(2,2)--(3,2)--(3,1);
  \draw (2.5, 1.5) node {1};
\draw (3,2)--(4,2)--(4,1);
\draw (3.5, 1.5) node {2};
\end{tikzpicture}

\vspace{0.5cm}

\begin{tikzpicture}[scale=0.6, font=\footnotesize, fill=black!20]
 \draw (1, 0) -- (3,0);
 \draw (1, 1) -- (4,1);
  \foreach \x in {1,2,3}
    {
      \draw (\x, 0) -- (\x,1);
    }
  \draw (1.5, 0.5) node {0};    
  \draw (2.5, 0.5) node {2};

  \draw (3,1)--(4,0)--(4,1);
  \draw (3.75, 0.75) node {4};

\draw (2,1)--(2,2)--(3,2)--(3,1);
  \draw (2.5, 1.5) node {1};
\draw (3,2)--(4,2)--(4,1);
\draw (3.5, 1.5) node {2};

 \draw (4,2)--(5,1)--(5,2)--(4,2);
\draw (4.75, 1.75) node {3};
\end{tikzpicture}\qquad
\begin{tikzpicture}[scale=0.6, font=\footnotesize, fill=black!20]
 \draw (1, 0) -- (3,0);
 \draw (1, 1) -- (4,1);
  \foreach \x in {1,2,3}
    {
      \draw (\x, 0) -- (\x,1);
    }
  \draw (1.5, 0.5) node {0};    
  \draw (2.5, 0.5) node {2};

  \draw (3,1)--(4,0)--(3,0);
  \draw (3.25, 0.25) node {3};

\draw (2,1)--(2,2)--(3,2)--(3,1);
  \draw (2.5, 1.5) node {1};
\draw (3,2)--(4,2)--(4,1);
\draw (3.5, 1.5) node {2};

 \draw (4,2)--(5,1)--(4,1);
\draw (4.25, 1.25) node {4};
\end{tikzpicture}
\end{center}
A quick computation shows that $Z_Y$ is a point in each case, except
for the last Young wall in the first row, when it is an affine line
(as in Example~\ref{ex:3sidepyr} except that in this case there is also a generator in $V_{2,2,0}$). This affine line contains the point corresponding to the Young wall next to it in its closure, giving the central $\SP^1$. 
\end{example}

\subsection{Incidence varieties}
\label{subsec:incvargr}

The purpose of this section is to introduce some incidence varieties inside products of the Schubert cells defined in~\ref{subsec:schubertcells}. We state some propositions regarding these incidence varieties and morphisms between them, whose proofs we defer to~\ref{subsec:prop:proofs} below. We discuss four different cases.

\smallskip

\noindent{\bf Case \ref{subsec:incvargr}.1\ } Assume that $m \equiv 0\; \mathrm{mod}\; n-2$ is a nonnegative integer, such that at position $(0,m)$ there is a divided block with labels $(c_1, c_2)$. Let $c$ be the label of the block at position $(1,m)$. Let $S_{c} \subseteq B_{m+1,c}$ be a nonempty maximal subset. Let $S_{1} \subseteq B_{m,c_1}$ and $S_{2} \subseteq B_{m,c_2}$ be two maximal subsets which are \emph{allowed} by $S_{c}$. This means, by definition, that each block above every composite block whose two halves are both contained in $S_1\cup S_2$, and each block to the right of every composite block at least one half-block of which in $S_1\cup S_2$, is in $S_c$.

Consider the incidence varieties
\[ X_{S_1,S_2}^{S_c}=\{ (U_1,U_2, U_c)\;:\; (U_1,U_2)\cap P_{m+1,c} \subseteq U_c \} \subseteq V_{S_1} \times V_{S_2} \times V_{S_{c},c},\]
and
\[ Y_{\overline{S}_1,\overline{S}_2}^{S_c}=\{ (\overline{U}_1,\overline{U}_2, U_c)\;:\; (\overline{U}_1, \overline{U}_2)\cap P_{m+1,c} \subseteq U_c \} \subseteq V_{\overline{S}_1} \times V_{\overline{S}_2} \times V_{S_{c},c}.\]
Using the maps $\omega$ from Lemma~\ref{lemma_omega}, these varieties fit into the diagram
\[
\begin{array}{ccccc}
 X_{S_1,S_2}^{S_c} & \subseteq & V_{S_1} \times V_{S_2} \times V_{S_c,c} & \ra{\mathrm{Id} \times \mathrm{Id} \times \omega}  &  V_{S_1} \times V_{S_2} \times V_{\overline{S}_c,c}  \\ 
& & \da{\omega \times \omega \times \mathrm{Id}} & &\da{\omega \times \omega \times \mathrm{Id} } \\
Y_{\overline{S}_1, \overline{S}_2}^{S_c} & \subseteq &  V_{\overline{S}_1} \times V_{\overline{S}_2} \times V_{S_c,c}& \ra{\mathrm{Id} \times \mathrm{Id} \times \omega} & V_{\overline{S}_1} \times V_{\overline{S}_2} \times V_{\overline{S}_c,c} \;.
\end{array}
\]

\begin{proposition}
\label{prop:incvar1}
\begin{enumerate}
\item The image of $X_{S_1,S_2}^{S_c}$ under the vertical morphism $V_{S_1} \times V_{S_2} \times V_{S_{c},c} \to V_{\overline{S}_1} \times V_{\overline{S}_2} \times V_{S_{c},c}$ is precisely $Y_{\overline{S}_1, \overline{S}_2}^{S_c}$.
\item The induced morphism $X_{S_1,S_2}^{S_c} \rightarrow Y_{\overline{S}_1, \overline{S}_2}^{S_c}$ is a trivial fibration over its image with affine fibers of dimension $|S_{c}|-|S_1|-|S_2|+1$. 
\item The horizontal morphism $X_{S_1,S_2}^{S_c} \rightarrow V_{S_1} \times V_{S_2}  \times V_{\overline{S}_{c},c}$ is injective.
\end{enumerate}
\end{proposition}

\smallskip

\noindent{\bf Case \ref{subsec:incvargr}.2\ }
Let $m \equiv 0\; \mathrm{mod}\; n-2$, but this time consider only one half block of label $c_0=\kappa(c)$ at the position $(0,m)$. Let $S_{c} \subseteq B_{m+2,c}$ be a nonempty maximal subset, and $S \subseteq B_{m,\kappa(c)}$ be a maximal subset which is allowed by $S_{c}$. This means that for each block in $S$ there is a block in $S_c$ at the top right corner. In analogy with the previous case, let
\[ X_{S}^{S_c}=\{ (U, U_c)\;:\; (U)\cap P_{m+2,c} \subseteq U_c \} \subseteq V_{S} \times V_{S_{c},c}\;,\]
and
\[ Y_{\overline{S}}^{S_c}=\{ (\overline{U}, U_c)\;:\; (\overline{U})\cap P_{m+2,c} \subseteq U_c \} \subseteq V_{\overline{S}} \times V_{S_{c},c}\;,\]
which fit into the diagram
\[
\begin{array}{ccccc}
 X_{S}^{S_c} & \subseteq & V_{S} \times V_{S_c,c} & \ra{\mathrm{Id} \times \omega}  &  V_{S} \times V_{\overline{S}_c,c}  \\ 
& & \da{\omega \times \mathrm{Id}} & &\da{\omega \times \mathrm{Id} } \\
Y_{\overline{S}}^{S_c} & \subseteq &  V_{\overline{S}} \times V_{S_c,c}& \ra{\mathrm{Id} \times \omega} & V_{\overline{S}} \times V_{\overline{S}_c,c} \;.
\end{array}
\]

\begin{proposition}
\label{prop:incvar2}
\begin{enumerate}
\item The image of $X_{S}^{S_c}$ under the vertical morphism $ V_{S} \times V_{S_{c},c} \to V_{\overline{S}} \times V_{S_{c},c}$ is exactly $Y_{\overline{S}}^{S_c}$.
\item The induced morphism $X_{S}^{S_c} \rightarrow Y_{\overline{S}}^{S_c}$ is a trivial fibration over its image with affine fibers of dimension $|S_{c}|-|S|$. 
\item The horizontal morphism $X_{S}^{S_c} \rightarrow V_{S} \times V_{\overline{S}_{c},c}$ is injective.
\end{enumerate}
\end{proposition}

\smallskip

\noindent{\bf Case \ref{subsec:incvargr}.3\ } Let $m \equiv 1\; \textrm{mod}\; n-2$, and $c_1$ and $c_2$ the labels of the divided block immediately above the block at position $(0,m)$. Let $S_{1} \subseteq B_{m+1,c_1}$, $S_{2} \subseteq B_{m+1,c_2}$ be nonempty subsets at least one of which is maximal. Let moreover, $S \subseteq B_{m,j}$ be a maximal subset which is allowed by $S_{1}$ and $S_{2}$. In this case, this means the following: for each block $b$ in $S$, there is a divided block of with labels $(c_1, c_2)$ in the pattern either immediately above or to the right of $b$. In the first case, we require that at least one of these half-blocks is in $S_{1} \cup S_{2}$. In the second case, we require that both are contained in $S_{1} \cup S_{2}$.

Given this data, we define
\[ X_{S}^{S_1,S_2}=\{ (U, U_1,U_2)\;:\; (U)\cap P_{m+1,c_1} \subseteq U_1, (U) \cap P_{m+1,c_2} \subseteq U_2 \} \subseteq V_{S} \times V_{S_{1},c_1} \times V_{S_{2},c_2}\;,\]
and
\[ Y_{\overline{S}}^{S_1,S_2}=\{ (\overline{U}, U_1,U_2)\;:\; (\overline{U})\cap P_{m+1,c_1} \subseteq U_1, (\overline{U})\cap P_{m+1,c_2} \subseteq U_2 \} \subseteq V_{\overline{S}} \times V_{S_{1},c_1} \times V_{S_{2},c_2}\;. \]
We now have the following diagram:
\[
\begin{array}{ccccc}
  X_{S}^{S_1,S_2} & \subseteq & V_{S} \times V_{S_{1},c_1} \times V_{S_{2},c_2} & \ra{\mathrm{Id} \times \omega\times\omega}  & V_{S} \times V_{\overline{S}_1,c_1} \times V_{\overline{S}_2,c_2}  \\ 
 & & \da{\omega \times \mathrm{Id} \times \mathrm{Id}} & & \da{\omega \times \mathrm{Id} \times \mathrm{Id}} \\
 Y_{\overline{S}}^{S_1,S_2} &  \subseteq & V_{\overline{S}} \times V_{S_{1},c_1} \times V_{S_{2},c_2} & \ra{\mathrm{Id} \times \omega \times \omega} & V_{\overline{S}} \times V_{\overline{S}_1,c_1} \times V_{\overline{S}_2,c_2}\;.
\end{array}
\]

\begin{proposition} 
\label{prop:incvar3}
\begin{enumerate}
\item The image of $X_{S}^{S_1,S_2}$ under the vertical morphism $ V_{S} \times V_{S_{1},c_1} \times V_{S_{2},c_2} \to V_{\overline{S}} \times V_{S_{1},c_1} \times V_{S_{2},c_2}$ is exactly $Y_{\overline{S}}^{S_1,S_2}$. 
\item The induced morphism $ X_{S}^{S_1,S_2} \to Y_{\overline{S}}^{S_1,S_2}$ is a trivial fibration with fibers isomorphic to affine spaces of dimension $|S_{1}|+|S_{2}|-|S|$. 
\end{enumerate}
\end{proposition}

We remark that the analogue of (3) of Propositions~\ref{prop:incvar1} and~\ref{prop:incvar2} is not true in this case. What happens to $X_{S}^{S_1,S_2}$ when we project $V_{S_{1},c_1} \times V_{S_{2},c_2}$ to $V_{\overline{S}_{1},c_1} \times V_{\overline{S}_{2},c_2}$ will be the subject of~\ref{Dnspecialsect} below.

\smallskip

\noindent{\bf Case \ref{subsec:incvargr}.4\ } Finally, assume $m \not\equiv 0, 1\; \mathrm{mod}\; n-2$ with a full block in position $(0,m)$. Let $c$ be the label of the full block immediately above this position, and $S_{c} \subseteq B_{m+1,c}$ a nonempty maximal subset. Let moreover $S \subseteq B_{m,j}$ be a maximal subset which is allowed by $S_{c}$; in this case, this means that above every block of $S$ there is a block in $S_c$. Consider the incidence 
varieties
\[ X_{S}^{S_c}=\{ (U, U_c)\;:\; (U)\cap P_{m+1,c} \subseteq U_c \} \subseteq V_{S} \times V_{S_{c},c}\]
and
\[ Y_{\overline{S}}^{S_c}=\{ (\overline{U}, U_c)\;:\; (\overline{U})\cap P_{m+1,c} \subseteq U_c \} \subseteq V_{\overline{S}} \times V_{S_{c},c}.\]
There is the following diagram:
\[
\begin{array}{ccccc}
 X_{S}^{S_c} & \subseteq & V_{S} \times V_{S_c,c} & \ra{\mathrm{Id} \times \omega}  &  V_{S} \times V_{\overline{S}_c,c}  \\ 
& & \da{\omega \times \mathrm{Id}} & &\da{\omega \times \mathrm{Id} } \\
Y_{\overline{S}}^{S_c} & \subseteq &  V_{\overline{S}} \times V_{S_c,c}& \ra{\mathrm{Id} \times \omega} & V_{\overline{S}} \times V_{\overline{S}_c,c} \;.
\end{array}
\]

\begin{proposition}
\label{prop:incvar4}
\begin{enumerate}
\item The image of $X_{S}^{S_c}$ under the vertical morphism $ V_{S} \times V_{S_{c},c} \to V_{\overline{S}} \times V_{S_{c},c}$ is exactly $Y_{\overline{S}}^{S_c}$.
\item The induced morphism $X_{S}^{S_c} \rightarrow Y_{\overline{S}}^{S_c}$ is a trivial fibration over its image with affine fibers of dimension $|S_{c}|-|S|$. 
\item The horizontal morphism $X_{S}^{S_c} \rightarrow V_{S} \times V_{\overline{S}_{c},c}$ is injective.
\end{enumerate}
\end{proposition}

\subsection{Proof of Theorem \ref{thm:Zstrata}}
\label{sec:proof:orbicells}

In this section we prove Theorem \ref{thm:Zstrata}, thus completing the proof of Theorem~\ref{thm:dnorbcells}, using the constructions and results stated in~\ref{subsec:incvargr}. Given a Young wall $Y\in{\mathcal Z}_\Delta$, we need to show that the corresponding stratum $Z_Y$ is an affine space. The following is the key combinatorial definition which underlies much of the rest of the paper. 

\begin{definition}
\label{def:globsalient}
Consider the Young wall $Y$, as usual in the transformed pattern. The \emph{salient blocks} of $Y$ are those blocks in the complement of $Y$, whose absence from $Y$ does not follow from the shape of the rows below it, and which are at the leftmost positions in their rows with this property. In particular, these are
\begin{itemize}
\item missing half blocks under which there is a block in $Y$;
\item missing undivided full blocks under which there is a block in $Y$;
\item missing divided full blocks immediately to the right of the boundary of $Y$;
\item the leftmost missing block(s) in the bottom row.
\end{itemize}
\end{definition}
Given an ideal $I\in Z_Y$, it is easy to see that $I$ is necessarily generated by elements lying in cells corresponding to the salient blocks of $Y$. In most cases it is also true that all cells corresponding to salient blocks must contain a generator, but not always; we have already seen~Example~\ref{ex:3sidepyr}, where the divided missing blocks at position $(1,3)$ are salient blocks of $Y_3$, since they lie immediately to the right of the boundary of $Y_3$, but the corresponding cells do not necessarily contain generators of an ideal $I\in Z_{Y_3}$.

We start our analysis by defining maps from the strata $Z_Y$ to the Grassmannian cells defined in~\ref{subsec:schubertcells}. Consider an arbitrary block or half-block of label $j$ at position $(k,l)$ in the Young wall pattern. Let $S(k,l,j) \subseteq B_{k+l,j}$ be the set of blocks of label $j$ from the $(k+l)$-th antidiagonal which are not in $Y$ and are above and including the position $(k,l)$. $S(k,l,j)$ is called the {\em index set} of $(k,l)$ in $Y$. If the block of label $j$ at position $(k,l)$ is not contained in $Y$, then the index set $S=S(k,l,j)$ contains $(k,l)$ as well. By the correspondence discussed at the end of \ref{subsec:schubertcells} between the maximal Schubert cells of the relevant Grassmannian for $V_{k,l,j}$ and subsets of  $B_{k+l,j}(k)$ which do contain $(k,l)$, there is a maximal Schubert cell $V_{S,j}$ corresponding to $S$.  The affine cell $V_{S,j} \subset G^r_{k+l,j}$ for $r=|S|$ parameterizes certain affine subspaces of $V_{k,l,j}$, or, equivalently, projective subspaces of $P_{k+l,j}$, the projectivization of the space of degree $(k+l)$ homogeneous polynomials which transform in the representation $\rho_j$ with respect to $G_{\Delta}$. The correspondence is by projective closure of the corresponding affine subspace of $V_{k,l,j}$. 

\begin{lemma} 
\label{lemma:grass-cells-morphism} 
For any block or half-block at position $(k,l)$ which is not contained in $Y$ and has index set $S=S(k,l,j)$,
there is a morphism \[ \begin{array}{rcl} Z_Y & \rightarrow & V_{S,j} \\  I & \mapsto & I \cap V_{k,l,j}.\end{array}\]
\end{lemma}
\begin{proof}
 Assume first for simplicity that $(k,l)=(0,m)$. Let ${\mathcal U}_Y\lhd \left({\mathcal O}_{Z_Y}\otimes \SC[x,y]\right)$ be the universal family of homogeneous ideals over $Z_Y$ introduced in Remark~\ref{rem:Zyuniv}. Consider the subfamily ${\mathcal V} = {\mathcal U}_Y \cap \left({\mathcal O}_{Z_Y} \otimes S_m[\rho_j] \right)$. This is a family of subspaces of dimension $r\dim{\rho_j}$ in $S_m[\rho_j]$ parameterized by $Z_Y$, where $r=|S|$. It is known, that the morphisms from $Z_Y$ to the equivariant Grassmannian  $G^r_{m,j}$ of $S_m[\rho_j]$ are in one-to-one correspondence with $G_\Delta$ invariant subbundles of ${\mathcal O}_{Z_Y} \otimes S_m[\rho_j]$ of  rank $r\dim{\rho_j}$ \cite[p. 88, Theorem 2.4]{eisenbud20163264}. Hence, there is a classifying morphism $Z_Y \to G^r_{m,j}$ inducing ${\mathcal V}$.
  
  By the definition of $Z_Y$, the multigraded Hilbert polynomial of ${\mathcal U}_Y$ is constant. The Hilbert polynomial encodes the dimensions of the intersections with the cells of $P_{m,j}$. Therefore, over closed points of $Z_Y$ the elements of the family ${\mathcal V}$, when considered as projective subspaces of $P_{m,j}$, intersect exactly the cells $V_{k_i,l_i,j}$ for each element $(k_i,l_i) \in S$. As in~\ref{subsec:schubertcells}, these projective subspaces are represented by points in the Schubert cell $V_{S,j} \subset G^r_{m,j}$. Hence, the image of the classifying morphism of ${\mathcal V}$ is inside the cell $V_{S,j}$. By the construction, 
 the classifying morphism is just the same as taking the intersection of $I_{m,j}$ with $V_{0,m,j}$. This is denoted as $I \cap V_{0,m,j}$ above.
  
  The general case follows similarly.
\end{proof}

We will prove Theorem \ref{thm:dnorbcells} by induction on the number of nonempty rows of $Y$. Consider an arbitrary Young wall~$Y$ consisting of~$l>0$ rows. Let $\overline{Y}$ denote the Young wall obtained from $Y$ by deleting its bottom row; we will call this the {\em truncation} of $Y$. Of course the labels of the half blocks are exchanged by $\kappa$, but we will supress this in the notations. The following result will be key to our induction. 

\begin{lemma} There exists a morphism of schemes
\[\begin{array}{rcccl} T & : & Z_Y & \rightarrow & Z_{\overline{Y}} \\ &&  I & \mapsto & L_1^{-1} \left(I \cap L_1 \SC[x,y] \right). \end{array}\]
\label{lemma:inductive_morphism}
\end{lemma}
\begin{proof} Let ${\mathcal U}_Y\lhd \left({\mathcal O}_{Z_Y}\otimes \SC[x,y]\right)$ be the universal family of homogeneous ideals over $Z_Y$. Consider 
${\mathcal I} = L_1^{-1} \left({\mathcal U}_Y \cap L_1 ({\mathcal O}_{Z_Y}\otimes\SC[x,y]) \right)$. It is straightforward to check locally that this is still a sheaf of ideals in ${\mathcal O}_{Z_Y}\otimes \SC[x,y]$. On closed points of $Z_Y$, it is also clear that the restriction has Young wall $\overline Y$.  As mentioned above, the multigraded Hilbert polynomial of ${\mathcal U}_Y$ is constant. As $Z_Y$ is reduced, it then follows from \cite[Ch III. Thm. 9.9]{hartshorne1977algebraic} that ${\mathcal I}$ is a flat family of homogeneous ideals with Young wall $\overline Y$ over $Z_Y$. By Remark \ref{rem:Zyuniv} there is a classifying morphism $Z_Y\to Z_{\overline Y}$  for this family, which is exactly the morphism $T$. 
\end{proof}

Next, we continue with an investigation of the combinatorics of the bottom two rows of our Young wall~$Y$. The boundary of~$Y$ in the transformed pattern is divided by the blocks into horizontal, vertical and diagonal straight line segments. The first two lines in the bottom can be connected in the following six possible ways:
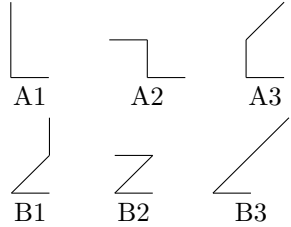
\begin{center}
\begin{tikzpicture}[scale=0.5]
\draw (0.5,0) -- (0.5,2);
\draw (0.5,0) -- (1.5,0);
\draw (1,-0.5) node {A1};
\end{tikzpicture}\qquad
\begin{tikzpicture}[scale=0.5]
\draw (1,0) -- (1,1);
\draw (1,1) -- (0,1);
\draw (1,0) -- (2,0);
\draw (1,-0.5) node {A2};
\end{tikzpicture}\qquad
\begin{tikzpicture}[scale=0.5]
\draw (0,0) -- (0,1);
\draw (0,1) -- (1,2);
\draw (0,0) -- (1,0);
\draw (0.5,-0.5) node {A3};
\end{tikzpicture}\qquad

\begin{tikzpicture}[scale=0.5]
\draw (0,0) -- (1,1);
\draw (1,1) -- (1,2);
\draw (0,0) -- (1,0);
\draw (0.5,-0.5) node {B1};
\end{tikzpicture}\qquad
\begin{tikzpicture}[scale=0.5]
\draw (0,0) -- (1,1);
\draw (1,1) -- (0,1);
\draw (0,0) -- (1,0);
\draw (0.5,-0.5) node {B2};
\end{tikzpicture}\qquad
\begin{tikzpicture}[scale=0.5]
\draw (0,0) -- (2,2);
\draw (0,0) -- (1,0);
\draw (1,-0.5) node {B3};
\end{tikzpicture}
\end{center}
Here a diagonal straight line borders a half block of the Young wall, which can be either a lower or an upper triangle. In the A cases the salient block in the bottom row is a full block, while in the B cases it is a half block. 

Let the salient block of $Y$ in its bottom row be at position $(0,m)$. It can be either a divided or undivided full block, or a half block. In the first case, we have a type A corner at the bottom of $Y$, while in the second case there is a type B corner. As in~\ref{subsec:incvargr}, we need to consider four cases. In each case, we are going to define morphisms $Z_Y \to {\mathcal X}_{Y}$ and $Z_{\overline{Y}} \to {\mathcal Y}_{ \overline{Y}}$ to incidence varieties defined in~\ref{subsec:incvargr}. 

\smallskip

\noindent{\bf Case \ref{sec:proof:orbicells}.1\ } Assume $m \equiv 0\; \mathrm{mod}\; n-2$ and we have vertex types A1 or A2 (A3 is not possible in this case). We are in the context of Case~\ref{subsec:incvargr}.1: the divided block at position $(0,m)$ has labels $(c_1, c_2)$, and  index sets $S_1, S_2$; the block at position $(1,m)$ has label $c$, and index set $S_c$. By the Young wall rules for $Y$, $S_1, S_2$ is allowed by $S_c$. Lemma~\ref{lemma:grass-cells-morphism} implies that there is a morphism 
\[ \begin{array}{ccc} Z_Y & \to & V_{S_1} \times V_{S_2} \times V_{S_c} \\
I & \mapsto & (I \cap V_{0,m,c_1}, I \cap V_{0,m,c_2}, I \cap V_{1,m,c}).\end{array}\]
By construction, the image of this morphism is contained in the incidence variety $X_{S_1,S_2}^{S_c}\subset V_{S_1} \times V_{S_2} \times V_{S_c}$ from Case~\ref{subsec:incvargr}.1. Denote ${\mathcal X}_{Y}=X_{S_1,S_2}^{S_c}\subseteq V_{S_1} \times V_{S_2} \times V_{S_c}$; we thus obtain an induced morphism $Z_Y \to {\mathcal X}_{Y}$.

There is also a morphism 
\[ \begin{array}{ccc} Z_{\overline{Y}} & \to & V_{\overline{S}_1} \times V_{\overline{S}_2 } \times V_{S_c}\\ I & \mapsto & (L_1 I \cap V_{k_{1},l_1,c_1},L_1 I \cap V_{k_{2},l_2,c_2},L_1 I \cap V_{1,m,c}),\end{array}\]
where $(k_i,l_i)$ is the lowest block in $\overline{S}_i$ for $i=1,2$. We obtain an induced morphism $Z_{\overline{Y}} \to {\mathcal Y}_{\overline{Y}}$, where ${\mathcal Y}_{\overline{Y}}=Y_{\overline{S}_1,\overline{S}_2}^{S_c}$. 

\smallskip

\noindent{\bf Case \ref{sec:proof:orbicells}.2\ } Assume $m \equiv 0\; \mathrm{mod}\; n-2$ with vertex types B1 to B3. This is Case~\ref{subsec:incvargr}.2: the block at position $(0,m)$ has label $\kappa(c)$, and index set $S=S_{\kappa(c)}$; the block at position $(1,m+1)$ has label $c$, and index set  $S_{c}$. We consider the morphisms
\[ \begin{array}{ccc}Z_Y & \to & V_{S} \times V_{S_c} \\ I & \mapsto & (I \cap V_{0,m,\kappa(c)},I \cap V_{1,m+1,c})\;, \end{array}\]
and
\[ \begin{array}{ccc} Z_{\overline{Y}} & \to & V_{\overline{S}} \times V_{S_c} \\ I & \mapsto & (L_1I \cap V_{k,l,\kappa(c)},L_1I \cap V_{1,m+1,c}),\end{array} \]
where again $(k,l)$ is the lowest block in $\overline{S}$. In this case we let ${\mathcal X}_{Y}=X_{S}^{S_c}$ and ${\mathcal Y}_{\overline{Y}}=Y_{\overline{S}}^{S_c}$. The images of the morphisms above are contained in these.

\smallskip

\noindent{\bf Case \ref{sec:proof:orbicells}.3\ }  Assume $m \equiv 1\; \mathrm{mod}\; n-2$  with vertex types A1, A2, A3. This is Case~\ref{subsec:incvargr}.3: $c_1$ and $c_2$ are the labels of the divided block immediately above the block at position $(0,m)$, $S_{1} \subseteq B_{m+1,c_1}$, $S_{2} \subseteq B_{m+1,c_2}$ are their index sets, both (in cases A1 and A2) or one (in case A3) of which is maximal; $S$ is the index
set of the block at position $(0,m)$ with label $j$. We get a morphism 
\[ \begin{array}{ccc} Z_Y & \to & V_{S}\times V_{S_{1},c_1} \times V_{S_{2},c_2} \\ I & \mapsto & (I \cap V_{0,m,j}, I \cap V_{1,m,c_1}, I \cap V_{1,m,c_2})\end{array}\]
whose image is contained in ${\mathcal X}_{Y}=X_{S}^{S_1,S_2}$. In this way we obtain an induced morphism $Z_Y \to {\mathcal X}_{Y, \overline{Y}}$. Similarly, consider the morphism 
\[ \begin{array}{ccc} Z_{\overline{Y}} & \to & V_{\overline{S}} \times V_{S_{1},c_1} \times V_{S_{2},c_2} \\ I & \mapsto & (L_1 I \cap V_{k,l,j},L_1 I \cap V_{1,m,c_1}, L_1 I \cap V_{1,m,c_2}),\end{array}\]
where $(k,l)$ is the lowest block in $\overline{S}$. We obtain an induced morphism $Z_{\overline{Y}} \to {\mathcal Y}_{\overline{Y}}$, where ${\mathcal Y}_{\overline{Y}}=Y_{\overline{S}}^{S_1,S_2}$. 

\smallskip

\noindent{\bf Case \ref{sec:proof:orbicells}.4\ } Assume finally that $m \not\equiv 1\; \mathrm{mod}\; n-2$ with vertex types A1 or A2. This is Case~\ref{subsec:incvargr}.4: the full block at position $(0,m)$ has label $j$ and index set $S$ which is maximal; the full block at position $(1,m)$ has label $c$ and index set $S_{c}$; $S$ is allowed by $S_{c}$. 
We get a morphism
\[ \begin{array}{ccc} Z_Y &\to & V_{S} \times V_{S_{c}} \\ I & \mapsto & (I \cap V_{0,m,j}, I \cap V_{1,m,c})\end{array}\]
whose image is contained in $X_{S}^{S_c} \subseteq V_{S} \times V_{S_{c}}$, an incidence variety we denote by 
${\mathcal X}_{Y}$ to obtain an induced morphism $Z_Y \to {\mathcal X}_{Y, \overline{Y}}$. 

Second, let
\[ \begin{array}{ccc} Z_{\overline{Y}} & \to & V_{\overline{S}} \times V_{S_c} \\ I & \mapsto & (L_1 I \cap V_{k,l,j},L_1 I \cap V_{1,m,c}),\end{array}\]
where $(k,l)$ is the lowest block in $\overline{S}$. By letting ${\mathcal Y}_{\overline{Y}}=Y_{\overline{S}}^{S_c}$ we obtain an induced morphism $Z_{\overline{Y}} \to {\mathcal Y}_{Y, \overline{Y}}$.
\smallskip

The last key step in our inductive proof is the following result, valid in all four cases above.
\begin{proposition} The following is a scheme-theoretic fiber product diagram, with the right hand vertical map in each case given by the map induced by statement (1) of Propositions \ref{prop:incvar1},  \ref{prop:incvar2} \ref{prop:incvar3} or \ref{prop:incvar4} as appropriate.
\label{prop:fiberproduct}
\begin{equation}
\label{eq:prfibeq}
\begin{array}{ccc}
Z_Y  &\ra{\phantom}  & {\mathcal X}_{Y}  \\
\da{T}& &\da{\omega \times \mathrm{Id}} \\
Z_{\overline{Y}} & \ra{\phantom} & {\mathcal Y}_{\overline{Y}}.
\end{array}
\end{equation}
\end{proposition}
\begin{proof} It is immediate from the definitions in the different cases that the diagram is commutative. We thus need to show that it is a fibre product. Let $B$ be an arbitrary base scheme and let $f \colon B \to Z_{\overline{Y}}$ and $g \colon B \to {\mathcal X}_{Y}$ be morphisms; we need to show that these induce a unique morphism $B \to Z_Y$. We consider Case \ref{sec:proof:orbicells}.1; the proof in the other cases is analogous. The map $f$ corresponds to a flat family of ideals ${\mathcal I}_f\lhd {\mathcal O}_B\otimes \SC[x,y]$ with Young wall $\overline Y$. The map $g$ corresponds to a 3-tuple of families ${\mathcal U}_{1,g}, {\mathcal U}_{2,g}, {\mathcal U}_{c,g}$ of subspaces of $\SC[x,y]$ over $B$. 
Given this data, consider the family of ideals 
\[{\mathcal I}_{f,g} = (L_1{\mathcal I}_f, {\mathcal U}_{1,g}, {\mathcal U}_{2,g})\lhd {\mathcal O}_B\otimes \SC[x,y]\]
over $B$. Here the parentheses mean the generated ideal as explained in \ref{subsec:auxiliary}. By the compatibility of $(f,g)$ it is immediate that the Young wall of the corresponding ideals is $Y$. The classifying map of this family is the unique possible extension of $(f,g)$ to a morphism $B \to Z_Y$.
\end{proof}

\begin{proof}[Conclusion of the Proof of Theorem~\ref{thm:Zstrata}]
Assume that we have shown for any Young wall~$Y$ having less then $l$ rows that the corresponding stratum~$Z_Y$ is affine, the $l=1$ case being obvious. Consider an arbitrary Young wall~$Y$ consisting of~$l$ rows. Let $\overline{Y}$ denote its truncation, as defined above. By the induction assumption, the space $Z_{\overline{Y}}$ is affine. Also, by Propositions \ref{prop:incvar1}, \ref{prop:incvar2}, \ref{prop:incvar3} or \ref{prop:incvar4} respectively, the map ${\mathcal X}_{Y}\to {\mathcal Y}_{\overline{Y}}$ of Proposition~\ref {prop:fiberproduct} is a trivial affine fibration in all cases. By Proposition~\ref{prop:fiberproduct}, the map $Z_Y\to Z_{\overline{Y}}$ is a pullback of a trivial affine fibration and thus itself a trivial affine fibration. Using the induction hypothesis, $Z_Y$ is thus an affine space. The proof of Theorem~\ref{thm:Zstrata} is complete.
\end{proof}

\begin{remark} 
One can deduce from the above proof that one can in fact \emph{canonically} choose generators of a homogeneous ideal $I\in Z_Y$, which are in the cells of the some of the salient blocks of $Y$; as discussed before, not all salient cells necessarily contain a generator. For describing the coarse Hilbert scheme we have to keep track of these generators, but we will do this only implicitly.
\end{remark}

\begin{example} Returning to Examples~\ref{ex:4sidepyr}-\ref{ex:5sidepyr}, we see that for $Y_3$ the triangle of side $5$, $\overline{Y_3}=Y_2$, the triangle of size $4$. The map $T\colon Z_{Y_3}\cong\SA^2\to Z_{\bar Y_3}\cong\SA^1$ is the map identified at the end of the discussion of Example~\ref{ex:5sidepyr}.
\end{example}

\subsection{Preparation for the proof  of the incidence propositions}
\label{subsec:auxiliary}

To prepare the ground for the proof of the propositions announced in~\ref{subsec:incvargr}, consider the operators defined in~\ref{subsec:subops}. We use these operators to describe projective coordinates on some of the Grassmannians~$P_{m,j}$. We first record the following equalities, computing the isotypical summands of the homogeneous pieces of the ring $S=\SC[x,y]$.

\begin{lemma} We have 
\label{l0}
\[S_{2k(n-2)}[\rho_{\kappa(0,k)}] = (L_{n-1} + L_{n})^{2k}[\rho_{\kappa(0,k)}]=
 (L_{n-1}^{2} + L_{n}^{2})^{k};
\]
\[ S_{2k(n-2)}[\rho_{\kappa(1,k)}] = (L_{n-1} + L_{n})^{2k}[\rho_{\kappa(1,k)}]=
 L_{n-1} L_{n}(L_{n-1}^{2} + L_{n}^{2})^{k-1};
\]
\[S_{(2k+1)(n-2)}[\rho_{\kappa(n-1,k)}] = (L_{n-1} + L_{n})^{2k+1}[\rho_{\kappa(n-1,k)}]=
 L_{n-1} (L_{n-1}^{2} + L_{n}^{2})^{k};
\]
\[S_{(2k+1)(n-2)}[\rho_{\kappa(n,k)}] = (L_{n-1} + L_{n})^{2k+1}[\rho_{\kappa(n,k)}]=
 L_{n} (L_{n-1}^{2} + L_{n}^{2})^{k}.
\]
\end{lemma}
\begin{proof} For each equality on the right, use~\eqref{eq:repstensor}
and an easy induction argument. For the equalities on the left, $\supseteq$ is always clear; then use dimension counting.
\end{proof}

Thus, given an element $v\in P_{2k(n-2),\kappa(0,k)}$, we can write it uniquely in the form 
\[v=\sum_{i=0}^k \alpha_i L_{n-1}^{2i} L_{n}^{2(k-i)},\] for certain projective coordinates $[\alpha_0 : \dots : \alpha_k]$.
Analogous coordinates exist on $P_{2k(n-2),\kappa(1,k)}$ and $P_{(2k+1)(n-2), j}$ for $j=n-1,n$. 

For the two-dimensional representations we similarly have
\begin{lemma} For $1<j<n-1$, 
\[\begin{array}{rcl} S_{2k(n-2)+j-1}[\rho_j]  & = & L_j \left( (L_{n-1} + L_{n})^{2k}\right)
\\
& = & L_j \left( (L_{n-1} + L_{n})^{2k}[\rho_0]\right)\oplus L_j \left( (L_{n-1} + L_{n})^{ 2k}[\rho_1]\right);\\ \\
S_{(2k+2)(n-2)-j+1}[\rho_j] & = & L_{n-j}\left( (L_{n-1} + L_{n})^{2k+1}\right)\\
& = &  L_{n-j} \left( (L_{n-1} + L_{n})^{2k+1}[\rho_{n-1}]\right) \oplus L_{n-j}\left( (L_{n-1} + L_{n})^{2k+1}[\rho_n]\right).\end{array}\]
\label{lem:otherms}
\end{lemma}
\begin{proof} Analogous. \end{proof}

For such $1<j<n-1$, the space $P_{2k(n-2)+j-1,j}$ of two-dimensional $G_\Delta$-equivariant subspaces of $S_{2k(n-2)+j-1}[\rho_j]$
is a projective space as noted above. 
Using Lemma~\ref{lem:otherms}, we get a collection of distinguished two-dimensional $G_\Delta$-equivariant subspaces
$L_jL_{n-1}^iL_n^{2k-i}$ in $S_{2k(n-2)+j-1}[\rho_j]$; an arbitrary element $v\in P_{2k(n-2)+j-1,j}$ can be uniquely written as
\[v=\sum_{i=0}^{2k} \alpha_i L_j L_{n-1}^{i} L_{n}^{2k-i}\]
for certain projective coordinates $[\alpha_0:\ldots: \alpha_{2k}]$. Analogous coordinates also exist on the space $P_{(2k+2)(n-2)-j+1, j}$.
	
For subspaces $U_1,\dots,U_i \in \mathrm{Gr}(S)$ denote by
$(U_1,\dots,U_i)$ the $G$-invariant ideal of $S$ generated by the
corresponding subspaces. In particular, the ideal generated by (the
subspaces represented by) points $v_1,\dots,v_i \in \SP S$ is denoted
by $(v_1,\dots,v_i)$. Then $(U_1,\dots, U_i)_{m,j}$ is represented by a
projective linear subspace of $P_{m,j}$ (cf.~Lemma
\ref{lem:intdim}). Thus, we can talk about its intersection with the
cells of $P_{m,j}$. For simplicity we will use the notation
$(U_1,\dots,U_i) \cap V_{k,l,j}=(U_1,\dots,U_i)_{k+l,j} \cap
V_{k,l,j}$ for the intersection with $V_{k,l,j}$. 

We need to study indidence relations between ideals generated by subspaces from the various strata defined above. First of all, 
let $v_j\in V_{0,m,j}$ for some $m$ and $j$, corresponding to a full or half block. Denote by $C$ the set of 
labels of full or half blocks immediately above or on the right of this block, i.e.~in the positions $(1,m)$ or $(0,m+1)$.
Then the definition of the McKay quiver clearly implies that we have $(v_j)\cap S_{m+1,c}=\emptyset$ whenever $c\not\in C$.
The following long statement discusses all the remaining cases when $c\in C$, split into the different possibilities.  

\begin{proposition}
\label{prop:1} 

\begin{enumerate}
\item For $j=0,1$, fix $v_j \in V_{0,2k(n-2), j}$.
\begin{enumerate}
\item[(a)] We have $(v_j)\cap \left(\bigcup_{l \geq 1} V_{l,2k(n-2)-l+1,2}\right) = \emptyset$. Hence the unique point $(v_j)\cap P_{2k(n-2)+1}$ necessarily lies in $V_{0,2k(n-2)+1,2}$. This provides an injection \[V_{0,2k(n-2),j} \to V_{0,2k(n-2)+1,2}.\]
\item[(b)] $ (v_0, v_1) \cap \left(\bigcup_{l>1} V_{l,2k(n-2)-l+1,2}\right) = \emptyset$. In particular, the projective line $(v_0,v_1) \cap P_{2k(n-2)+1,2}$ necessarily intersects $V_{1,2k(n-2),2}$. Let the intersection point be $L_1 v_2$ for a certain $v_2 \in V_{0,2k(n-2)-1,2}$. Then $v_2$ is the unique point of $V_{0,2k(n-2)-1,2}$ such that $v_0,v_1 \in (v_2)$. As a consequence, for any projective subspace $U_2\subseteq P_{2k(n-2)-1,2}$, the intersection $(v_0,v_1) \cap \left(\bigcup_{l>0} V_{l,2k(n-2)-l+1,2}\right)$ is contained in $L_1 U_2$ if and only if $v_0, v_1 \in (U_2)$.
\end{enumerate} 

\item Let $ v_2 \in V_{0,2k(n-2)+1, 2}$. For $j=0,1$, if $(v_2) \cap \left( \bigcup_{l>0} V_{l,2k(n-2)-l+2,j} \right)$ is not empty, then it is necessarily one-dimensional, and it equals $L_1 v_j$ for a certain $v_j \in V_{0,2k(n-2), j}$. Exactly one of the following three cases happens.
\begin{itemize}
 \item $(v_2) \cap \left(\bigcup_{l>0} V_{l,2k(n-2)-l+2,1}\right) = L_1 v_0$ and $(v_2) \cap \left(\bigcup_{l>0} V_{l,2k(n-2)-l+2,0}\right) = \emptyset$. This happens if and only if $v_2 \in (v_0)$. In this case, $v_0 \in V_{0,2k(n-2),0}$, and $(v_2)\cap S_{2k(n-2)+2}$ has two (resp.~three if $n=4$) irreducible components: $L_1 v_0$ and $(v_2) \cap V_{0,2k(n-2)+2,3}$ (resp.~$(v_2) \cap V_{0,2k(n-2)+2,3}$ and $(v_2) \cap V_{0,2k(n-2)+2,4}$).
 
 \item $(v_2) \cap \left(\bigcup_{l>0} V_{l,2k(n-2)-l+2,1}\right) = \emptyset$ and $(v_2) \cap \left(\bigcup_{l>0} V_{l,2k(n-2)-l+2,0}\right) = L_1 v_1$ with symmetrical statements as in the previous case.
  
 \item $(v_2) \cap \left(\bigcup_{l>0} V_{l,2k(n-2)-l+2,1}\right) = L_1 v_0$ and $(v_2) \cap \left(\bigcup_{l>0} V_{l,2k(n-2)-l+2,0}\right) = L_1 v_1$. This happens if and only if $v_2 \in (v_0,v_1)$ but  $v_2 \notin (v_0) \cup (v_1)$. In this case at least one of the inclusions $v_0 \in V_{0,2k(n-2),0}$, $v_1 \in V_{0,2k(n-2),1}$ is satisfied but not necessarily both. Furthermore, $(v_2)\cap S_{2k(n-2)+2}$ has three (resp.~four if $n=4$) irreducible components: $L_1 v_0$, $L_1 v_1$ and $(v_2) \cap V_{0,2k(n-2)+2,3}$ (resp.~$(v_2) \cap V_{0,2k(n-2)+2,3}$ and $(v_2) \cap V_{0,2k(n-2)+2,4}$).
\end{itemize}
Thus for $n>4$, we obtain an isomorphism
\[ \begin{array}{rcl} V_{0,2k(n-2)+1,2} & \to & V_{0,2k(n-2)+2,3}\\ v_2 &\mapsto & (v_2) \cap V_{0,2k(n-2)+2,3},\end{array}\] 
whereas for $n=4$, we get an isomorphism
\[ \begin{array}{rcl}V_{0,2k(n-2)+1,2} & \to & V_{0,2k(n-2)+2,3} \times V_{0,2k(n-2)+2,4}\\ v_2& \mapsto &\left( (v_2) \cap V_{0,2k(n-2)+2,3}, (v_2) \cap V_{0,2k(n-2)+2,4}\right). \end{array}\] 
Finally, for projective subspaces $U_0 \subseteq P_{2k(n-2),0}, U_1 \subseteq P_{2k(n-2),1}$, 
\begin{itemize}
 \item  the conditions $(v_2) \cap \left(\bigcup_{l>0} V_{l,2k(n-2)-l+2,1}\right) \subseteq L_1 U_0$ and $(v_2) \cap \left(\bigcup_{l>0} V_{l,2k(n-2)-l+2,0}\right) = \emptyset$ are satisfied if and only if $v_2 \in (U_0)$; 
 \item  the conditions $(v_2) \cap \left(\bigcup_{l>0} V_{l,2k(n-2)-l+2,1}\right) = \emptyset$ and $(v_2) \cap \left(\bigcup_{l>0} V_{l,2k(n-2)-l+2,0}\right) \subseteq L_1 U_1$ are satisfied if and only if $v_2 \in (U_1)$; 
 \item  the conditions $(v_2) \cap \left(\bigcup_{l>0} V_{l,2k(n-2)-l+2,1}\right) \subseteq L_1 U_0$ and $(v_2) \cap \left(\bigcup_{l>0} V_{l,2k(n-2)-l+2,0}\right) \subseteq L_1 U_1$ are satisfied if and only if $v_2 \in (U_0,U_1)$ but  $v_2 \notin (U_0) \cup(U_1)$.
\end{itemize}

\item Assume that $3 \leq j \leq n-3$ (resp.~$j=n-2$) and set $v_j \in V_{0,2k(n-2)+j-1,j}$. Then $(v_j) \cap \left(\bigcup_{l > 1} V_{l,2k(n-2)-l+j,j-1} \right)=\emptyset$. Furthermore, $(v_j)\cap S_{2k(n-2)+j}$ has two (resp.~three) irreducible components: a point $(v_j)\cap V_{1,2k(n-2)+j-1,j-1}$ of the form $L_1 v_{j-1}$ for some $v_{j-1} \in V_{0,2k(n-2)+j-2,j-1}$, which is the unique element with $v_j \in (v_{j-1})$, and another point $(v_j) \cap V_{0,2k(n-2)+j,j+1}$ (resp.~two other points $(v_j) \cap V_{0,2k(n-2)+j,n-1}$, $(v_j) \cap V_{0,2k(n-2)+j,n}$). We obtain an isomorphism \[V_{0,2k(n-2)+j-1,j} \to V_{0,2k(n-2)+j,j+1}\] for $j\leq n-3$ and an isomorphism \[V_{0,2k(n-2)+j-1,j} \to V_{0,2k(n-2)+j,n-1} \times V_{0,2k(n-2)+j,n}\] for $j=n-2$. Also, for a projective subspace $U_{j-1} \subseteq P_{2k(n-2)+j-2,j-1}$, the intersection $(v_j) \cap \left(\bigcup_{l>0} V_{l,2k(n-2)-l+j,j-1}\right)$ is contained in $L_1 U_{j-1}$ if and only if $v_j \in (U_{j-1})$.

\item For $j=n-1,n$, fix $v_j \in V_{0,(2k+1)(n-2), j}$.
\begin{enumerate}
\item[(a)] We have $(v_j)\cap \left(\bigcup_{l \geq 1} V_{l,(2k+1)(n-2)-l+1,n-2}\right) = \emptyset$. Hence, the point $(v_j)\cap P_{(2k+1)(n-2)+1}$ is necessarily in $V_{0,(2k+1)(n-2)+1,n-2}$. This provides an injection \[V_{0,(2k+1)(n-2),j} \to V_{0,(2k+1)(n-2)+1,n-2}.\]
\item[(b)] $ (v_{n-1}, v_n) \cap \left(\bigcup_{l>1} V_{l,(2k+1)(n-2)-l+1,2}\right) = \emptyset$. In particular, the projective line $(v_{n-1},v_n) \cap P_{(2k+1)(n-2)+1,n-2}$ necessarily intersects $V_{1,(2k+1)(n-2),n-2}$. Let the intersection point be $L_1 v_{n-2}$ for a certain $v_{n-2} \in V_{0,(2k+1)(n-2)-1,n-2}$. Then $v_{n-2}$ is the unique point of $V_{0,(2k+1)(n-2)-1,n-2}$ such that $v_{n-1},v_n \in (v_{n-2})$. As a consequence, for any projective subspace $ U_{n-2} \subseteq P_{(2k+1)(n-2)-1,n-2}$, the intersection $(v_{n-1},v_n) \cap \left(\bigcup_{l>0} V_{l,(2k+1)(n-2)-l+1,n-2}\right)$ is contained in $L_1 U_{n-2}$ if and only if $v_{n-1}, v_n \in (U_{n-2})$.
\end{enumerate} 

\item Let $ v_{n-2} \in V_{0,(2k+1)(n-2)+1, n-2}$. For $j=n-1,n$ if $(v_{n-1}) \cap \left( \bigcup_{l>0} V_{l,(2k+1)(n-2)-l+2,j} \right) \neq \emptyset$ then this invariant subspace is one-dimensional, and it is of the form $L_1 v_j$ for certain $v_j \in V_{0,(2k+1)(n-2), j}$. Exactly one of the following three possibilities happens.
\begin{itemize}
 \item $(v_{n-2}) \cap \left(\bigcup_{l>0} V_{l,(2k+1)(n-2)-l+2,n}\right) = L_1 v_{n-1}$ and $(v_{n-2}) \cap \left(\bigcup_{l>0} V_{l,(2k+1)(n-2)-l+2,n-1}\right) = \emptyset$. This happens if and only if $v_{n-2} \in (v_{n-1})$. In this case $v_{n-1} \in V_{0,(2k+1)(n-2),n-1}$, and $(v_2)\cap S_{(2k+1)(n-2)+2}$ has two (resp.~three if $n=4$) irreducible components: $L_1 v_{n-1}$ and $(v_{n-2}) \cap V_{0,(2k+1)(n-2)+2,n-3}$ (resp.~$(v_2) \cap V_{0,(2k+1)(n-2)+2,0}$ and $(v_2) \cap V_{0,(2k+1)(n-2)+2,1}$).
 
 \item $(v_{n-2}) \cap \left(\bigcup_{l>0} V_{l,(2k+1)(n-2)-l+2,n}\right) = \emptyset$ and $(v_{n-2}) \cap \left(\bigcup_{l>0} V_{l,2k(n-2)-l+2,n-1}\right) = L_1 v_{n}$ with symmetrical statements as in the previous case.
  
 \item $(v_{n-2}) \cap \left(\bigcup_{l>0} V_{l,(2k+1)(n-2)-l+2,n}\right) = L_1 v_{n-1}$ and $(v_{n-2}) \cap \left(\bigcup_{l>0} V_{l,(2k+1)(n-2)-l+2,n-1}\right) = L_1 v_n$. This happens if and only if $v_{n-2} \in (v_{n-1},v_n)$ but  $v_{n-2} \notin (v_{n-1}) \cup (v_n)$. In this case at least one of the inclusions $v_{n-1} \in V_{0,(2k+1)(n-2),n-1}$, $v_n \in V_{0,(2k+1)(n-2),n}$ is satisfied but not necessarily both. Furthermore, $(v_{n-2})\cap S_{(2k+1)(n-2)+2}$ has three (resp.~four if $n=4$) irreducible components: $L_1 v_{n-1}$, $L_1 v_{n}$ and $(v_{n-2}) \cap V_{0,(2k+1)(n-2)+2,n-3}$ (resp.~$(v_{n-2}) \cap V_{0,(2k+1)(n-2)+2,0}$ and $(v_{n-2}) \cap V_{0,(2k+1)(n-2)+2,1}$).
\end{itemize}
For $n>4$, we obtain isomorphisms 
\[\begin{array}{rcl} V_{0,(2k+1)(n-2)+1,n-2} & \to & V_{0,(2k+1)(n-2)+2,n-3}\\ v_{n-2} & \mapsto & (v_{n-2}) \cap V_{0,(2k+1)(n-2)+2,n-3}\end{array},\]
whereas for $n=4$ we obtain an isomorphism 
\[\begin{array}{rcl} V_{0,(2k+1)(n-2)+1,n-2} & \to & V_{0,2k(n-2)+2,0} \times V_{0,(2k+1)(n-2)+2,1}\\ v_{n-2} & \mapsto& \left( (v_{n-2}) \cap V_{0,(2k+1)(n-2)+2,0}, (v_{n-2}) \cap V_{0,(2k+1)(n-2)+2,1}\right)\end{array}. \] 
Moreover, for projective subspaces $U_{n-1} \subseteq P_{(2k+1)(n-2),n-1}, U_n \subseteq P_{(2k+1)(n-2),n}$ the conditions 
\begin{itemize}
 \item $(v_{n-2}) \cap \left(\bigcup_{l>0} V_{l,(2k+1)(n-2)-l+2,n}\right) \subseteq L_1 U_{n-1}$ and $(v_{n-2}) \cap \left(\bigcup_{l>0} V_{l,(2k+1)(n-2)-l+2,n-1}\right) = \emptyset$ are satisfied if and only if $v_{n-2} \in (U_{n-1})$; 
 \item $(v_{n-2}) \cap \left(\bigcup_{l>0} V_{l,(2k+1)(n-2)-l+2,n}\right) = \emptyset$ and $(v_{n-2}) \cap \left(\bigcup_{l>0} V_{l,(2k+1)(n-2)-l+2,n-1}\right) \subseteq L_1 U_n$ are satisfied if and only if $v_{n-2} \in (U_n)$; 
 \item $(v_{n-2}) \cap \left(\bigcup_{l>0} V_{l,(2k+1)(n-2)-l+2,n}\right) \subseteq L_1 U_{n-1}$ and $(v_{n-2}) \cap \left(\bigcup_{l>0} V_{l,(2k+1)(n-2)-l+2,n-1}\right) \subseteq L_1 U_n$ are satisfied if and only if $v_{n-2} \in (U_{n-1},U_n)$ but  $v_2 \notin (U_{n-1}) \cup(U_n)$.
\end{itemize}

\item Assume that $3 \leq j \leq n-3$ (resp.~$j=2$) and set $v_j \in V_{0,(2k+2)(n-2)-j+1,j}$. Then $(v_j) \cap \left(\bigcup_{l > 1} V_{l,(2k+2)(n-2)-l-j+2,j+1} \right)=\emptyset$. Furthermore, $(v_j)\cap S_{(2k+2)(n-2)-j+2}$ has two (resp.~three) irreducible components: a point $(v_j)\cap V_{1,(2k+2)(n-2)-j+1,j+1}$ of the form $L_1 v_{j+1}$ for some $v_{j+1} \in V_{0,(2k+2)(n-2)-j,j+1}$, which is the unique element with $v_j \in (v_{j+1})$, and another point $(v_j) \cap V_{0,(2k+2)(n-2)-j+2,j-1}$ (resp.~two other points $(v_j) \cap V_{0,(2k+2)(n-2)-j+2,0}$, $(v_j) \cap V_{0,(2k+2)(n-2)-j+2,1}$) providing an isomorphism $V_{0,(2k+2)(n-2)-j+1,j} \to V_{0,2k(n-2)-j+2,j-1}$ (resp.~$V_{0,(2k+2)(n-2)-j+1,j} \to V_{0,2k(n-2)-j+2,0} \times V_{0,2k(n-2)-j+2,1}$). As a consequence, for a projective subspace $U_{j+1} \subseteq P_{(2k+2)(n-2)-j,j+1}$, the intersection $(v_j) \cap \left(\bigcup_{l>0} V_{l,(2k+2)(n-2)-l-j+2,j+1}\right)$ is contained in $L_1 U_{j+1}$ if and only if $v_j \in (U_{j+1})$.

\end{enumerate}
\end{proposition}

\begin{proof}
For ease of notation in the proof, we will assume that $n$ is even. If $n$ is odd, then the argument works in the same way, except that $\kappa$ should be applied to the indices when appropriate.

(1) Statement (a) follows immediately from the definition of the cells. For the first part of (b),
write $v_{0}=\sum_{i=0}^k \alpha_i L_{n-1}^{2i} L_{n}^{2(k-i)}$ and $v_{1}=\sum_{i=0}^{k-1} \beta_i L_{n-1}^{2i+1} L_{n}^{2(k-i)-1}$. The assumptions guarantee that $\sum \alpha_i \neq 0$ and $\sum \beta_i \neq 0$. The image $(v_0,v_1) \cap P_{2k(n-2)+1,2}$, as subset of $Gr(2, S_{2k(n-2)+1}[\rho_{2}])$, is a projective line and is spanned by
\[ L_2 v_{0}=\sum_{i=0}^k \alpha_i L_2 L_{n-1}^{2i} L_{n}^{2(k-i)} \in L_2 (L_{n-1}^2 + L_{n}^2)^{k}\]
and
\[ L_2 v_{1}=\sum_{i=0}^{k-1} \beta_i L_2 L_{n-1}^{2i+1} L_{n}^{2(k-i)-1} \in L_2 L_{n-1} L_{n}(L_{n-1}^{2} + L_{n}^{2})^{k-1}.\]
In particular, if $(v_0,v_1) \cap V_{1,2k(n-2)}=\{L_1 v_2\}$, then there exist vectors $v_x, v_y$ in the two-dimensional vector space $v_2$ satisfying $v_y= \tau(v_x)$, as well as $a_x, b_x, a_y, b_y \in \SC$, so that
\begin{equation}
\label{eq:lem1eq1}
\begin{aligned}
L_1 v_x & =  a_x L_{2,x} v_0+b_x L_{2,x} v_1, \\
L_1 v_y & = a_y L_{2,y} v_0+b_y L_{2,y} v_1.
\end{aligned}
\end{equation}
In $v_0$ the highest powers of $x$ and $y$ are $x^{2k(n-2)}$ and $y^{2k(n-2)}$, both with coefficient $\sum_i \alpha_i$. In $v_1$ the highest powers of $x$ and $y$ are $x^{2k(n-2)}$ and $y^{2k(n-2)}$, the first has coefficient $\sum_i \beta_i$, the second has coefficient $-\sum_i \beta_i$. We apply $L_{2,x}$ to these. In order for the sum to avoid the cell $V_{0,2k(n-2)+1}$ the coefficients must satisfy $[a_x:b_x]=[\sum_i \beta_i:-\sum_i \alpha_i]$, since in this case the coefficient of $x^{2k(n-2)+1}$ is 0. Then the linear combination is necessarily in $V_{1,2k(n-2)}$. Similarly, when applying $L_{2,y}$, the required coefficients are $[a_y:b_y]=[\sum_i \beta_i:\sum_i \alpha_i]$, so we have $a_x b_y=-a_y b_x$. Therefore, $b_y L_{2,y} L_1 v_x - b_x L_{2,x}L_1 v_y= b_y a_x L_{2,y} L_{2,x} v_0 - b_x a_y L_{2,x} L_{2,y} v_0 = L_1 v_0, $
where we have used that $L_{2,x} L_{2,y}= L_1$. As a consequence, $b_y L_{2,y} v_x + b_x L_{2,x} v_y=v_0$ and similarly $-a_y L_{2,y} v_x + a_x L_{2,x} v_y=v_1$, i.e. $v_0,v_1 \in (v_2)$. This proves the first part of (b). The second part of (b), concerning projective subspaces, follows immediately from the first part.

(2)
For the first part of the statement let $v_2=\sum_{i=0}^{2k} \varepsilon_i L_2 L_{n-1}^{i} L_{n}^{2k-i}=L_2 \left(\sum_{i=0,i \textrm{ even}}^{2k} \varepsilon_i  L_{n-1}^{i} L_{n}^{2k-i}+\sum_{i=0,i \textrm{ odd}}^{2k} \varepsilon_i  L_{n-1}^{i} L_{n}^{2k-i}\right)=L_2(v_{\textrm{even}}+v_{\textrm{odd}})$. If $n\neq 4$, then by applying $L_2$ again we get $L_2^2(v_{\textrm{even}}+v_{\textrm{odd}})=(L_1+L_3)(v_{\textrm{even}}+v_{\textrm{odd}})$, where the first sum is operator sum, and the second is vector sum. So $L_2 v_2 [\rho_0]=\{L_1 v_{\textrm{even}}\}$ and $L_2 v_2 [\rho_1]=\{L_1 v_{\textrm{odd}}\}$. If $(v_2) \cap \left(\bigcup_{l>1} V_{l,2k(n-2)-l+2,1}\right) = L_1 v_0$ and $(v_2) \cap \left(\bigcup_{l>1} V_{l,2k(n-2)-l+2,0}\right) = \emptyset$, then $v_{\textrm{even}}=v_0$ and $v_{\textrm{odd}}=0$. The second case is just the opposite, and when $(v_2) \cap \left(\bigcup_{l>1} V_{l,2k(n-2)-l+2,1}\right) = L_1 v_0$ and $(v_2) \cap \left(\bigcup_{l>1} V_{l,2k(n-2)-l+2,0}\right) = L_1 v_1$, then $v_{\textrm{even}}=a v_0$ and $v_{\textrm{odd}}=b v_1$ for some coefficients $a,b \in \SC$. If $n=4$, then $L_2^2(v_{\textrm{even}}+v_{\textrm{odd}})=(L_1+L_3+L_4)(v_{\textrm{even}}+v_{\textrm{odd}})$, and the rest is very similar to the first case.

For the second part of the statement observe that the results so far imply that $V_{0,2k(n-2)+1, 2}$ stratifies into disjoint, locally closed subspaces
\[
\begin{aligned}
V_{0,2k(n-2)+1, 2} 
= & \bigsqcup _{(v_{c_1},v_{c_2}) \in P_{2k(n-2),c_1} \times P_{2k(n-2),c_2}} \left( (v_{c_1},v_{c_2})\setminus ((v_{c_1}) \cup (v_{c_2})) \cap V_{0,2k(n-2)+1,2} \right)\\
 & \bigsqcup\bigsqcup_{v_{c_2} \in V_{0,2k(n-2),c_2}}\left( (v_{c_2}) \cap V_{0,2k(n-2)+1,2}  \right) \\
 & \bigsqcup\bigsqcup_{v_{c_1} \in V_{0,2k(n-2),c_1}}\left( (v_{c_1}) \cap V_{0,2k(n-2)+1,2}    \right).
\end{aligned}
\]
The subset $\cup_{(v_{c_1},v_{c_2}) \in V_{0,2k(n-2),c_1} \times V_{0,2k(n-2),c_2}} \left( (v_{c_1},v_{c_2})\setminus ((v_{c_1}) \cup (v_{c_2})) \cap V_{0,2k(n-2)+1,2}\right)$ is dense in the third stratum, since $V_{0,2k(n-2),c_1} \times V_{0,2k(n-2),c_2}$ is dense in $P_{2k(n-2),c_1} \times P_{2k(n-2),c_2}$. The first statement of (2) implies that the second statement is valid if $v_2$ is in this subset of $V_{0,2k(n-2)+1, 2}$. Similarly, the first statement implies the second statement on the loci $\sqcup_{v_{c_1} \in V_{0,2k(n-2),c_1}}\left( (v_{c_1}) \cap V_{0,2k(n-2)+1,2}  \right)$ and $\sqcup_{v_{c_2} \in V_{0,2k(n-2),c_2}}\left( (v_{c_2}) \cap V_{0,2k(n-2)+1,2}  \right)$. For $v_2$ in the closed complement of the union of these loci, the second statement follows from the linearity (and thus continuity) of the solution of (\ref{eq:lem1eq1}), since the $U_i$ are projective. 

(3) The statements in this case follow similarly to (2) by observing that $L_2 L_{j}=L_1 L_{j-1}+L_{j+1}$ (resp.~$L_2 L_{n-2}=L_1 L_{n-3}+L_{n-1}+L_{n}$).

The cases (4), (5) and (6) are analogous.
\end{proof}

Consider a full block in position $(0,m)$ with $m=k(n-2)+1$, with label $j$ which is $2$ or $n-2$. 
In positions $(0,k(n-2))$ and $(1,m)$, above and to the left of this full block, 
are divided blocks with labels $(c_1, c_2)$, either $(0,1)$ or $(n-1, n)$. 
The next lemma gives $P_{m,j}$ the structure of a join  (see~\ref{sec:joins} in the Appendix)
of two projective subspaces.

\begin{lemma}\label{lem:imgdescr}
\begin{enumerate}
 \item The morphisms $V_{0,k(n-2),c_i}  \to  V_{0,m,j}$ constructed in Proposition \ref{prop:1} extend to 
morphisms
\[\begin{array}{rcccl}  \phi_{i}& \colon & P_{k(n-2),c_i} & \to & P_{m,j} \\
  && v & \mapsto & (v) \cap P_{m,j}.\end{array}\]
The morphism $\phi_i$ is injective with image $N_{c_i}^0:={\rm im}(\phi_i)\subset P_{m,j}$ such that 
$N_{c_1}^0$, $N_{c_2}^0$ are disjoint projective linear subspaces of $P_{m,j}$. 
\item The join
of the disjoint linear subspaces $N_{c_1}^0, N_{c_1}^0\subset P_{m,j}$ is 
$P_{m,j}$ itself. Thus given $(v_1, v_2)\in P_{k(n-2),c_1} \times P_{k(n-2),c_2}$, there is a projective line 
$\SP^1\cong v_1v_2\subset P_{m,j}$
containing both $\phi_i(v_i)$, namely, the {\em line defined by} $v_1, v_2$ with {\em endpoints} $\phi_i(v_i)$.
The lines $v_1v_2$ cover $P_{m,j}$. For all 
$(v_1,v_2), (v_1', v_2') \in P_{k(n-2),c_1} \times P_{k(n-2),{c_2}}$, the 
intersection $v_1v_2 \cap v_1'v_2'$ can only be at a common endpoint.
\item For all $(v_1,v_2) \in P_{k(n-2),c_1} \times P_{k(n-2),c_2}$,  the intersection 
$ v_1v_2 \cap V_{0,m,j}$ is  
 \begin{itemize}
 \item either empty, exactly when $v_1 \notin V_{0,k(n-2),c_1}$ and $v_2  \notin V_{0,k(n-2),c_2}$;
 \item or an affine line otherwise.
 \end{itemize} 
\end{enumerate}
\end{lemma}
\begin{proof} (1) is immediate.  (2) then follows from $\dim P_{k(n-2),c_1} + \dim P_{k(n-2),c_i} + 1 = \dim P_{m,j}$
and Lemma~\ref{lemma_linear_join}. (3) is again immediate.
\end{proof} 
As we did already in the statement above, we will sometimes omit the inclusion maps $\phi_i$; thus, for subspaces $U_{1} \subseteq P_{k(n-2),c_1} $ and $U_{2} \subseteq P_{k(n-2),c_2}$, we will denote by  $J(U_{1},U_{2}) \subseteq P_{m,j}$ the join of $\phi_1(U_{1})$ and $\phi_1(U_{1})$ in $P_{m,j}$.

Let 
\[M^0_{c_i}:=\phi_{i}(V_{0,k(n-2),c_i})\subset V_{0,m,j};\] 
these are disjoint affine linear subspaces of the affine space $V_{0,m,j}$.
Also consider 
\[ N_{c_i} = J(P_{k(n-2),c_i}, \overline P_{k(n-2),c_{3-i}})\subset P_{m,j} .\]
This is the locus of points in $P_{m,j}$ covered by lines $v_1v_2$ one of whose endpoints is at a point 
``at infinity'', in $\overline P_{k(n-2),c_{3-i}}=P_{k(n-2),c_{3-i}}\setminus V_{0,k(n-2), c_{3-i}}$. Let
\[ M_{c_i} = N_{c_i}\cap V_{0,m,j}\subset V_{0,m,j}\]
be the intersection with the large affine cell of $P_{m,j}$. 

\begin{lemma} 
\label{lem:mctrivbundle}
There exists morphisms $\psi_i\colon M_{c_i} \to M^0_{c_i}$, given by 
associating to a point \[v\in M_{c_i} \subset V_{0,m,j}\] the 
``non-infinity'' endpoint of the (mostly unique) line $v_1v_2$ passing through it. 
The maps $\psi_i$ are trivial vector bundles over affine spaces. 
\end{lemma}
\begin{proof} See Lemma~\ref{lemma_linear_proj}. 
\end{proof}

\begin{corollary} \label{cor:mcstrat}
\begin{enumerate}
 \item For $i=1,2$ the decomposition $P_{k(n-2),c_{3-i}}=\bigsqcup_{(k',l') \in B_{k(n-2),c_{3-i}}} V_{k',l',c_{3-i}}$ induces a decomposition into locally closed subspaces
 \begin{equation} \label{eq:mcdec1} M_{c_i} \setminus M_{c_i}^0 =\bigsqcup_{(k',l') \in B_{k(n-2),c_{3-i}} \setminus {(1,m)}} \left( (J(V_{0,k(n-2),c_i},V_{k',l',c_{3-i}}) \cap V_{0,m,j})\setminus M_{c_i}^0\right).\end{equation}
 \item Taking into account the bijections $B_{k(n-2),c_i}\cong B_{k(n-2)+2,c_{3-i}}$ and the decomposition (\ref{eq:mcdec1}), the space $V_{0,m,j}$ decomposes into locally closed subspaces as
\[
\begin{aligned}
V_{0,m,j} 
= & V_{0,m,j}(1,m,1,m)\\
 & \bigsqcup \left(\bigsqcup_{(k_1,l_1) \in B'_{k(n-2)+2,c_1}} V_{0,m,j}(k_1,l_1,1,m)\right)\\
 & \bigsqcup \left(\bigsqcup_{(k_2,l_2) \in B'_{k(n-2)+2,c_2}} V_{0,m,j}(1,m,k_2,l_2)\right),
\end{aligned}
\]
where we introduced the notations
\begin{itemize}
\item $B'_{k(n-2)+2,c_i}=(B_{k(n-2)+2,c_i}\cup \{\emptyset\}) \setminus \{(1,m)\}$;
\item $V_{0,m,j}(\emptyset,1,m)=M_{c_1}^0$;
\item $V_{0,m,j}(k_1,l_1,1,m)= (J(V_{0,k(n-2),c_1},V_{k_1,l_1,c_{2}}) \cap V_{0,m,j})\setminus M_{c_1}^0$;
\item $V_{0,m,j}(1,m,\emptyset)=M_{c_2}^0$;
\item $V_{0,m,j}(1,m,k_2,l_2)= (J(V_{0,k(n-2),c_2},V_{k_2,l_2,c_{1}}) \cap V_{0,m,j})\setminus M_{c_2}^0$;
\item $V_{0,m,j}(1,m,1,m)=V_{0,m,j} \setminus (M_{c_1} \cup M_{c_2})$.
\end{itemize}

 \end{enumerate}
\end{corollary}
The meaning of the notations is that $V_{0,m,j}(k_1,l_1,k_2,l_2)$ consists of exactly those points $v \in V_{0,m,j}$ such that $(v) \cap P_{k(n-2)+2,c_i} \in V_{k_i,l_i,c_i}$ for $i=1,2$. The symbol $\emptyset$ at an argument replacing a pair $(k_i,l_i)$ means that there is no such intersection.

\subsection{Proofs of propositions about incidence vareties}
\label{subsec:prop:proofs}

Here we prove the propositions announced in~\ref{subsec:incvargr}. The arguments for Propositions \ref{prop:incvar1}, \ref{prop:incvar2} and \ref{prop:incvar4} are very similar, so we will spell out the proof for one of these. The discussion will also prepare the ground for the proof of Proposition~\ref{prop:incvar3}, which is substantially more complicated.

Consider first the situation of \ref{prop:incvar4}. Namely, $m \not\equiv 0, 1\; \mathrm{mod}\; n-2$ is a positive integer, $V_{0,m,j}$ the cell of a full block, $c$ is the label of the full block immediately above the position $(0,m)$, $S_{c} \subseteq B_{m+1,c}$ is a nonempty maximal subset, and $S \subseteq B_{m,j}$ is a maximal subset which is \emph{allowed} by $S_{c}$.

\begin{proof}[Proof of Proposition \ref{prop:incvar4}] By Proposition \ref{prop:1} (3) and (6), for an arbitrary $U \in V_S$, $(U, U_c) \in X_{S}^{S_c}$ if and only if $U \subseteq (L_1^{-1} U_c)\cap V_{0,m,j}$. Moreover, the composition of $L_1^{-1}$ and the isomorphism $V_{0,m-1,c} \to V_{0,m,j}$ gives an isomorphism $V_{1,m,c} \to V_{0,m,j}$. Hence, for a pair $(\overline{U},U_c) \in Y_{\overline{S}}^{S_c}$, we have
\[\{U \in V_S|_{\overline{U}} : (U, U_c) \in  X_{S}^{S_c} \}= \left((L_1^{-1} U_c)\cap V_{0,m,j}\right)/\overline{U},\]
and there is a canonical quotient map $V_{0,m,j}/\overline{U} \to V_{1,m,c}/U_c$.

Let us define two families parameterized by $V_{\overline{S}} \times V_{S_{c},c}$. The family $\mathcal{F}_c$ is defined to have the fiber $ (L_1^{-1}U_{c}) \cap V_{0,m,j}$ over a pair $(\overline{U},U_c) \in V_{\overline{S}} \times V_{S_{c},c}$. This is a family of affine subspaces of $V_{0,m,j}$. The family $\overline{\mathcal{F}}$ is defined to have the fiber $\overline{U} \subset P_{m,j}$ over a pair $(\overline{U},U_c) \in V_{\overline{S}} \times V_{S_{c},c}$. This is a family of projective subspaces contained in $\overline{P}_{m,j}$. Since the tautological bundle over any Schubert cell in any Grassmannian is trivial, the two families are trivial with affine, respectively projective space fibres. Consider these families over the subset $Y_{\overline{S}}^{S_c} \subset V_{\overline{S}} \times V_{S_{c},c}$. By construction, over each point of $Y_{\overline{S}}^{S_c}$, the fibre of $\overline{\mathcal{F}}$  is a subspace of the projective closure of the fiber of $\mathcal{F}_c$ over the same point. In particular, we can take quotients fiberwise. By the considerations above,
\[ X_{S}^{S_c}=\mathcal{F}_c / \overline{\mathcal{F}}. \]
Moreover, the morphism $\omega \times \mathrm{Id} \colon  X_{S}^{S_c} \to Y_{\overline{S}}^{S_c}$ over a pair $(\overline{U},U_c)$ is given by the quotient morphism $V_{0,m,j}/\overline{U} \to V_{1,m,c}/U_c$ times the identity. This shows (1) and (2).

The injectivity statement (3) follows again from the isomorphism $V_{1,m,c} \cong V_{0,m,j}$ given by $L_1$, since for every pair $(U, U_c) \in  X_{S}^{S_1,S_2}$ one has $U_c=(U,\overline{U}_c) \cap V_{1,m,c}$.
\end{proof}

Consider now the situation of 
Proposition \ref{prop:incvar3}; thus $m \equiv 1\; \textrm{mod}\; n-2$ is a positive integer, $c_1$ and $c_2$ are the labels of the divided block immediately above the block at position $(m,j)$, $S_{1} \subseteq B_{m+1,c_1}$, $S_{2} \subseteq B_{m+1,c_2}$ are nonempty subsets at least one of which is maximal, and $S \subseteq B_{m,j}$ is a maximal subset which is allowed by $S_{1}$ and $S_{2}$.

\begin{lemma} For $i=1,2$ fix $U_i \in V_{S_{i},c_i}$.
\begin{enumerate} 
\label{lem:incfiber}
\item[(a)]  For an arbitrary $U \in V_S$, $(U, U_1, U_2) \in X_{S}^{S_1,S_2}$ if and only if $U \subseteq J(\phi_1(L_1^{-1} U_1),\phi_2(L_1^{-1} U_2))\cap V_{0,m,j}$.
\item[(b)] If $(\overline{U}, U_1,U_2) \in Y_{\overline{S}}^{S_1,S_2}$, then $\overline{U}\subseteq J(\phi_1(L_1^{-1} U_1),\phi_2(L_1^{-1} U_2))\cap V_{0,m,j}$.
\item[(c)]  If $(\overline{U}, U_1,U_2) \in Y_{\overline{S}}^{S_1,S_2}$, then 
\[\{U \in V_S|_{\overline{U}} : (U, U_1,U_2) \in  X_{S}^{S_1,S_2} \}= \left(J(\phi_1(L_1^{-1} U_1),\phi_2(L_1^{-1} U_2))\cap V_{0,m,j}\right)/\overline{U}.\]
\end{enumerate}
\end{lemma}
\begin{proof}
(a) By Proposition \ref{prop:1} for any pair of vectors $(v_{1},v_{2}) \in U_{1}\times U_{2}$ those points of $P_{m,j}$ for which $(v_{1},v_{2}) \cap P_{m+1,c_i}$ is either $v_{i}$ or empty are exactly those which are on $J(\phi_1(L_1^{-1}v_{1}),\phi_2(L_1^{-1}v_{2}))$. Hence, to satisfy the conditions $U$ has to be a subset of
\[\bigcup_{(v_1,v_2) \in U_1\times U_2} J(\phi_1(L_1^{-1}v_{1}),\phi_2(L_1^{-1}v_{2}))\cap V_{0,m,j}=J(\phi_1(L_1^{-1}U_{1}),\phi_2(L_1^{-1}U_{2}))\cap V_{0,m,j}.\]

(b) If $(\overline{U}, U_1,U_2) \in Y$, then $(\overline{U}) \cap P_{m+1,c_i} \subseteq U_i$. Hence,
$\phi_i(L_1^{-1}((\overline{U}) \cap P_{m+1,c_i})) \subseteq \phi_i(L_1^{-1}U_i)$, and
\[ J(\phi_1(L_1^{-1}((\overline{U}) \cap P_{m+1,c_1})),\phi_2(L_1^{-1}((\overline{U}) \cap P_{m+1,c_2}))) \subseteq J(\phi_1(L_1^{-1}U_1),\phi_2 (L_1^{-1}U_2)).\]
By Proposition \ref{prop:1} there is an isomorphism $V_{1,m-1,j} \cong V_{0,m-1,c_1} \times V_{0,m-1,c_1}$ in such a way that
\[ \overline{U} \cap V_{1,m-1,j} \subseteq J(\phi_1(L_1^{-1}((\overline{U}) \cap P_{m+1,c_1})),\phi_2(L_1^{-1}((\overline{U}) \cap P_{m+1,c_2})))\cap V_{1,m-1,j}.\]
Similarly, on each cell $V_{k,l,j}$ such that $k+l=m$ and $k\geq 1$, the affine subspace $\overline{U} \cap V_{k,l,j}$ is a subvariety of $J(\phi_1(L_1^{-1}((\overline{U}) \cap P_{m+1,c_1})),\phi_2(L_1^{-1}((\overline{U}) \cap V_{m+1,c_2}))) \cap V_{k,l,j}$. All these mean that $\overline{U} \subseteq  J(\phi_1(L_1^{-1}U_1), \phi_2(L_1^{-1}U_2))$.

(c) Recall, that $\overline{U}$ also represents a subspace at infinity for $V_{0,m,j}$, and $V_S|_{\overline{U}}=V_{0,m,j}/\overline{U}$. In fact, we can take the quotient of an arbitrary subspace of $V_{0,m,j}$, whose closure in $P_{m,j}$ contains $\overline{U}$ with respect to (an arbitrary affine subspace representing) $\overline{U}$. Then the statement follows from (a) and (b).
\end{proof}
\begin{proof}[Proof of Proposition \ref{prop:incvar3}] 
It follows from the definitions that $\left(\omega \times \mathrm{Id} \times \mathrm{Id}\right)(X_{S}^{S_1,S_2}) \subseteq Y_{\overline{S}}^{S_1,S_2}$. The surjectivity will follow from the calculation of the fibers. 

We will define three families of subspaces in $P_{m,j}$ over $V_{\overline{S}} \times V_{S_{1},c_1} \times V_{S_{2},c_2}$. For $i=1,2$ the family $\mathcal{F}_i$ is defined to have the fiber $ \phi_i(L_1^{-1}U_{i}) \subseteq P_{m,j}$ over a three-tuple $(\overline{U},U_1,U_2) \in V_{\overline{S}} \times V_{S_{1},c_1} \times V_{S_{2},c_2}$. Let the third family $\mathcal{F}$ has the fiber $\overline{U} \subseteq P_{m,j}$ over the same element. This is of course empty, if $|S|=1$. It is important to note, that in all cases the fibers are always projective subspaces of $P_{m,j}$.

By Lemma \ref{lem:imgdescr}, there is an embedding $\phi_i \circ L_1^{-1}: P_{m+1,c_{3-i}} \to N_{c_i}^0\subset P_{m,j}$. Apply this embedding on the fibers of the projectivization of the tautological bundle over the Schubert cell $V_{S_{3-i}}$. Then multiply the base with $V_{\overline{S}} \times V_{S_i}$, and extend the family into this direction as a constant. This gives the bundle $\mathcal{F}_i$. Again, by the fact that the tautological bundle over any Schubert cell is trivial it follows that the $\mathcal{F}_i$'s are also trivial, that is, $\mathcal{F}_i \cong \SP^{|S_i|-1} \times V_{\overline{S}} \times V_{S_1} \times V_{S_2}$. Similarly, $\mathcal{F} \cong \SP^{|S|-2} \times V_{\overline{S}} \times V_{S_1} \times V_{S_2}$.

By Lemma  \ref{lem:joinbasechange}, the join of trivial families over a common base is a trivial family of the joins of the fibers:
\[\begin{aligned} 
J(\mathcal{F}_1,\mathcal{F}_2)& =J(\SP^{|S_1|-1} \times V_{\overline{S}} \times V_{S_1} \times V_{S_2}, \SP^{|S_2|-1} \times V_{\overline{S}} \times V_{S_1} \times V_{S_2}) \\ & \cong J(\SP^{|S_1|-1},\SP^{|S_2|-1}) \times V_{\overline{S}} \times V_{S_1} \times V_{S_2} \\
&  \cong \SP^{|S_1|+|S_2|-1} \times V_{\overline{S}} \times V_{S_1} \times V_{S_2}  \subseteq P_{m,j} \times V_{\overline{S}} \times V_{S_1} \times V_{S_2}.
\end{aligned} \]
Therefore, $J(\mathcal{F}_1,\mathcal{F}_2) \cap (V_{0,m,j} \times V_{\overline{S}} \times V_{S_1} \times V_{S_2})$ is a trivial family of affine subspaces of $V_{0,m,j}$ over $V_{\overline{S}} \times V_{S_1} \times V_{S_2}$.

By Lemma \ref{lem:incfiber} (b) $\mathcal{F}$ is a (trivial) subfamily of $J(\mathcal{F}_1,\mathcal{F}_2)$ over $Y_{\overline{S}}^{S_1,S_2}$. By Lemma \ref{lem:incfiber} (c) $X_{S}^{S_1,S_2}$ can be constructed as 
\[ X_{S}^{S_1,S_2} = (J(\mathcal{F}_1,\mathcal{F}_2) \cap (V_{0,m,j} \times V_{\overline{S}} \times V_{S_1} \times V_{S_2}))/\mathcal{F}|_{Y_{\overline{S}}^{S_1,S_2}}. \]
Hence, $X_{S}^{S_1,S_2}$ is a trivial family of affine spaces of dimension $|S_1|+|S_2|-|S|$, since it is the quotient of a trivial affine family of fibre dimension $|S_1|+|S_2|-1$ by another trivial affine family of fibre dimension $|S|-1$.
\end{proof}

\section{Type \texorpdfstring{$D_n$}{Dn}: special loci}
\label{Dnspecialsect}

\subsection{Support blocks}
\label{subsec:support}

In this section, we analyze the cases when a cell corresponding to a salient block (see Def.~\ref{def:globsalient}) of a Young wall $Y$ fails to contain a generator of a corresponding ideal $I\in Z_Y$. As an example, recall once again~Example~\ref{ex:3sidepyr}, where the divided missing blocks at position $(1,3)$ are salient blocks of $Y_3$, but the corresponding cells do not necessarily contain generators of an ideal $I\in Z_{Y_3}$. That this phenomenon can happen at all is one of the main sources of difficulty in our analysis of the strata of the singular Hilbert scheme. We introduce the notion of a support block for a salient block. Intuitively, the intersection of an ideal $I$ with the cell in the support block can generate the intersection in the salient block (such as the support block at position $(0,3)$ for $Y_3$), and thus the salient block contains no new generator of $I$. 
We will make this statement more precise in the rest of this section.

We start with some combinatorial preliminaries. Recall the setup of Proposition~\ref{prop:incvar3}: $m \equiv 1\; \textrm{mod}\; n-2$; $c_1$ and $c_2$ are the labels of the divided block immediately above the block of label $j$ at position $(m,0)$; $S_{1} \subseteq B_{m+1,c_1}$ and $S_{2} \subseteq B_{m+1,c_2}$ are nonempty subsets at least one of which is maximal; $S \subseteq B_{m,j}$ is a maximal subset which is allowed by $S_{1}$ and $S_{2}$. 

For a half-block $b$ of $S_i$, consider the following two conditions. 
\begin{enumerate}
\item The blocks below or to the left of $b$ are not contained in $S$.
\item The block below $b$ is contained in $S$, the complementary half-block $b'$ is contained in $S_{3-i}$, and the block to the left of their position in not contained in $S$.
\end{enumerate}
For $i, j=1,2$, let us denote by $S_i^{s,j} \subset S_i$ the subset of half-blocks of label $c_i$ satisfying condition~$(j)$. Let moreover $S_i^s=S_i^{s,1} \cup S_i^{s,2}$. 

The next lemma, whose proof is immediate, connects the global Definition~\ref{def:globsalient} with the local conditions (1)-(2) where we consider only the $m$-th and $m+1$-st diagonals for a particular $m$, and index sets $S$, $S_1$ and $S_2$ as above.

\begin{lemma} Given a Young wall $Y\in{\mathcal Z}_\Delta$, let $S$, respectively $S_1$ and $S_2$ denote the set of missing blocks, respectively half-blocks of $Y$ on the $m$-th and $(m+1)$-st diagonals. The blocks $S_i^s \subset S_i$ are exactly the salient blocks of $Y$ of label $c_i$ on the $(m+1)$-st diagonal.
\end{lemma}
Thus we can legitimately call the blocks in $S_i^s$ salient blocks in this local situation. 

Let us introduce the following subsets of $S$. 
\begin{itemize}
\item $S^l$ consists of blocks $b\in S$ that are directly to the left of a divided block with labels $(c_1,c_2)$.
\item $S^{b,0}$ consists of blocks $b\in S$, so that $b$ is immediately below a divided block with labels $(c_1,c_2)$, the block immediately up-left of $b$ is not in $S$, and both of the divided blocks above $b$ are in $S_1 \cup S_2$.
\item $S^{b,c_i}$ consists of blocks $b\in S$ that are immediately below a divided block, so that the block immediately up-left of $b$ is not in $S$, and the block of label $c_{3-i}$ above $b$ is in $S_{3-i}$.
\item $S^{b,c_1 \cup c_2}$ consists of blocks $b\in S$ such that $b$ lies immediately below a divided block, and the block immediately up-left of $b$ is contained in $S$. 
\item $S^b=S^{b,0} \cup S^{b,c_1} \cup S^{b,c_2} \cup S^{b,c_1 \cup c_2}$.
\end{itemize}
Note that by the Young wall rules, we necessarily have $S^b=S \setminus S^l$. We will call the blocks in the set $S^{c_i}=S^{b,0} \cup S^{b,c_i} \cup S^{b,c_1 \cup c_2}$ \emph{support blocks for label $c_i$}. We will define a support relation from $S^{c_i}$ to $S_i^s$ in \ref{subsec:loci:strata} below. 

\subsection{Special loci in orbifold strata and the supporting rules}
\label{subsec:loci:strata}

Let $Y\in{\mathcal Z}_\Delta$ be a Young wall with a salient block in its bottom row in position $(0,m)$ with $m \equiv 1\; \textrm{mod}\; n-2$, a full block immediately below a divided block with labels $(c_1, c_2)$. As before, let $S$, respectively $S_1$ and $S_2$ denote the set of missing blocks, respectively half-blocks of $Y$ on the $m$-th and $(m+1)$-st diagonals with the corresponding labels.

We introduce index sets depending on $S$, $S_1$ and $S_2$. We consider two cases. 

If $\overline{S}$ is not maximal, then let
\[I(S,S_1,S_2)=\left\{(k_1,l_1,k_2,l_2) \;:\; \begin{array}{c}(k_i,l_i) \in S_i^s \cup \{ \emptyset \} \cup (\{(1,m)\} \cap S_i) \textrm{ for } i=1,2, \\ \textrm{and at least one } (k_i,l_i)=(1,m)\end{array} \right\}.\]
We partition this index set into the following (possibly empty) disjoint subsets:
\begin{itemize}
\item $I(S,S_1,S_2)_0= \{(k_1,l_1,k_2,l_2) \in I(S,S_1,S_2) \;:\;  (k_i,l_i)\notin \{\emptyset,(1,m)\} \textrm{ for some } i=1,2 \}$;
\item $I(S,S_1,S_2)_1=\{ (1,m,\emptyset),(\emptyset,1,m)\} \cap I(S,S_1,S_2)$;
\item $I(S,S_1,S_2)_{-1}=\{ (1,m,1,m)\} \cap I(S,S_1,S_2)$.
\end{itemize}

If $\overline{S}$ is maximal, then let
\[I(S,S_1,S_2)=\{(k_1,l_1,k_2,l_2) \;:\; (k_i,l_i) \in S_i^s \cup \{ \emptyset \} \textrm{ for } i=1,2 \}.\]
We remark that in this case $(1,m) \notin S_i^s$ for both $i=1,2$. The index set $I(S,S_1,S_2)$ in this case can be partitioned into the following subsets:
\begin{itemize}
\item $I(S,S_1,S_2)_0= \{(k_1,l_1,k_2,l_2) \in I(S,S_1,S_2) \;:\;  (k_i,l_i)\neq \emptyset \textrm{ for some } i=1,2 \}$;
\item $I(S,S_1,S_2)_1=\{(\emptyset,\emptyset)\}$.
\end{itemize}
As before, $\emptyset$ is used as a symbol replacing a pair in these defintions.

For projective subspaces $P_1 \subseteq P_2 \subseteq P_{m+1,c}$ we introduce the following notation. $(P_2 \setminus P_1) \dashv V_{k,l,c}$ if and only if $(P_2 \setminus P_1) \cap V_{k,l,c} \neq \emptyset$ and $k$ is maximal with this property. This is the smallest cell whose intersection with $P_2$ is larger than that with $P_1$.

Recall the truncated Young wall $\overline Y$ and the morphism  $T\colon Z_{Y} \to Z_{\overline{Y}}$ from ~\ref{sec:proof:orbicells}. The following statement will be proved below in~\ref{proofofethrm}. 

\begin{theorem} 
\label{thm:zinfty} There is a decomposition into locally closed subspaces
\[ Z_{Y}= \bigsqcup_{(k_1,l_1,k_2,l_2) \in I(S,S_1,S_2)}Z_{Y}(k_1,l_1,k_2,l_2),\]
where
\[Z_{Y}(k_1,l_1,k_2,l_2)=\{ I \in Z_{Y}\;:\; ((I \cap P_{m,j}) \setminus (I \cap \overline{P}_{m,j})) \cap P_{m+1,c_i} \dashv V_{k_i,l_i,c_i} \textrm{ for } i=1,2\}. \]
The symbol $\emptyset=(k_i,l_i)$ means that there is no intersection with $P_{m+1,c_i}$.
Moreover, if $(k_1,l_1,k_2,l_2) \in I(S,S_1,S_2)_e$, then the nonempty fibers of $T\colon Z_{Y}(k_1,l_1,k_2,l_2) \to Z_{\overline{Y}}$ have Euler characterestic $e$.
\end{theorem}
The space $Z_{Y}(k_1,l_1,k_2,l_2)$ should be thought as the space of those ideals, where the generator in the cell of support block at position $(0,m)$ has an image on the $(m+1)$-st diagonal at the cells $V_{k_i,l_i,c_i}$ for $i=1,2$ which does not come from the rows above.
We remark that the mentioned support block in the bottom row is also salient since it is the first missing block in its row. But if one attaches further rows to the bottom of $Y$, then it may become non-salient.

Recall from \ref{subsec:support} the set $S^{c_i}=S^{b,0} \cup S^{b,c_i} \cup S^{b,c_1 \cup c_2}$ of support blocks for label $c_i$. In the light of the definition of the sets $I(S,S_1,S_2)$ and \autoref{thm:zinfty} we will say that the blocks in $S^{c_i}$ can {\em support} the salient blocks of label $c_i$ in $S_i^s$ in the following sense. 

Each support block $b\in S$ for label $c_i$ can support at most one salient block of label $c_i$ above and to the left of $b$ on the $(m+1)$-st antidiagonal. More precisely, this supporting relationship has to respect the following {\em supporting rules}.
\begin{itemize}
\item Each block in $S^{b,0}$ supports precisely one or two salient blocks, at most one from each label, and at least one of these has to be immediately above it;
\item each block in $S^{b,c_i}$ can support at most one salient block of label $c_i$ which is not immediately above it;
\item each block in $S^{b,c_1 \cup c_2}$ can support none, one or two salient blocks, at most one from each label, and neither of these is immediately above it.
\end{itemize}
In this way we define a correspondence from a subset of $S^{c_1}$ to $S_1^s$ and one from a subset of $S^{c_2}$ to $S_2^s$ but these two have to satisfy the restrictions mentioned above on the intersection $S^{c_1} \cap S^{c_2}=S^{b,0} \cup S^{b,c_1 \cup c_2}$. Neither correspondence has to be surjective or be defined on the whole domain, but, where they are defined, they should be injective.

We shall call a salient block $b'\in S_i$ of label $c_i$ \emph{supported}, if the number of support blocks for label $c_i$ in $S$ which are below $b$ is at least as much as the total number of salient blocks of label $c_i$ in $S_i$ counted from the top left, including $b'$ itself. A supported salient block $b'$ satisfying condition (2) above, so that there is a support block in $S$ immediately below $b'$, will be called \emph{directly supported}. The others will be called \emph{non-directly supported}. The supporting relationship will be globalized for the whole diagram in the notion of closing datum, to be defined in~\ref{subsect:dis0gen} below.

Recall that during the inductive process in the proof of Theorem~\ref{thm:dnorbcells}, at each step a new generator appears in the cell corresponding to the salient block in the bottom row. Assume that for $I \in Z_Y$, $(I \cap P_{m,j}) \cap P_{m+1,c_i} = I \cap P_{m+1,c_i}$ for $i=1,2$. In this case we will say that \emph{there is no generator of label $c_i$ on the $(m+1)$-st antidiagonal}. Let $S$, $S_1$ and $S_2$ be the index sets for $V_{0,m,j}$, $V_{1,m,c_1}$ and $V_{1,m,c_2}$ respectively. Then using inductively Theorem \ref{thm:zinfty} for each block $b \in S^{c_i}$ we get that there is at most one block $b_i \in S_{i}^s$ such that when the row of $b$ is added to the Young wall, the new generator in the cell of $b$ has nontrivial image in the cell of $b_i$. Conversely, for each block $b_i \in S_{i}^s$ there corresponds a support block $b \in S^{c_i}$ determined by $I$. In particular, this implies
\begin{corollary} 
\label{cor:nongen} Assume that for $I \in Z_Y$, $(I \cap P_{m,j}) \cap P_{m+1,c_i} = I \cap P_{m+1,c_i}$ for some $i=1,2$. Let $S$, $S_1$ and $S_2$ be the index sets for $V_{0,m,j}$, $V_{1,m,c_1}$ and $V_{1,m,c_2}$ respectively.
\begin{enumerate}
\item $|S_i^{u,1}| \leq |S^{b,c_i}|+|S^{b,c_1 \cup c_2}|$;
\item every salient block of label $c_i$ is supported;
\item to each salient block of label $c_i$ there corresponds a unique support block for label $c_i$ in the way described above.
\end{enumerate}
\end{corollary}

\subsection{Special loci in Grassmannians}
\label{subsect:specloci}

We prepare the ground for the proof of Theorem \ref{thm:zinfty} by analyzing the incidence varieties of~\ref{subsec:incvargr} in the case $m \equiv 1\; \textrm{mod}\; n-2$. Once again, we use the notations of Proposition \ref{prop:incvar3}. The composition with the projection from $V_{S} \times V_{S_1,c_1} \times V_{S_2,c_2}$ to its first factor, followed by the affine linear fibration $\omega: V_S \to V_{\overline{S}}$, defines a projection map $p_{V_{\overline{S}}}: V_{S} \times V_{S_1,c_1} \times V_{S_2,c_2}\to V_{\overline{S}}$. 

 For $i=1,2$ let $S_i(\overline{U})=\{ (k_i,l_i) \in B_{m+1,c_i}\; :\; (\overline{U}) \cap V_{k_i,l_i,c_i}=\emptyset\}$ be the blocks in the partial profile of $(\overline{U}) \cap P_{m+1,c_i}$ on the $(m+1)$-st diagonal. Then the index sets $I(S,S_1(\overline{U}),S_2(\overline{U}))$ and $I(S,S_1(\overline{U}),S_2(\overline{U}))_e$ introduced above make sense. The following lemma stratifies the fibers of the affine linear fibration $\omega: V_S \to V_{\overline{S}}$.
 
\begin{lemma}
\label{lem:fiberstrat}
For any $\overline{U} \in V_{\overline{S}}$, there is a stratification
\[ V_{0,m,j}/\overline{U}= \bigsqcup_{(k_1,l_1,k_2,l_2) \in I(S,S_1(\overline{U}),S_2(\overline{U}))} V_{0,m,j}(k_1,l_1,k_2,l_2)/\overline{U},\]
where
\[V_{0,m,j}(k_1,l_1,k_2,l_2)/\overline{U}=\{ U \in V_{0,m,j}/\overline{U}\;:\; ((U) \setminus (\overline{U})) \cap P_{m+1,c_i} \dashv V_{k_i,l_i,c_i} \textrm{ for } i=1,2\}. \]
Moreover, if $(k_1,l_1,k_2,l_2) \in I(S,S_1(\overline{U}),S_2(\overline{U}))_e$, then the space $V_{0,m,j}(k_1,l_1,k_2,l_2)/\overline{U}$ is of Euler characterestic $e$.
\end{lemma}
\begin{proof}
We have to distinguish the cases when $\overline{S}$ is maximal or not. The latter case is significantly simpler, so we start with that.

If $\overline{U} \in V_{\overline{S}}$ with $\overline{S}$ not maximal, then $(M_{c_1}+\overline{U}) \cap (M_{c_2}+\overline{U})=\emptyset$ since $M_{c_1}$ and $M_{c_2}$ are distinct parallel hyperplanes in $V_{0,m,j}$, and there are affine subspaces $U_i$ representing $\overline{U}$ such that $U_i \subseteq M_{c_i}$. Recall from Corollary \ref{cor:mcstrat} the stratification of $V_{0,m,j}$ which basically comes from the join structure on its closure $P_{m,j}$. This induces a decomposition of $V_{0,m,j}/\overline{U} $ into non-empty, locally closed, but not necessarily disjoint spaces
\[
\begin{aligned}
V_{0,m,j}/\overline{U} 
= & ((V_{0,m,j}\setminus(M_{c_1} \cap M_{c_2}))/\overline{U}) \cup (M_{c_1}/\overline{U}) \cup (M_{c_1}/\overline{U}) \\
 & \bigcup \left(\bigcup_{(k_1,l_1) \in B_{m+1,c_1}}  ((J(V_{0,m-1,c_1},V_{k_1,l_1,c_{2}}) \cap V_{0,m,j})\setminus M_{c_1}^0+\overline{U})/\overline{U}\right)\\
 & \bigcup \left(\bigcup_{(k_2,l_2) \in B_{m+1,c_1}} ((J(V_{0,m-1,c_2},V_{k_2,l_2,c_{1}}) \cap V_{0,m,j})\setminus M_{c_2}^0+\overline{U})/\overline{U}\right).
\end{aligned}
\]
Consider a block $(k_i,l_i) \in B_{m+1,c_i} \setminus S_i(\overline{U})$. Then the intersection $(\overline{U}) \cap V_{k_i,l_i,c_i}\neq\emptyset$. Assume that there is an $U \in V_{0,m,j}/\overline{U} $ such that $((U)\setminus (\overline{U})) \cap  V_{k_i,l_i,c_i} \neq \emptyset$. Then $\mathrm{dim}((U) \cap V_{k_i,l_i,c_i})>\mathrm{dim}((\overline{U}) \cap V_{k_i,l_i,c_i})$ so by Lemma \ref{lem:intdim} there is at least one other block in a row above $k_i$ which has a trivial intersection with $(\overline{U})$ but a nontrivial one with $(U)$. Hence, for any  $(k_i,l_i) \in B_{m+1,c_i} \setminus S_i(\overline{U})$ we have
\[\{ U \in V_{0,m,j}/\overline{U}\;:\; ((U) \setminus (\overline{U})) \cap P_{m+1,c_i} \dashv V_{k_i,l_i,c_i}\}=\emptyset.\]

On the other hand, if $(k_i,l_i) \in S_i(\overline{U}) \cup \{\emptyset\}$ then
\begin{gather*}
 ((J(V_{0,m-1,c_i},V_{k_i,l_i,c_{3-i}}) \cap V_{0,m,j})\setminus M_{c_i}^0 +\overline{U})/\overline{U} = \\
\{ U \in V_{0,m,j}/\overline{U}\;:\; ((U) \setminus (\overline{U})) \cap P_{m+1,c_i} \dashv V_{k_i,l_i,c_i} \textrm{ and } ((U) \setminus (\overline{U})) \cap P_{m+1,c_{3-i}} \dashv V_{1,m,c_{3-i}}\}.
\end{gather*}
By dimension constrains these spaces are disjoint and it is easy to see that together with $(V_{0,m,j}\setminus(M_{c_1} \cup M_{c_2}))/\overline{U}$,$ M^0_{c_1}/\overline{U}$, and $M^0_{c_2}/\overline{U}$ they cover $V_{0,m,j}/\overline{U}$. Thus we get a stratification
\[
\begin{aligned}
V_{0,m,j}/\overline{U} 
= & V_{0,m,j}(1,m,1,m)/\overline{U}\\
 & \bigsqcup \left(\bigsqcup_{(k_1,l_1) \in (S_1(\overline{U})\cup\{\emptyset\})\setminus \{(1,m)\}} V_{0,m,j}(k_1,l_1,1,m)/\overline{U}\right)\\
 & \bigsqcup\left(\bigsqcup_{(k_2,l_2) \in (S_2(\overline{U})\cup\{\emptyset\})\setminus \{(1,m)\}} V_{0,m,j}(1,m,k_2,l_2)/\overline{U}\right).
\end{aligned}
\]
In particular, there is a stratification 
\begin{equation} \label{eq:mcqstrat} M_{c_{3-i}}/\overline{U}=\bigsqcup_{(k_i,l_i) \in (S_i(\overline{U})\cup\{\emptyset\})\setminus \{(1,m)\}} V_{0,m,j}(k_i,l_i,1,m)/\overline{U}.\end{equation}
Being an affine space, $M^0_{c_i}/\overline{U}$ has Euler characteristic 1 for $i=1,2$. By Lemma \ref{lem:affjoinchar} the spaces $(J(V_{0,k(n-2),c_i},V_{k_i,l_i,c_{3-i}}) \cap V_{0,m,j})\setminus M_{c_i}^0$ have Euler characteristic 0, and the same is true for $((J(V_{0,k(n-2),c_i},V_{k_i,l_i,c_{3-i}}) \cap V_{0,m,j})\setminus M_{c_i}^0+\overline{U})/\overline{U}$. This last step follows from the fact that the subspace $\overline{U} \subset \overline{P}_{m,j}$ avoids both the image of $V_{k_i,l_i,c_{3-i}}$ and $M_{c_i}^0$.

If $\overline{U} \in V_{\overline{S}}$ such that $\overline{S}$ is maximal, then $(M_{c_1}+\overline{U})=(M_{c_2}+\overline{U})=V_{0,m,j}$ since $\overline{U}$ is transversal to $M_{c_1}$ and $M_{c_2}$. Therefore, there are two stratifications for $V_{0,m,j}/\overline{U}$ with $i=1,2$ as in (\ref{eq:mcqstrat}). The claimed stratification is the largest common refinement of these two. In particular, there are three types of strata. First, if 
$U \in ((J(V_{0,k(n-2),c_1},V_{k_1,l_1,c_{2}}) \cap V_{0,m,j})\setminus M_{c_1}^0+\overline{U}) \cap  ((J(V_{0,k(n-2),c_2},V_{k_2,l_2,c_{1}}) \cap V_{0,m,j})\setminus M_{c_2}^0+\overline{U}) /\overline{U}$
for arbitrary $(k_1,l_1) \in S_1(\overline{U})$ and $(k_2,l_2) \in S_2(\overline{U})$, then
$((U) \setminus (\overline{U})) \cap P_{m+1,c_i} \dashv V_{k_i,l_i,c_i}$ for $i=1,2$.
Second, if
$U \in  (((J(V_{0,k(n-2),c_1},V_{k_i,l_i,c_{3-i}}) \cap V_{0,m,j})+\overline{U})  \setminus \bigcup_{(k_{3-i},l_{3-i}) \in  S_{3-i}(\overline{U})}  (J(V_{0,k(n-2),c_2},V_{k_{3-i},l_{3-i},c_{i}}) \cap V_{0,m,j})+\overline{U}) /\overline{U}$,
then $((U) \setminus (\overline{U})) \cap P_{m+1,c_i} \dashv V_{k_i,l_i,c_i}$ but $((U) \setminus (\overline{U})) \cap P_{m+1,c_{3-i}} = \emptyset$.
Third, if
$U \in ((M_{c_1}^0+\overline{U}) \cap (M_{c_2}^0+\overline{U})) /\overline{U}$,
then  $((U) \setminus (\overline{U})) \cap P_{m+1,c_{i}} = \emptyset$ for $i=1,2$.
To sum it up, we get a stratification into locally closed spaces
\[
V_{0,m,j}/\overline{U} = \bigsqcup_{\substack{(k_1,l_1) \in S_1(\overline{U})\cup\{\emptyset\} \\ (k_2,l_2) \in S_2(\overline{U})\cup\{\emptyset\}}} V_{0,m,j}(k_1,l_1,k_2,l_2)/\overline{U}.
\]
The Euler characteristic of the stratum $V_{0,m,j}(\emptyset,\emptyset)/\overline{U}=((M_{c_1} + \overline{U}) \cap (M_{c_2}+\overline{U}))/\overline{U}$ is 1. It is left to the reader that the others have Euler characteristic 0.
\end{proof}

Let $\mathcal{I}_S=\{ (S_1(\overline{U}),S_2(\overline{U}))\; :\; \overline{U} \in V_{\overline{S}}\}$. Actually $\mathcal{I}_S$ only depends on $\overline{S}$. For each $(S'_1,S'_2) \in \mathcal{I}_S$, let 
\[V_{\overline{S}}(S'_1,S'_2)=\{\overline{U} \in V_{\overline{S}} \;:\;  (S_1(\overline{U}),S_2(\overline{U}))=(S'_1,S'_2)\}.\]

\begin{corollary} 
For a fixed $(S'_1,S'_2) \in \mathcal{I}(S)$ and $(k_1,l_1,k_2,l_2) \in I(S,S'_1,S'_2)$ the spaces $V_{0,m,j}(k_1,l_1,k_2,l_2)/\overline{U}$ are isomorphic for every $\overline{U} \in V_{\overline{S}}(S'_1,S'_2)$. Moreover, they fit together into a locally closed subvariety $V_S(k_1,l_1,k_2,l_2) \subseteq V_S$ which is a trivial family over $V_{\overline{S}}(S'_1,S'_2)$. Using induction and the fact that the fiber product of locally closed spaces is locally closed, we get that there is a stratification 
\[V_{S}=\bigsqcup_{(S'_1,S'_2) \in \mathcal{I}_S} \left(\bigsqcup_{(k_1,l_1,k_2,l_2) \in I(S,S'_1,S'_2)} V_S(k_1,l_1,k_2,l_2)\right) \]
into locally closed subvarieties.
Furthermore, if $(k_1,l_1,k_2,l_2) \in I(S,S'_1,S'_2)_e$, then the fiber of $\omega \colon V_S(k_1,l_1,k_2,l_2) \to V_{\overline{S}}$ has Euler characteristic $e$.
\end{corollary}
\begin{proof}
The triviality of the family $V_S(k_1,l_1,k_2,l_2) \to V_{\overline{S}}(S'_1,S'_2)$ follows from Lemma \ref{lem:joinbasechange} and the fact that $V_{0,m,j}(k_1,l_1,k_2,l_2)/\overline{U}$ is constructed using (union, intersection and difference of) joins in $P_{m,j}$. The rest of the statement is obvious.
\end{proof}

\subsection{Proof of Theorem \ref{thm:zinfty}}\label{proofofethrm}

As before, we fix $S$, $S_1$ and $S_2$. Recall that the fiber of the morphism $\omega \times \mathrm{Id} \times \mathrm{Id}\colon X_S^{S_1,S_2} \to Y_{\overline{S}}^{S_1,S_2}$ over an element $(\overline{U},U_1,U_2) \in Y_{\overline{S}}^{S_1,S_2}$ is $J(L_1^{-1}U_1,L_1^{-1}U_2)/{\overline{U}}$.

For $i=1,2$ let $S_i(\overline{U})=\{ (k_i,l_i) \in S_i\; :\; (\overline{U}) \cap V_{k_i,l_i,c_i}=\emptyset\}$ be the blocks in the partial profile of $(\overline{U}) \cap P_{m+1,c_i}$ on the $(m+1)$-st diagonal. Then the index sets $I(S,S_1(\overline{U}),S_2(\overline{U}))$ and $I(S,S_1(\overline{U}),S_2(\overline{U}))_e$ are defined. The following lemma, whose proof is the same as that of Lemma \ref{lem:fiberstrat}, stratifies the fibers of the affine linear fibration $\omega \times \mathrm{Id} \times \mathrm{Id}\colon X_S^{S_1,S_2} \to Y_{\overline{S}}^{S_1,S_2}$.

\begin{lemma}
For any $U_1 \in V_{S_1,c_1}, U_2 \in V_{S_2,c_2}$ the stratification of Lemma \ref{lem:fiberstrat} restricts to a stratification
\[ J(L_1^{-1}U_1,L_1^{-1}U_2)/{\overline{U}}=\bigsqcup_{(k_1,l_1,k_2,l_2) \in I(S,S_1(\overline{U}),S_2(\overline{U}))}  J(L_1^{-1}U_1,L_1^{-1}U_2)(k_1,l_1,k_2,l_2)/\overline{U},\]
where
\begin{gather*} J(L_1^{-1}U_1,L_1^{-1}U_2)(k_1,l_1,k_2,l_2)/\overline{U}=\\ \{ U \in  J(L_1^{-1}U_1,L_1^{-1}U_2)/\overline{U}\;:\; ((U) \setminus (\overline{U})) \cap P_{m+1,c_i} \dashv V_{k_i,l_i,c_i} \textrm{ for } i=1,2\}. \end{gather*}
Moreover, if $(k_1,l_1,k_2,l_2) \in I(S,S_1(\overline{U}),S_2(\overline{U}))_e$, then the space $ J(L_1^{-1}U_1,L_1^{-1}U_2)(k_1,l_1,k_2,l_2)/\overline{U}$ is of Euler characterestic $e$.
\end{lemma}

Let $\mathcal{I}_S^{S_1,S_2}=\{ (S_1(\overline{U}),S_2(\overline{U}))\;:\; (\overline{U},U_1,U_2) \in Y_{\overline{S}}^{S_1,S_2}\}$. Actually $\mathcal{I}_S^{S_1,S_2}$ only depends on $\overline{S}$, $S_1$ and $S_2$. For each $(S'_1,S'_2) \in \mathcal{I}_S^{S_1,S_2}$, let 
\[Y_{\overline{S}}^{S_1,S_2}(S'_1,S'_2)=\{(\overline{U},U_1,U_2) \in Y_{\overline{S}}^{S_1,S_2} \;:\;  (S_1(\overline{U}),S_2(\overline{U}))=(S'_1,S'_2)\}.\]

\begin{corollary}
\label{cor:xinfty}
For fixed $(k_1,l_1,k_2,l_2) \in I(S,S'_1,S'_2)$ the spaces $J(L_1^{-1}U_1,L_1^{-1}U_2)(k_1,l_1,k_2,l_2)/\overline{U}$ are isomorphic for every $\overline{U} \in Y_{\overline{S}}^{S_1,S_2}(S'_1,S'_2)$. Moreover, they fit together into a locally closed subvariety $X_S^{S_1,S_2}(k_1,l_1,k_2,l_2)\subseteq X_S^{S_1,S_2}$. Using induction and the fact that the fiber product of locally closed spaces is locally closed, we get that there is a stratification 
\[X_S^{S_1,S_2}=\bigsqcup_{(S'_1,S'_2) \in \mathcal{I}_S^{S_1,S_2}} \left(\bigsqcup_{(k_1,l_1,k_2,l_2) \in I(S,S'_1,S'_2)} X_S^{S_1,S_2}(k_1,l_1,k_2,l_2)\right) \]
into a locally closed subvarieties.
Furthermore, if $(k_1,l_1,k_2,l_2) \in I(S,S'_1,S'_2)_e$, then the fiber of $\omega \times \mathrm{Id} \times \mathrm{Id}\colon X_S^{S_1,S_2}(k_1,l_1,k_2,l_2) \to Y_{\overline{S}}^{S_1,S_2}(S'_1,S'_2)$ has Euler characteristic $e$.
\end{corollary}

\begin{proof}[Proof of Theorem \ref{thm:zinfty}]
With all these preparations the proof itself is very easy. We just observe that $Z_{Y}(k_1,l_1,k_2,l_2)$ consist of those points in $Z_Y$, which map in (\ref{eq:prfibeq}) to $X_S^{S_1,S_2}(k_1,l_1,k_2,l_2)$ for some $(S'_1,S'_2) \in \mathcal{I}_S^{S_1,S_2}$ such that $(k_1,l_1,k_2,l_2) \in I(S,S'_1,S'_2)$ . The result then follows from Corollary \ref{cor:xinfty}.
\end{proof}

\section{Type \texorpdfstring{$D_n$}{Dn}: decomposition of the coarse Hilbert scheme}
\label{sect:coarse}

\subsection{Distinguished 0-generated Young walls}
\label{subsect:dis0gen}

In this section, we describe some distinguished subsets of the set of Young walls $\mathcal{Z}_\Delta$ of type $D_n$. 
They will consist of Young walls which are the analogues of the $0$-generated partitions from~\ref{subsectypeA}; as
always, in the type $D$ case there are substantial extra complications. For a Young wall $Y\in\mathcal{Z}_\Delta$, denote
by ${\rm wt}_0(Y)$ the $0$-weight of $Y$, the number of half-blocks labelled~$0$ in~$Y$. 

Recall from~\ref{sec:proof:orbicells}, respectively~\ref{subsec:support} the notions of a salient block and a support block for a given label $c \in \{0,1,n-1,n\}$; we will use also all other notation introduced in the latter section. 
We call a pair of salient half-blocks $(b,b')$ sharing the same position a {\em salient block-pair}.

Consider the following conditions for a Young wall $Y\in\mathcal{Z}_\Delta$. 
\begin{enumerate}
\item[(A1)] All salient blocks of $Y$ are labelled $0$, $1$, $n-1$ or $n$.
\item[(A2)] Every salient block of $Y$ labelled $c \in\{1,n-1,n\}$ is supported.
\end{enumerate}
Let $\mathcal{Z}_{\Delta}' \subset \mathcal{Z}_\Delta$ be the set of
Young walls $Y$ which satisfy conditions (A1)-(A2).  We will prove in
\autoref{thm:assdiag} that $\mathcal{Z}_{\Delta}'$ is the set of
Young walls $Y\in\mathcal{Z}_\Delta$ for which $Z_Y \cap \mathrm{im}(i^{\ast}) \neq \emptyset$,
where $i^{\ast}$ is the pullback map defined in~\ref{sec_notation}. 

Given a Young wall $Y\in\mathcal{Z}_\Delta$, a \emph{closing datum} for $Y$ is a function $d$ from the set of the salient blocks of $Y$ of label $c \in\{1,n-1,n\}$, and some {\em subset} of the salient blocks of $Y$ with label $0$, to the set of support blocks of $Y$, such that 
\begin{itemize}
\item for each salient block $b$ of label $c$ for which $d$ is defined, the associated support block $d(b)$ is a support block for label $c$, and lies on the previous antidiagonal and in a lower row than that of $b$;
\item for each fixed $c \in \{0,1,n-1,n\}$ the different salient blocks of label $c$ are mapped to different support blocks;
\item each support block for label $c$ can support at most one salient block of label $c$ according to the supporting rules spelled out at the end of \ref{subsec:support}.
\end{itemize}
 By condition (A2), for every $Y\in \mathcal{Z}_{\Delta}'$ with a nonempty set of salient blocks the set $\mathrm{cd(Y)}$ of closing data for~$Y$ is nonempty. If all salient blocks of $Y$ of label $1$, $n-1$ or $n$ are directly supported, then a closing datum $d\in \mathrm{cd(Y)}$ is called \emph{contributing}, if to every salient block of label on which $d$ is defined, it associates the support block immediately below it.

We define two subsets of $\mathcal{Z}_{\Delta}'$. Consider the following conditions for a Young wall $Y\in\mathcal{Z}_\Delta$.
\begin{enumerate}
\item[(R1)] The salient blocks of $Y$ of label $n-1$ or $n$ are all part of a directly supported salient block-pair.  
\item[(R2)] $Y$ has no salient block with label in the set $\{1, \dots, n-2\}$.
\item[(R3)] Any consecutive series of rows of $Y$ having equal length and ending in half $n-1$/$n$-blocks is longer than $n-2$, or $n-1$ if the length of the rows is $n-1$, and the last one starts with a block labelled $1$ (see Example~\ref{ex:singr2} below for the latter condition being broken).
\end{enumerate}
Young walls satisfying (R1)--(R2) will be called \emph{$0$-generated}. Let $\mathcal{Z}_{\Delta}^1 \subset \mathcal{Z}_\Delta$ denote the set of $0$-generated Young walls. They automatically satisfy (A1)--(A2), so indeed $\mathcal{Z}_{\Delta}^1 \subset \mathcal{Z}_\Delta'$. Let further $\mathcal{Z}_{\Delta}^0 \subset \mathcal{Z}_\Delta^1$ be the set of those Young walls which in addition satisfy (R3) also. These will be called \emph{distinguished}.

\begin{lemma} 
\label{lem:yprimeunique} 
Let $Y \in \mathcal{Z}_\Delta$ be an arbitrary Young wall.
\begin{enumerate}
\item The positions of the $0$ blocks of $Y$ determine uniquely a Young wall $Y_0 \in \mathcal{Z}_\Delta^0$
such that ${\rm wt}_0(Y)={\rm wt}_0(Y_0)$, and the $0$ blocks in $Y_0$
are exactly the $0$ blocks in $Y$. $Y_0$ does not necessarily contain
$Y$.
\item There is a unique $Y_1\in\mathcal{Z}^1_\Delta$ that contains $Y$ and is minimal with this property with respect to containment.
\end{enumerate}

\end{lemma}
We will prove Lemma \ref{lem:yprimeunique} at the end of the section. 

\begin{lemma}
\label{lem:1red}
There is a combinatorial reduction map ${\rm red}\colon \mathcal{Z}_\Delta'\to \mathcal{Z}_\Delta^0$, restricting to the identity on
$\mathcal{Z}_\Delta^0\subset \mathcal{Z}_\Delta'$, which associates to 
each $Y \in \mathcal{Z}_\Delta'$ a distinguished $0$-generated 
Young wall ${\rm red}(Y) \in \mathcal{Z}_\Delta^0$ of the same $0$-weight.
\end{lemma}
\begin{proof}
Starting with a Young wall $Y\in \mathcal{Z}_\Delta'$, we construct ${\rm red}(Y)$ by enforcing (R1)--(R2) and (R3) in turn, making sure in the second step that (R1)--(R2) remain fulfilled. 

First, if a Young wall $Y\in \mathcal{Z}_\Delta'$ violates (R1) or (R2) at a salient block of label $1$, $n-1$ or $n$, find the lowest row where this happens, and extend $Y$ by adding as many extra blocks as possible to this row without modifying the $0$-weight. Thus, the extension stops either just before any full block below which there is a missing full block, or just before the next $0$ block in the row, whichever comes earlier. Then in the row above this, one or two blocks may become salient. If at least one of these new salient blocks is not of label $0$, then we repeat the same procedure. Following this procedure all the way to the top of $Y$ gives a new Young wall which satisfies (R1) and (R2). These moves may increase the number of places where the Young wall violates (R3).

Second, assume a Young wall $Y$ satisfies (R1) and (R2) but violates (R3): there is a consecutive series of rows having equal length and ending in half $n-1$/$n$-blocks, but the length of this series is $m \leq n-2$. Remove the half block from the end of the lowest row of such a series. Then a new supported salient block-pair appears.
If the block $b$ immediately above this block-pair is contained in $Y$, then we remove $b$, as well as the blocks to the right of it in order to obtain a valid Young wall. Any full block above $b$ also cannot be present in a valid Young wall, so we remove that as well, together with all of the blocks to the right of it. 
Continue the removal process until there is a full block above the
last removed block. This process terminates after $m$ steps, when it
arrives at a row which is already short enough. In this way we
decreased the number places where the Young wall violates (R3), but
haven't increased the number of places where it violates (R1) or
(R2). The $0$-weight of the Young wall remains unchanged, since the length of the series was at most $n-2$. 

Combining these steps, we obtain a Young wall ${\rm red}(Y)$ that satisfies (R3) as well as (R1) and (R2), and so lies in $\mathcal{Z}_\Delta^0$, and has the same $0$-weight. 
\end{proof}

\begin{example}
\label{ex:singr1}

Let $n=4$ and let us consider the following six Young walls.
\begin{center}
\begin{tabular}{c c}
\begin{tabular}[x]{@{}c@{}}
\begin{tikzpicture}[scale=0.6, font=\footnotesize, fill=black!20]
 \draw (1, 0) -- (9,0);
  \foreach \x in {1,2,3,4,5,6}
    {
      \draw (\x, 0) -- (\x,\x);
    }
  
   \foreach \y in {1,2,3,4,5}
       {
            \draw (\y,\y) -- (10,\y);
       }
       
\draw (6,6)--(8,6);
\draw (7, 0) -- (7,6);
\draw (8, 0) -- (8,6);
\draw (9, 0) -- (9,5);
 \draw (10, 1) -- (10,5);

\foreach \x in {0,1,2,3,4,5}
        {
        	\draw (\x+2.5, \x+0.5) node {2};
        }
  \foreach \x in {0,1,2,3,4}
       {
           \draw (\x+4.5, \x+0.5) node {2};
          }  
    \foreach \x in {0,1,2,3,4}
         {
             \draw (\x+6.5, \x+0.5) node {2};
            } 
        \foreach \x in {0,1}
             {
                 \draw (\x+8.5, \x+0.5) node {2};
                } 
     
     \foreach \x in {0,2,4}
         {
                	\draw (\x+1.5, \x+0.5) node {0};
                	\draw (\x+2.5, \x+1.5) node {1};
                }
 \foreach \x in {0,2,4}
           {
           \draw (\x+3.75,\x+0.75) node {4};
                        	\draw (\x+3.25,\x+0.25) node {3};
                        	\draw (\x+3,\x+1) -- (\x+4,\x+0);
           }
  \foreach \x in {1,3}
             {
             \draw (\x+3.75,\x+0.75) node {3};
                                     	\draw (\x+3.25,\x+0.25) node {4};
                                     	\draw (\x+3,\x+1) -- (\x+4,\x+0);
             }
       \foreach \x in {0,2,4}
                 {
                 \draw (\x+5.75,\x+0.75) node {0};
                              	\draw (\x+5.25,\x+0.25) node {1};
                              	\draw (\x+5,\x+1) -- (\x+6,\x+0);
                 }
        \foreach \x in {1,3}
                   {
                   \draw (\x+5.75,\x+0.75) node {1};
                                           	\draw (\x+5.25,\x+0.25) node {0};
                                           	\draw (\x+5,\x+1) -- (\x+6,\x+0);
                   }       
       \foreach \x in {0,2}
                 {
                 \draw (\x+7.75,\x+0.75) node {4};
                              	\draw (\x+7.25,\x+0.25) node {3};
                              	\draw (\x+7,\x+1) -- (\x+8,\x+0);
                 }
        \foreach \x in {1}
                   {
                   \draw (\x+7.75,\x+0.75) node {3};
                                           	\draw (\x+7.25,\x+0.25) node {4};
                                           	\draw (\x+7,\x+1) -- (\x+8,\x+0);
                   }
                                          	\draw (8.25,5.25) node {4};
                                          	\draw (8,6) -- (9,5); 
     	\draw (10.75,3.75) node {3};
     	\draw (11.75,4.75) node {4};
        \draw (10,4) -- (11,3)--(11,4);
        \draw (10,4) -- (11,4)--(11,5)--(10,5);
        \draw (11,5) -- (12,4)--(12,5)--(11,5); 
     
     \draw (9,1)--(10,0)--(9,0);
     \draw (9.25,0.25) node {1};
     \node (2) at ( 8.75,5.75) [circle,draw, fill=black, inner sep=0pt,minimum size=4pt] {};
\end{tikzpicture} 
\\
$Y_1$ 
\end{tabular}
&
\begin{tabular}[x]{@{}c@{}}
\begin{tikzpicture}[scale=0.6, font=\footnotesize, fill=black!20]
 \draw (1, 0) -- (9,0);
  \foreach \x in {1,2,3,4,5,6}
    {
      \draw (\x, 0) -- (\x,\x);
    }
  
   \foreach \y in {1,2,3,4,5}
       {
            \draw (\y,\y) -- (10,\y);
       }
       
\draw (6,6)--(10,6);
\draw (7, 0) -- (7,6);
\draw (8, 0) -- (8,6);
\draw (9, 0) -- (9,6);
 \draw (10, 1) -- (10,6);

\foreach \x in {0,1,2,3,4,5}
        {
        	\draw (\x+2.5, \x+0.5) node {2};
        }
  \foreach \x in {0,1,2,3,4,5}
       {
           \draw (\x+4.5, \x+0.5) node {2};
          }  
    \foreach \x in {0,1,2,3,4}
         {
             \draw (\x+6.5, \x+0.5) node {2};
            } 
        \foreach \x in {0,1}
             {
                 \draw (\x+8.5, \x+0.5) node {2};
                } 
     
     \foreach \x in {0,2,4}
         {
                	\draw (\x+1.5, \x+0.5) node {0};
                	\draw (\x+2.5, \x+1.5) node {1};
                }
 \foreach \x in {0,2,4}
           {
           \draw (\x+3.75,\x+0.75) node {4};
                        	\draw (\x+3.25,\x+0.25) node {3};
                        	\draw (\x+3,\x+1) -- (\x+4,\x+0);
           }
  \foreach \x in {1,3,5}
             {
             \draw (\x+3.75,\x+0.75) node {3};
                                     	\draw (\x+3.25,\x+0.25) node {4};
                                     	\draw (\x+3,\x+1) -- (\x+4,\x+0);
             }
       \foreach \x in {0,2,4}
                 {
                 \draw (\x+5.75,\x+0.75) node {0};
                              	\draw (\x+5.25,\x+0.25) node {1};
                              	\draw (\x+5,\x+1) -- (\x+6,\x+0);
                 }
        \foreach \x in {1,3}
                   {
                   \draw (\x+5.75,\x+0.75) node {1};
                                           	\draw (\x+5.25,\x+0.25) node {0};
                                           	\draw (\x+5,\x+1) -- (\x+6,\x+0);
                   }       
       \foreach \x in {0,2}
                 {
                 \draw (\x+7.75,\x+0.75) node {4};
                              	\draw (\x+7.25,\x+0.25) node {3};
                              	\draw (\x+7,\x+1) -- (\x+8,\x+0);
                 }
        \foreach \x in {1}
                   {
                   \draw (\x+7.75,\x+0.75) node {3};
                                           	\draw (\x+7.25,\x+0.25) node {4};
                                           	\draw (\x+7,\x+1) -- (\x+8,\x+0);
                   }
     	\draw (10.75,3.75) node {3};
     	\draw (11.75,4.75) node {4};
        \draw (10,4) -- (11,3)--(11,4);
        \draw (10,4) -- (11,4)--(11,5)--(10,5);
        \draw (11,5) -- (12,4)--(12,5)--(11,5); 
        
        \draw (10,6) -- (11,5)--(11,6)--(10,6);
        \draw (10.75,5.75) node {1};
     
     \draw (9,1)--(10,0)--(9,0);
     \draw (9.25,0.25) node {1};
     \node (2) at ( 12.75,5.75) [circle,draw, fill=black, inner sep=0pt,minimum size=4pt] {};
\end{tikzpicture} 
\\
$Y_2$
\end{tabular}
\\[2.35cm]
\begin{tabular}[x]{@{}c@{}}
\begin{tikzpicture}[scale=0.6, font=\footnotesize, fill=black!20]
 \draw (1, 0) -- (9,0);
  \foreach \x in {1,2,3,4,5,6}
    {
      \draw (\x, 0) -- (\x,\x);
    }
  
   \foreach \y in {1,2,3,4,5}
       {
            \draw (\y,\y) -- (10,\y);
       }
       
\draw (6,6)--(8,6);
\draw (7, 0) -- (7,6);
\draw (8, 0) -- (8,6);
\draw (9, 0) -- (9,5);
 \draw (10, 1) -- (10,5);

\foreach \x in {0,1,2,3,4,5}
        {
        	\draw (\x+2.5, \x+0.5) node {2};
        }
  \foreach \x in {0,1,2,3,4}
       {
           \draw (\x+4.5, \x+0.5) node {2};
          }  
    \foreach \x in {0,1,2,3,4}
         {
             \draw (\x+6.5, \x+0.5) node {2};
            } 
        \foreach \x in {0,1}
             {
                 \draw (\x+8.5, \x+0.5) node {2};
                } 
     
     \foreach \x in {0,2,4}
         {
                	\draw (\x+1.5, \x+0.5) node {0};
                	\draw (\x+2.5, \x+1.5) node {1};
                }
 \foreach \x in {0,2,4}
           {
           \draw (\x+3.75,\x+0.75) node {4};
                        	\draw (\x+3.25,\x+0.25) node {3};
                        	\draw (\x+3,\x+1) -- (\x+4,\x+0);
           }
  \foreach \x in {1,3}
             {
             \draw (\x+3.75,\x+0.75) node {3};
                                     	\draw (\x+3.25,\x+0.25) node {4};
                                     	\draw (\x+3,\x+1) -- (\x+4,\x+0);
             }
       \foreach \x in {0,2,4}
                 {
                 \draw (\x+5.75,\x+0.75) node {0};
                              	\draw (\x+5.25,\x+0.25) node {1};
                              	\draw (\x+5,\x+1) -- (\x+6,\x+0);
                 }
        \foreach \x in {1,3}
                   {
                   \draw (\x+5.75,\x+0.75) node {1};
                                           	\draw (\x+5.25,\x+0.25) node {0};
                                           	\draw (\x+5,\x+1) -- (\x+6,\x+0);
                   }       
       \foreach \x in {0,2}
                 {
                 \draw (\x+7.75,\x+0.75) node {4};
                              	\draw (\x+7.25,\x+0.25) node {3};
                              	\draw (\x+7,\x+1) -- (\x+8,\x+0);
                 }
        \foreach \x in {1}
                   {
                   \draw (\x+7.75,\x+0.75) node {3};
                                           	\draw (\x+7.25,\x+0.25) node {4};
                                           	\draw (\x+7,\x+1) -- (\x+8,\x+0);
                   }
        \draw (8.75,5.75) node {3};
        \draw (8,6) -- (9,6) -- (9,5)--(8,6); 
     	
     	\draw (10.25,3.25) node {4};
     	\draw (11.25,4.25) node {3};
        \draw (10,4) -- (11,3)--(10,3);
        \draw (10,4) -- (11,4)--(11,5)--(10,5);
        \draw (11,5) -- (12,4)--(11,4)--(11,5); 
     
     \draw (9,1)--(10,0)--(9,0);
     \draw (9.25,0.25) node {1};
     \node (2) at ( 8.25,5.25) [circle,draw, fill=black, inner sep=0pt,minimum size=4pt] {};
\end{tikzpicture} 
\\
$Y_3$ 
\end{tabular}
&
\begin{tabular}[x]{@{}c@{}}
\begin{tikzpicture}[scale=0.6, font=\footnotesize, fill=black!20]
 \draw (1, 0) -- (9,0);
  \foreach \x in {1,2,3,4,5,6}
    {
      \draw (\x, 0) -- (\x,\x);
    }
  
   \foreach \y in {1,2,3,4,5}
       {
            \draw (\y,\y) -- (10,\y);
       }
       
\draw (6,6)--(10,6);
\draw (7, 0) -- (7,6);
\draw (8, 0) -- (8,6);
\draw (9, 0) -- (9,6);
 \draw (10, 1) -- (10,6);

\foreach \x in {0,1,2,3,4,5}
        {
        	\draw (\x+2.5, \x+0.5) node {2};
        }
  \foreach \x in {0,1,2,3,4,5}
       {
           \draw (\x+4.5, \x+0.5) node {2};
          }  
    \foreach \x in {0,1,2,3,4}
         {
             \draw (\x+6.5, \x+0.5) node {2};
            } 
        \foreach \x in {0,1}
             {
                 \draw (\x+8.5, \x+0.5) node {2};
                } 
     
     \foreach \x in {0,2,4}
         {
                	\draw (\x+1.5, \x+0.5) node {0};
                	\draw (\x+2.5, \x+1.5) node {1};
                }
 \foreach \x in {0,2,4}
           {
           \draw (\x+3.75,\x+0.75) node {4};
                        	\draw (\x+3.25,\x+0.25) node {3};
                        	\draw (\x+3,\x+1) -- (\x+4,\x+0);
           }
  \foreach \x in {1,3,5}
             {
             \draw (\x+3.75,\x+0.75) node {3};
                                     	\draw (\x+3.25,\x+0.25) node {4};
                                     	\draw (\x+3,\x+1) -- (\x+4,\x+0);
             }
       \foreach \x in {0,2,4}
                 {
                 \draw (\x+5.75,\x+0.75) node {0};
                              	\draw (\x+5.25,\x+0.25) node {1};
                              	\draw (\x+5,\x+1) -- (\x+6,\x+0);
                 }
        \foreach \x in {1,3}
                   {
                   \draw (\x+5.75,\x+0.75) node {1};
                                           	\draw (\x+5.25,\x+0.25) node {0};
                                           	\draw (\x+5,\x+1) -- (\x+6,\x+0);
                   }       
       \foreach \x in {0,2}
                 {
                 \draw (\x+7.75,\x+0.75) node {4};
                              	\draw (\x+7.25,\x+0.25) node {3};
                              	\draw (\x+7,\x+1) -- (\x+8,\x+0);
                 }
        \foreach \x in {1}
                   {
                   \draw (\x+7.75,\x+0.75) node {3};
                                           	\draw (\x+7.25,\x+0.25) node {4};
                                           	\draw (\x+7,\x+1) -- (\x+8,\x+0);
                   }
     	
     	\draw (10.25,3.25) node {4};
     	\draw (11.25,4.25) node {3};
        \draw (10,4) -- (11,3)--(10,3);
        \draw (10,4) -- (11,4)--(11,5)--(10,5);
        \draw (11,5) -- (12,4)--(11,4)--(11,5); 
        
        \draw (10,6) -- (11,5)--(11,6)--(10,6);
        \draw (10.75,5.75) node {1};
     
     \draw (9,1)--(10,0)--(9,0);
     \draw (9.25,0.25) node {1};
    \node (2) at ( 12.25,5.25) [circle,draw, fill=black, inner sep=0pt,minimum size=4pt] {};         
\end{tikzpicture} 
\\
$Y_4$
\end{tabular}
\\[2.35cm]
\begin{tabular}[x]{@{}c@{}}
\begin{tikzpicture}[scale=0.6, font=\footnotesize, fill=black!20]
 \draw (1, 0) -- (9,0);
  \foreach \x in {1,2,3,4,5,6}
    {
      \draw (\x, 0) -- (\x,\x);
    }
  
   \foreach \y in {1,2,3,4,5}
       {
            \draw (\y,\y) -- (10,\y);
       }
       
\draw (6,6)--(10,6);
\draw (7, 0) -- (7,6);
\draw (8, 0) -- (8,6);
\draw (9, 0) -- (9,6);
 \draw (10, 1) -- (10,6);

\foreach \x in {0,1,2,3,4,5}
        {
        	\draw (\x+2.5, \x+0.5) node {2};
        }
  \foreach \x in {0,1,2,3,4,5}
       {
           \draw (\x+4.5, \x+0.5) node {2};
          }  
    \foreach \x in {0,1,2,3}
         {
             \draw (\x+6.5, \x+0.5) node {2};
            } 
        \foreach \x in {0,1}
             {
                 \draw (\x+8.5, \x+0.5) node {2};
                } 
     
     \foreach \x in {0,2,4}
         {
                	\draw (\x+1.5, \x+0.5) node {0};
                	\draw (\x+2.5, \x+1.5) node {1};
                }
 \foreach \x in {0,2,4}
           {
           \draw (\x+3.75,\x+0.75) node {4};
                        	\draw (\x+3.25,\x+0.25) node {3};
                        	\draw (\x+3,\x+1) -- (\x+4,\x+0);
           }
  \foreach \x in {1,3,5}
             {
             \draw (\x+3.75,\x+0.75) node {3};
                                     	\draw (\x+3.25,\x+0.25) node {4};
                                     	\draw (\x+3,\x+1) -- (\x+4,\x+0);
             }
       \foreach \x in {0,2,4}
                 {
                 \draw (\x+5.75,\x+0.75) node {0};
                              	\draw (\x+5.25,\x+0.25) node {1};
                              	\draw (\x+5,\x+1) -- (\x+6,\x+0);
                 }
        \foreach \x in {1,3}
                   {
                   \draw (\x+5.75,\x+0.75) node {1};
                                           	\draw (\x+5.25,\x+0.25) node {0};
                                           	\draw (\x+5,\x+1) -- (\x+6,\x+0);
                   }       
       \foreach \x in {0,2}
                 {
                 \draw (\x+7.75,\x+0.75) node {4};
                              	\draw (\x+7.25,\x+0.25) node {3};
                              	\draw (\x+7,\x+1) -- (\x+8,\x+0);
                 }
        \foreach \x in {1}
                   {
                   \draw (\x+7.75,\x+0.75) node {3};
                                           	\draw (\x+7.25,\x+0.25) node {4};
                                           	\draw (\x+7,\x+1) -- (\x+8,\x+0);
                   }
     	
        
     
     \draw (9,1)--(10,0)--(9,0);
     \draw (9.25,0.25) node {1};
     \node (2) at ( 10.75,5.75) [circle,draw, fill=black, inner sep=0pt,minimum size=4pt] {};        
\end{tikzpicture} 
\\
$Y_5$
\end{tabular} &
\begin{tabular}[x]{@{}c@{}}
\begin{tikzpicture}[scale=0.6, font=\footnotesize, fill=black!20]
 \draw (1, 0) -- (9,0);
  \foreach \x in {1,2,3,4,5,6}
    {
      \draw (\x, 0) -- (\x,\x);
    }
  
   \foreach \y in {1,2,3,4,5}
       {
            \draw (\y,\y) -- (10,\y);
       }
       
\draw (6,6)--(10,6);
\draw (7, 0) -- (7,6);
\draw (8, 0) -- (8,6);
\draw (9, 0) -- (9,6);
 \draw (10, 1) -- (10,6);

\foreach \x in {0,1,2,3,4,5}
        {
        	\draw (\x+2.5, \x+0.5) node {2};
        }
  \foreach \x in {0,1,2,3,4,5}
       {
           \draw (\x+4.5, \x+0.5) node {2};
          }  
    \foreach \x in {0,1,2,3}
         {
             \draw (\x+6.5, \x+0.5) node {2};
            } 
        \foreach \x in {0,1}
             {
                 \draw (\x+8.5, \x+0.5) node {2};
                } 
     
     \foreach \x in {0,2,4}
         {
                	\draw (\x+1.5, \x+0.5) node {0};
                	\draw (\x+2.5, \x+1.5) node {1};
                }
 \foreach \x in {0,2,4}
           {
           \draw (\x+3.75,\x+0.75) node {4};
                        	\draw (\x+3.25,\x+0.25) node {3};
                        	\draw (\x+3,\x+1) -- (\x+4,\x+0);
           }
  \foreach \x in {1,3,5}
             {
             \draw (\x+3.75,\x+0.75) node {3};
                                     	\draw (\x+3.25,\x+0.25) node {4};
                                     	\draw (\x+3,\x+1) -- (\x+4,\x+0);
             }
       \foreach \x in {0,2,4}
                 {
                 \draw (\x+5.75,\x+0.75) node {0};
                              	\draw (\x+5.25,\x+0.25) node {1};
                              	\draw (\x+5,\x+1) -- (\x+6,\x+0);
                 }
        \foreach \x in {1,3}
                   {
                   \draw (\x+5.75,\x+0.75) node {1};
                                           	\draw (\x+5.25,\x+0.25) node {0};
                                           	\draw (\x+5,\x+1) -- (\x+6,\x+0);
                   }       
       \foreach \x in {0,2}
                 {
                 \draw (\x+7.75,\x+0.75) node {4};
                              	\draw (\x+7.25,\x+0.25) node {3};
                              	\draw (\x+7,\x+1) -- (\x+8,\x+0);
                 }
        \foreach \x in {1}
                   {
                   \draw (\x+7.75,\x+0.75) node {3};
                                           	\draw (\x+7.25,\x+0.25) node {4};
                                           	\draw (\x+7,\x+1) -- (\x+8,\x+0);
                   }
     	
        
        \draw (10,6) -- (11,5)--(11,6)--(10,6);
        \draw (10.75,5.75) node {1};
     
     \draw (9,1)--(10,0)--(9,0);
     \draw (9.25,0.25) node {1};
             
\end{tikzpicture}
\\
$Y_6$
\end{tabular}
\end{tabular}
\end{center}
It can be checked that all these satisfy conditions (A1) and (A2), and hence lie in $\mathcal{Z}_\Delta'$. On the other hand, the Young walls $Y_1$ and $Y_3$ violate (R1), and are extended to $Y_2$, resp.~$Y_4$ in the first step of the reduction process. Both of these still violate (R3); in the second step, they become ${\rm red}(Y_1)= {\rm red}(Y_3)=Y_6\in \mathcal{Z}_\Delta^0$. $Y_5$ satisfies (R1), but violates (R2) and is extended to $Y_6$ in the first step with ${\rm red}(Y_5)=Y_6\in \mathcal{Z}_\Delta^0$ also. In each case the bullets indicate where exactly these violations occur. The blocks without numbers are not contained in the Young walls.
\end{example}

For a Young wall $Y\in {\mathcal Z}_\Delta^0$, let $\mathrm{Rel}(Y)={\rm red}^{-1}(Y)$
denote the set of relatives of $Y$, the Young walls we can get from $Y$ by the inverses of the reduction steps above. This is a finite directed set, directed by the steps of the proof of Lemma~\ref{lem:1red}. In Example~\ref{ex:singr1}, all $Y_i$ are relatives of $Y_6$. 

\begin{proof}[Proof of Lemma \ref{lem:yprimeunique}]
Fix $Y \in \mathcal{Z}_\Delta$. For part (1), consider the minimal Young wall $Y_m$ that contains the same 0-blocks as $Y$. Similarly to the proof of Lemma \ref{lem:1red} extend each row of $Y_m$ with as many blocks as possible taking into account the Young wall rules and the conditions (R1)--(R3), and without modifying the 0-weight.

For part (2), consider the set $\mathrm{Rel}_Y(Y_0)\subset \mathcal{Z}_\Delta'$
of those relatives of $Y_0$ which contain $Y$. Here $Y_0$ is the Young wall obtained in part (1). 
This set of relatives is nonempty,
since we can always extend $Y_0$ with the inverse of the move (R3) in
Lemma \ref{lem:1red} until there are only label 0 salient
blocks. There can be several of these since there is an ambiguity in
the inverse of the move (R3), but there is no Young wall having the
same 0 weight as $Y_0$ which is not contained in at least one of these
extended Young walls. 

Suppose that $\mathrm{Rel}_Y(Y_0)$ has two distinct minimal elements
$Y_2, Y_3$ with respect to containment. Then there is at least one row ending in
a half block, where one of $Y_2, Y_3$ has a left
triangle, and the other has a right triangle, but otherwise the row
has the same length. Then the length of the
series of successive rows with the same length is the same in the two
Young walls. If this length is more than $n-2$, then they cannot both
contain $Y$. If it is $n-2$ or less, then $Y_2, Y_3$ are the two
results of the inverse of the move (R3) applied on a smaller Young
wall. Since $Y$ was a Young wall, also this smaller Young wall
contains $Y$. Hence, neither of $Y_2, Y_3$ could be minimal. 
The same reasoning applies to all places where there is the left
triangle/right triangle ambiguity.
Thus, there is a unique minimal element $Y_1$ in the set of relatives of $Y_0$ containing $Y$.

Since $Y_1$ can be obtained from $Y_0$ by inverses of the move (R3), it is an element in $\mathcal{Z}^1_\Delta$
\end{proof}

\subsection{The decomposition of the coarse Hilbert scheme}

Let us turn to the Hilbert scheme of points on the quotient $\SC^2/G_\Delta$, the coarse Hilbert scheme 
$\mathrm{Hilb}(\SC^2/G_\Delta)=\sqcup_n \mathrm{Hilb}^n(\SC^2/G_\Delta)$. Recall that the 
inclusion $\SC[x,y]^{G_\Delta}\subset \SC[x,y]$ defines a morphism
\[ p_\ast: \mathrm{Hilb}([\SC^2/G_\Delta])\rightarrow \mathrm{Hilb}(\SC^2/G_\Delta), \quad J \mapsto J^{G_\Delta}=J\cap \SC[x,y]^{G_\Delta}\]
and a map of sets 
\[ i^\ast: \mathrm{Hilb}(\SC^2/G_\Delta)(\SC) \rightarrow \mathrm{Hilb}([\SC^2/G_\Delta])(\SC), \quad I \mapsto \SC[x,y].I \] 
between the coarse and the orbifold Hilbert schemes.

The purpose of this section is to prove the following result. 

\begin{theorem} The decomposition of the equivariant Hilbert scheme $\mathrm{Hilb}([\SC^2/G_\Delta])$ from 
Theorem~\ref{thm:dnorbcells} induces a locally closed decomposition
\[\mathrm{Hilb}(\SC^2/G_\Delta) = \displaystyle\bigsqcup_{Y \in{\mathcal Z}_\Delta'} \mathrm{Hilb}(\SC^2/G_\Delta)_Y\]
of the coarse Hilbert scheme $\mathrm{Hilb}(\SC^2/G_\Delta)$ into strata indexed bijectively by the
set ${\mathcal Z}_\Delta'$ of Young walls of type $D_n$ satisfying (A1)-(A2) above.
The stratum $\mathrm{Hilb}(\SC^2/G_\Delta)_Y$ is contained in the $m$-th Hilbert scheme 
$\mathrm{Hilb}^m(\SC^2/G_\Delta)$ for $m={\rm wt}_0(Y)$.
\label{thm:assdiag}
\end{theorem}
\begin{proof} We start with the universal ideal 
${\mathcal J}\lhd \mathcal{O}_{\mathrm{Hilb}(\SC^2/G_\Delta)}\otimes\SC[x,y]^{G_\Delta}$, which exists since
$\mathrm{Hilb}(\SC^2/G_\Delta)$ is a fine moduli space. Using the relative pullback, 
we obtain an invariant ideal $\SC[x,y].{\mathcal J}\lhd \mathcal{O}_{\mathrm{Hilb}(\SC^2/G_\Delta)}\otimes\SC[x,y]$, which however
is not a {\em flat} family of invariant ideals over $\mathrm{Hilb}(\SC^2/G_\Delta)$. Take the flattening stratification of the base with index set $F$,
to obtain a decomposition 
\begin{equation}
\label{eq:flatstrat}
\mathrm{Hilb}(\SC^2/G_\Delta)=\sqcup_{f \in F}\mathrm{Hilb}(\SC^2/G_\Delta)_f
\end{equation}
over which the restrictions $(\SC[x,y].{\mathcal J})_f$ are flat.
These flat families of invariant ideals of $\SC[x,y]$
define classifying maps
\[i_f\colon \mathrm{Hilb}(\SC^2/G_\Delta)_f\to  \mathrm{Hilb}^\rho([\SC^2/G_\Delta])\]
from these strata to components of the equivariant Hilbert scheme.
The latter smooth varieties are decomposed into 
locally closed strata by Theorem~\ref{thm:dnorbcells} as
\begin{equation}
\label{eq:Ystratf}
\mathrm{Hilb}([\SC^2/G_\Delta])=\sqcup_{Y \in \in{\mathcal Z}_\Delta}\mathrm{Hilb}([\SC^2/G_\Delta])_Y.
\end{equation}
The stratification \eqref{eq:Ystratf} gives a stratification on $\mathrm{im}(i_f)$ for each $f \in F$ since over each $\mathrm{Hilb}([\SC^2/G_\Delta])_Y$ the classifying map is flat. Hence, we can pull it back to obtain a decomposition
\[\mathrm{Hilb}(\SC^2/G_\Delta) = \displaystyle\bigsqcup_{Y \in{\mathcal Z}_\Delta} \mathrm{Hilb}(\SC^2/G_\Delta)_Y,\]
where we have, set-theoretically, 
\[\mathrm{Hilb}(\SC^2/G_\Delta)_Y(\SC) = \{ I \in \mathrm{Hilb}(\SC^2/G_\Delta)(\SC) : i^\ast(I) \in \mathrm{Hilb}([\SC^2/G_\Delta])_Y(\SC)\}.\]
The whole construction is compatible with the $T=\SC^*$-action, so we can also decompose the $T$-fixed locus representing
homogeneous ideals as
\[\mathrm{Hilb}(\SC^2/G_\Delta)^T = \displaystyle\bigsqcup_{Y \in{\mathcal Z}_\Delta} W_Y,
\]
where 
\[ W_Y(\SC)=\{ I \in \mathrm{Hilb}(\SC^2/G_\Delta)(\SC) : i^\ast(I) \in Z_Y(\SC)\}.\]
Notice also that by construction, the maps $i_\rho$ above are given by the pullback map $i^\ast$. In other words, when restricted to 
the strata $\mathrm{Hilb}(\SC^2/G_\Delta)_Y\supset W_Y$, the map $i^\ast$ becomes a morphism of schemes. 
On the other hand, it is also clear that, letting $\widetilde W_Y$ denote the image of $i^\ast$ inside $Z_Y$, 
$p_\ast$ and $i^\ast$ are two-sided inverses and so $W_Y\cong \widetilde W_Y\subset Z_Y$.

To conclude, we need to show that for a fixed $I \in \mathrm{Hilb}(\SC^2/G_\Delta)$, the Young wall~$Y$ associated 
with the pullback ideal $J=i^{\ast}(I)$ necessarily lies in $\mathcal{Z}_\Delta'$. It is clearly enough to assume that~$I$, and so~$J$,
are homogeneous. The ideal $J$, being a pullback, is of course generated by invariant polynomials. 
On the other hand, 
as we have seen during the proof of Theorem \ref{thm:dnorbcells}, a homogeneous ideal is generated by polynomials lying 
in salient blocks. While not all salient blocks necessarily contain new a generator, it is clear that 
salient blocks labelled~$j$ with $2\leq j \leq n-2$ must contain a generator. Since such a generator is not allowed in an invariant ideal, $Y$ must satisfy condition (A1). 

To discuss the other condition (A2), let us return to the inductive proof of Theorem~\ref{thm:dnorbcells}. Corollary \ref{cor:nongen} (2) implies that if there is no generator on a given antidiagonal of~$Y$, then the salient blocks of label $c \in \{1,n,n-1\}$ on this antidiagonal are supported. For an invariant ideal, this condition is required to be satisfied for all salient blocks of label $c \in \{1,n,n-1\}$. This concludes the proof. 
\end{proof}

\subsection{Possibly and almost invariant ideals} 
\label{subsec:possalmosid}

We wish to study the Euler characteristics of the strata of the coarse Hilbert scheme obtained in
Theorem~\ref{thm:assdiag}, using the inductive approach used in~\ref{sec:proof:orbicells}
in our study of the orbifold Hilbert scheme. However, as things stand, the setup does not lend itself well to 
induction based on the removal of the bottom row from a Young wall, since the set of Young walls ${\mathcal Z}_\Delta'$ 
is clearly not closed under the removal of the bottom row. To remedy this, we introduce two auxiliary constructions. From now on, except when noted, every ideal is supposed to be $T$-invariant.

First, call an ideal $I\lhd \SC[x,y]$ {\em possibly invariant}, if it is generated by 
\begin{itemize}
\item polynomials which transform under $G_\Delta$ according to $\rho_0$, $\rho_1$, $\rho_{n-1}$ or $\rho_n$,
\item and at most one $\tau$-invariant pair of polynomials of the same degree, forming a two-dimensional
representation of $G_\Delta$, and not lying in the image of the operator $L_1$.
\end{itemize}
Equivalently, the second condition says that the 
corresponding two-dimensional 
subspace lies in the large stratum of the appropriate projective space parameterizing such subspaces.

Second, a possibly invariant ideal $I$ will be called \emph{almost invariant}, if it is generated by
\begin{itemize}
\item $G_\Delta$-invariant elements,
\item and at most a single polynomial, or pair of polynomials of the same degree, forming a one- or two-dimensional
representation of $G_\Delta$, and not lying in the image of the operator~$L_1$.
\end{itemize}

Let us denote by ${\mathcal Z}_\Delta^P\subset {\mathcal Z}_\Delta$ the subset of all Young walls which are characterized by the following condition:
\begin{enumerate}
\item[{\rm (A1')}] all salient blocks of $Y$ are labelled $0$, $1$, $n-1$ or $n$, except possibly for a single
salient block in the bottom row of a different label.
\end{enumerate}
Moreover, let ${\mathcal Z}_\Delta^A\subset {\mathcal Z}_\Delta^P$ be the set of Young walls which satisfy condition (A1') as well as the following second condition:
\begin{enumerate}
\item[{\rm (A2')}] every salient block of $Y$ labelled $c \in\{1,n-1,n\}$ is supported, except possibly the ones in the bottom row.
\end{enumerate}

The following statement follows immediately from our setup. 
\begin{proposition} \label{prop:almoststrat}
\hspace{2em}
\begin{enumerate}
\item Possibly invariant ideals correspond to points in the strata $Z_Y\subset \mathrm{Hilb}([\SC^2/G_\Delta])$ where $Y\in{\mathcal Z}_\Delta^P$.
\item Points parameterizing almost invariant ideals lie in constructible subsets $\widetilde W_Y\subset Z_Y$ of strata $Z_Y\subset \mathrm{Hilb}([\SC^2/G_\Delta])$ for $Y\in{\mathcal Z}_\Delta^A$. 
\end{enumerate}
\end{proposition}

By definition, we have ${\mathcal Z}_\Delta'\subset {\mathcal Z}_\Delta^A$. For 
$Y\in {\mathcal Z}_\Delta'$, the constructible subset provided by Proposition \ref{prop:almoststrat}(2) is exactly $\widetilde W_Y=i_{\ast} (W_Y)$. Therefore, in the sequel we will denote these strata as $\widetilde W_Y$ for arbitrary $Y \in {\mathcal Z}_\Delta^A$ as well. 
Further, if for $Y\in {\mathcal Z}_\Delta^A$, also $\overline{Y}\in {\mathcal Z}_\Delta^A$, where $\overline{Y}$
is the Young wall obtained by removing from $Y$ the bottom row as in~\ref{sec:proof:orbicells}, 
then the map $T\colon Z_Y\to Z_{\overline{Y}}$ of Lemma~\ref{lemma:inductive_morphism} takes $\widetilde W_Y$ to 
$\widetilde W_{\overline Y}$.

Furthermore, let $\mathcal{Z}^{0,A}_{\Delta}\subset \mathcal{Z}^A_{\Delta}$ be the subset defined by the following conditions:
\begin{enumerate}
\item[{\rm (R1')}] the salient blocks of label $n-1$ or $n$ not in the bottom row are all directly supported; 
\item[{\rm (R2')}] there is no salient block with label in $\{1, \dots, n-2\}$ except possibly for the bottom row; 
\item[{\rm (R3')}] any consecutive series of rows, except the one starting in the bottom row, having equal length and ending in half $n-1$/$n$-blocks is longer than $n-2$, or $n-1$ if the length of the rows is $n-1$ and the last one starts with a block labelled $1$.
\end{enumerate}

\begin{lemma}
There is a combinatorial reduction map ${\rm red}\colon \mathcal{Z}_\Delta^P\to \mathcal{Z}_\Delta^{0,A}$ associating to each $Y' \in \mathcal{Z}^P_{\Delta}$ a unique $Y\in \mathcal{Z}^{0,A}_{\Delta}$. 
\end{lemma}
\begin{proof}
The proof of Lemma~\ref{lem:1red} above goes through unchanged in this setting and the reduction process gives a well-defined element in $\mathcal{Z}_\Delta^P$. After the reduction there is no indirectly supported salient block. Hence, each salient block is directly supported, and in particular supported as required by condition (A2'). Therefore, the output of the reduction is an element in $\mathcal{Z}^A_{\Delta}$ with the properties (R1')-(R3').
\end{proof}
Once again, for $Y\in \mathcal{Z}^{0,A}_{\Delta}$ the Young walls $\mathrm{Rel}(Y)={\rm red}^{-1}(Y)$ will be called the 
relatives of $Y$. The following statement is clear from the definitions.
\begin{lemma} 
\label{lem:truncclosed}
The sets $\mathcal{Z}^{P}_{\Delta}$ and $\mathcal{Z}^{0,A}_{\Delta}$ are closed under the operation of bottom row removal.
\end{lemma}

The sets of Young walls introduced so far can be placed in the following commutative diagram:
\[
\begin{array}{ccccc}
\mathcal{Z}_{\Delta}^0  & \hra  & \mathcal{Z}_{\Delta}'  & &  \\[0.2cm]
\hda& &\hda  & &\\[0.2cm]
\mathcal{Z}_{\Delta}^{0,A} & \hra &\mathcal{Z}_{\Delta}^A &\hra &\mathcal{Z}_{\Delta}^P.
\end{array}
\]
The relatives of a Young wall $Y\in \mathcal{Z}_{\Delta}^0$ are the same in $\mathcal{Z}_{\Delta}'$ and in $\mathcal{Z}_{\Delta}^A$, but there may be some new relatives in $\mathcal{Z}_{\Delta}^P$. This will cause no problem; see the discussion above Corollary \ref{cor:relsum1}. Even though we are interested in strata of the Young walls in the upper row, it is easier to work in the lower row because of Lemma \ref{lem:truncclosed}.

The notion of a closing datum generalizes word by word for ideals in the stratum of $Y \in {\mathcal Z}_\Delta^A$. For ideals in the stratum of Young walls in $\mathcal{Z}_\Delta^P$ we have to relax it, since not all salient blocks of label $1$, $n-1$ or $n$ are supported. A \emph{partial closing datum} for a Young wall in $\mathcal{Z}_\Delta^P$ is given by associating to some of its salient blocks of label $c \in\{0, 1,n-1,n\}$ a support block for label $c$ in the previous antidiagonal and below the salient block, in such a way that to each support block for label $c$ at most one salient block of label $c$ is associated. We say that those salient blocks to which there is an associated support block are \emph{closed}. The set of all partial closing data for $Y \in \mathcal{Z}_\Delta^P$ will be denoted as $\mathrm{pcd}(Y)$. Closing data are special partial closing data for Young walls in ${\mathcal Z}_\Delta^A$, in which all salient blocks of label $1$, $n-1$ or $n$ are closed, except possibly one on the bottom row. 

Fix a Young wall $Y \in \mathcal{Z}_\Delta^P$ such that in the bottom row the salient block is a support block for label $c \in\{0,1,n-1,n\}$. Let $\overline{Y}$ be the truncation of $Y$. By Lemma~\ref{lem:truncclosed}, $\overline{Y} \in \mathcal{Z}_\Delta^P$. If $\overline{d}$ is a partial closing datum associated to some $\overline{I} \in Z_{\overline{Y}}$, then using the decomposition of Theorem \ref{thm:zinfty} we can extend it to a partial closing datum for each ideal in the fiber over $\overline{I}$; in particular, if $I \in Z_Y(k_1,l_1,k_2,l_2)$ for some pairs $(k_1,l_1)$ and $(k_2,l_2)$, and either of these is a salient block of label $c$, then we associate to them the support block in the bottom row. By induction, we obtain
\begin{corollary}
Given $Y \in \mathcal{Z}_\Delta^P$, every $I \in Z_{Y}$ defines a unique partial closing datum $d(I)\in\mathrm{pcd}(Y)$.
\end{corollary}

For $d \in \mathrm{pcd}(Y)$, let $Z_{Y}(d) \subseteq Z_{Y}$ be the subset of those ideals which have partial closing datum~$d$. Then
\[  Z_{Y}=\sqcup_{d \in \mathrm{pcd}(Y)} Z_{Y}(d) \]
is a locally closed decomposition of the stratum $Z_{Y}$. Similarly, for an element $Y \in \mathcal{Z}_\Delta^A$ let $\widetilde{W}_{Y}(d) \subseteq \widetilde{W}_{Y}$ be the subset of those ideals which have closing datum $d$. Then
\[  \widetilde{W}_{Y}=\sqcup_{d \in \mathrm{cd}(Y)} \widetilde{W}_{Y}(d) \]
is a locally closed decomposition.

\subsection{Euler characteristics of strata and the coarse generating series}

In this section, we derive information about the topological Euler characteristics of the strata of the coarse Hilbert scheme 
constructed above. 

Fix a Young wall $Y \in \mathcal{Z}_\Delta^P$, and a partial closing datum $d \in \mathrm{pcd}(Y)$. We say that a support block for label $c$ is \emph{of type $e \in \{-1,0,1\}$} if, when its row is considered as the bottom row, the associated half blocks according to $d$ are in the set $I(S_1,S_2,S)_e$ in the notation of Theorem \ref{thm:zinfty}. 

\begin{lemma} 
\label{lem:chiclosefiber}
Assume that $(Y,d)$ are such such that in the bottom row of $Y$, the salient block $b$ of label $j \in \{2,n-2\}$ is a support block for label $c \in\{0,1,n-1,n\}$. Let $(\overline{Y},\overline{d})$ be the truncation of $(Y, d)$. 
If the salient block $b$ is of type $e\in \{-1,0,1\}$, then
\[\chi(Z_Y(d))=e\cdot \chi(Z_{\overline{Y}}(\overline{d})).\]
\end{lemma}
\begin{proof}
Using the notations of \ref{subsec:loci:strata}, let $(k_1,l_1,k_2,l_2)$ be the quadruple of the half blocks associated to the support block, and consider the diagram
\[
\begin{array}{ccccc}
 Z_{Y}(d) & \subseteq & Z_{Y} & \ra{\phantom}  & \mathcal{X}_{Y} \\
& & \da{T} & & \da{\omega \times \mathrm{Id}}  \\
Z_{\overline{Y}}(\overline{d}) & \subseteq & Z_{\overline{Y}} &\ra{\phantom} & \mathcal{Y}_{\overline{Y}}.
\end{array}
\]
Returning once again to the process proving Theorem~\ref{thm:dnorbcells}, when we obtained $Z_{Y}$ from the truncation $Z_{\overline{Y}}$, we saw that those ideals in $Z_{Y}$ that does not have a generator in the strata of the missing half blocks at $(k_1,l_1)$ and at $(k_2,l_2)$ are necessarily in $Z_{Y}(k_1,l_1,k_2,l_2)$ and all ideals in $Z_{Y}(k_1,l_1,k_2,l_2)$ have this property. Formally, a point of $Z_{Y}$ over $Z_{\overline{Y}}(\overline{d})$ is in $Z_{Y}(d)$ if and only if
\[((I \cap P_{m,j}) \setminus (I \cap \overline{P}_{m,j})) \cap P_{m+1,c_i} \dashv V_{k_i,l_i,c_i} \textrm{ for } i=1,2.\]
Under the operation $T$ the space $Z_{Y}(d) $ necessarily maps onto $Z_{\overline{Y}}(\overline{d})$. Hence, we get that
\[Z_{Y}(d) = Z_{\overline{Y}}(\overline{d}) \times_{Z_{\overline{Y}}} Z_Y(k_1,l_1,k_2,l_2).\]
By Theorem \ref{thm:zinfty} the fibers of $T$ on $Z_{Y}(k_1,l_1,k_2,l_2)$ have Euler characteristic $e$. Thus
\[\chi(Z_Y(d))=e\cdot\chi(Z_{\overline{Y}}(\overline{d})).\]
\end{proof}

For $e \in \{-1,0,1\}$ and $c \in \{0/1,n-1/n\}$ let $s_{e}(d,c)$ be the number of support blocks for label $c$ of type $e$, and let $s_e(d)=s_e(d,0/1)+s_e(d,n-1/n)$.
Applying Lemma \ref{lem:chiclosefiber} inductively, we get the following. 
\begin{proposition} 
\label{prop:chiclosestratum}
For $Y \in \mathcal{Z}_\Delta^P$ and $d \in \mathrm{pcd}(Y)$,
\[\chi(Z_Y(d))=(-1)^{s_{-1}(d)} \cdot 0^{s_{0}(d)} \cdot 1^{s_{1}(d)},\]
where we adopt the convention $0^0=1$. 
\end{proposition}

\begin{corollary} 
\label{cor:inftychar0} 
\begin{enumerate}
\item Let $Y \in {\mathcal Z}_\Delta^P$ and $d \in \mathrm{pcd}(Y)$. If $Y$ contains a salient block of any label to which a support block not immediately below it is associated under $d$, then $\chi(Z_{Y}(d))=0$.
\item Let $Y \in {\mathcal Z}_\Delta^{A}$ and $d \in \mathrm{cd}(Y)$.
\begin{enumerate}[(a)]
\item If $Y$ contains a nondirectly supported salient block of label 1, $n-1$ or $n$, then $\chi(\widetilde W_{Y}(d))=0$.
\item If $Y$ does not contain any nondirectly supported salient block of label 1, $n-1$ or $n$, but $d$ is not contributing, then $\chi(\widetilde W_{Y}(d))=0$.
\end{enumerate}
\end{enumerate}
\end{corollary}
\begin{proof}
(1) follows from Proposition \ref{prop:chiclosestratum}, and both parts of (2) follows from (1).
\end{proof}

The main ingredient for calculating the coarse generating series is the following statement. 
\begin{proposition} 
\label{prop:relsum1}
For all $Y \in {\mathcal Z}_\Delta^{0,A}$,
\[ \sum_{Y' \in \mathrm{Rel}(Y)} \chi(\widetilde W_{Y'})=1.\]
\end{proposition}
Recall from \ref{subsec:possalmosid} that the relatives of a Young wall $Y\in \mathcal{Z}_{\Delta}^0$ are the same in $\mathcal{Z}_{\Delta}'$ and in $\mathcal{Z}_{\Delta}^A$, but there may be some new relatives in $\mathcal{Z}_{\Delta}^P$ which are not in $\mathcal{Z}_{\Delta}^A$. On the other hand, since $\widetilde W_{Y}=\emptyset$ for $Y \in \mathcal{Z}_{\Delta}^P \setminus \mathcal{Z}_{\Delta}^A$, Proposition \ref{prop:relsum1} implies
\begin{corollary} 
\label{cor:relsum1}
For all $Y \in {\mathcal Z}_\Delta^{0}$,
\[ \sum_{Y' \in \mathrm{Rel}(Y)} \chi(\widetilde W_{Y'})=1.\]
\end{corollary}
\begin{proof}
By Corollary \ref{cor:inftychar0} (2.a), the strata associated to those Young walls $Y' \in \mathrm{Rel}(Y)$ that have at least one undirectly supported salient block of label $n-1$ or $n$ above the bottom row do not contribute to the sum. Therefore we can restrict our attention to the subset of Young walls in which the salient blocks not in the bottom row are
\begin{itemize}
\item directly supported salient block-pairs of label $0/1$ or $n-1/n$, or
\item arbitrary salient blocks of label $0$.
\end{itemize}
In particular, we can assume that all Young walls we are working with satisfy (R1'). Moreover, by Corollary \ref{cor:inftychar0} (2.b) we can assume that each closing datum is contributing. That is, for each salient block of label 1, $n-1$ or $n$ not in the bottom row, the associated support block is immediately below it, and this holds also for those label-0 salient blocks which have an associated support block. 

As in the proof of Theorem~\ref{thm:Zstrata}, we will prove the proposition by induction on the number of rows. Fix a Young wall $Y \in {\mathcal Z}_\Delta^{0,A}$. We re-write the statement into the form
\[ \sum_{Y' \in \mathrm{Rel}(Y)}\sum_{d' \in \mathrm{cd}(Y')} \chi(\widetilde W_{Y'}(d'))=1.\]
Let $\overline{Y}$ be the truncation of $Y$; by Lemma~\ref{lem:truncclosed}, $\overline{Y}\in {\mathcal Z}_\Delta^{0,A}$ also. 

Assume first that the block above the salient block in the bottom row
of $Y$ is not divided. Then the relatives of $Y$ are exactly the
extensions of those of $\overline{Y}$ with the bottom row of $Y$. A
closing datum on any such Young wall extends uniquely to the extended
Young wall, and the corresponding strata are isomorphic. In this case, the induction step is obvious.

Assume next that the block above the salient block in the bottom row of $Y$ is divided. Let $S$ be the index set of the block at $(0,m)$, $S_1$ and $S_2$ the index sets for the blocks at $(1,m)$. The following cases can occur:
\begin{enumerate}
\item $\overline{S}$ is maximal;
\item  $\overline{S}$ is not maximal, and exactly one of $S_1$ or $S_2$ is maximal;
\item  $\overline{S}$ is not maximal, and both $S_1$ and $S_2$ are maximal.
\end{enumerate}

If $\overline{S}$ is maximal, then to each relative $\overline{Y}'\in{\rm Rel}(\overline{Y})$ there corresponds a unique relative $Y'\in{\rm Rel}(Y)$ which satisfies (R1') and (R2'): we extend each relative of $\overline{Y}$ with the bottom row of $Y$. A closing datum $\overline{d}'$ on one of these relatives $\overline{Y}'$ can be extended to a closing datum $d'$ on the corresponding relative of $Y$ by assigning $(\emptyset,\emptyset)$ to the new salient block appearing in the bottom row. By Theorem \ref{thm:zinfty}, $(\emptyset,\emptyset) \in I(S,S_1,S_2)_1$, and
\[ \sum_{Y' \in \mathrm{Rel}(Y)}\sum_{d' \in \mathrm{cd}(Y')}  \chi(\widetilde W_{Y'}(d'))=\sum_{\overline{Y}' \in \mathrm{Rel}(\overline{Y})}\sum_{\overline{d}' \in \mathrm{cd}(\overline{Y}')}  \chi(\widetilde W_{\overline{Y}'}(\overline{d}'))=1.\]

If $\overline{S}$ is not maximal, and $S_i$ is maximal while $S_{3-i}$ is not, then the missing block at $(1,m)$ is not salient, and the subspace in its stratum is necessarily in the image of the subspace at $(0,m)$. Again, to each relative of $\overline{Y}$ there corresponds a unique relative of $Y$, and
\[ \sum_{Y' \in \mathrm{Rel}(Y)}\sum_{d' \in \mathrm{cd}(Y')} \chi(\widetilde W_{Y'}(d'))=\sum_{\overline{Y}' \in \mathrm{Rel}(\overline{Y})} \sum_{\overline{d}' \in \mathrm{cd}(\overline{Y}')}\chi(\widetilde W_{\overline{Y}'}(\overline{d}'))=1.\]

Finally, suppose that $\overline{S}$ is not maximal, but $S_1$ and $S_2$ are maximal.
In this case new relatives appear when we add the bottom row of $Y$ to $\overline{Y}$. We investigate two possible cases individually. 

If the divided block above the new salient block is a full $n-1$/$n$ block, then for each relative of $\overline{Y}$ there are two other relatives, which contain either the label $(n-1)$ half block or the label $n$ half block. If the relatives of $\overline{Y}$ are $\{\overline{Y}_0=\overline{Y},\overline{Y}_1,\overline{Y}_2,\dots \}$, then the relatives of $Y$ are 
\[ \{Y_0=Y,Y_1,Y_2,\dots \}\cup \bigcup_{c\in\{n-1, n\}} \{Y_{0,c}=Y_c,Y_{1,c},Y_{2,c},\dots \}.\]
We obtain these by performing the inverse of the move (R3) in the algorithm of Lemma \ref{lem:1red}.
The addition of the bottom row of $Y$ to an $\overline{Y_i}$  schematically looks like this:
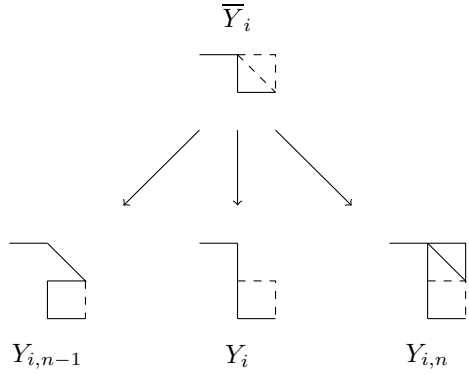
\begin{center}
\begin{tikzpicture}[scale=0.5]
\draw (1,6) -- (1,7);
\draw (1,7) -- (0,7);
\draw (1,6) -- (2,6);
\draw[dashed] (1,7) -- (2,7) --(2,6);
\draw[dashed] (1,7) -- (2,6);

\draw (1,0) -- (1,2);
\draw (1,2) -- (0,2);
\draw (1,0) -- (2,0);
\draw[dashed] (1,1) -- (2,1) --(2,0);

\draw (-4,0) -- (-4,1);
\draw (-4,2) -- (-5,2);
\draw (-4,0) -- (-3,0);
\draw(-4,2) -- (-3,1) -- (-4,1);
\draw[dashed]  (-3,1) -- (-3,0);

\draw (6,0) -- (6,2);
\draw (6,2) -- (5,2);
\draw (6,0) -- (7,0);
\draw (6,2) -- (7,1) -- (7,2) -- (6,2);
\draw[dashed] (6,1) -- (7,1) --(7,0);

\draw[->] (1,5) -- (1,3);
\draw[->] (0,5) -- (-2,3);
\draw[->] (2,5) -- (4,3);
\draw (1,8) node {$\overline{Y}_i$};
\draw (1,-1) node {$Y_i$};
\draw (-4,-1) node {$Y_{i,n-1}$};
\draw (6,-1) node {$Y_{i,n}$};
\end{tikzpicture}
\end{center}
Here the block of label $c$ is in the complement of the Young wall $Y_{i,c}$, while its pair is in $Y_{i,c}$. There is only one contributing closing datum on each of $Y_i$, $Y_{i,n-1}$ and $Y_{i,n}$ extending a contributing closing datum $\overline{d}_i$ on $\overline{Y}_i$. Namely, to the new support block of $Y_i$ we associate both the divided block above it, and every other part of $\overline{d}_i$ is kept constant. We denote these by $d_i$, $d_{i,n-1}$ and $d_{i,n}$, respectively.

We claim that for $c\in\{n-1,n\}$,
\[ \chi(\widetilde W_{Y_{i,c}}(d_{i,c})) = \chi(\widetilde W_{\overline{Y}_i}(\overline{d}_i)). \]

To show this, we define a morphism $\widetilde W_{\overline{Y}_i}(\overline{d}_i) \to \widetilde W_{\overline{Y}_{i,c}}(\overline{d}_{i,c})$, where $\overline{Y}_{i,c}$ is the truncation of $Y_{i,c}$. This morphism is given by the restriction of an ideal $I \in \widetilde W_{\overline{Y}_i}(\overline{d}_i)$ to the union of those cells whose block is missing from $\overline{Y}_{i,c}$. The Young wall $\overline{Y}_{i,c}$ has the same salient blocks as  $\overline{Y}_{i}$ except a half block in the bottom row. Hence, the result is necessarily an ideal, which has the same generators as $I$ except for the cell $V_{0,m,c}$ where it does not have any generator. Therefore, the image of this morphism is in $W_{\overline{Y}_{i,c}}(\overline{d}_{i,c})$; here $(0,m)$ is the position of the salient block-pair in the bottom row of $\overline{Y}_i$, which is also the position of the salient half block of label $\kappa(c)$ in the bottom row of $\overline{Y}_{i,c}$. 

Assume that there are two ideals $I, I' \in \widetilde W_{\overline{Y}_i}(\overline{d}_i)$ which map to the same element of $\widetilde W_{\overline{Y}_{i,c}}(\overline{d}_{i,c})$ under this morphism. Then they only differ in the function at $V_{0,m,c}$, or more precisely in the point of $V_{0,m,c}/\overline{U}$ which represents the subspace in $V_{0,m,c}$. Here $\overline{U}=I\cap \overline{P}_{m,c}=I'\cap \overline{P}_{m,c}$. Then any ideal $I''$ which is their affine linear combination, i.e.~whose corresponding subspace in $V_{0,m,c}$ is a linear combination of those of $I$ and $I'$, is also an element of $\widetilde W_{\overline{Y}_i}(\overline{d}_i)$, mapped to the same ideal as $I$ and $I'$ under the morphism. In particular, the fibers of the morphism are affine spaces.
Taking into account that $\chi(\widetilde W_{Y_{i,c}}(d_{i,c}))=\chi(\widetilde W_{\overline{Y}_{i,c}}(\overline{d}_{i,c}))$, this proves the claim.

Thus,
\begin{gather*}
\chi(\widetilde W_{Y_i}(d_i))+\chi(\widetilde W_{Y_{i,n-1}}(d_{i,n-1}))+\chi(\widetilde W_{Y_{i,n}}(d_{i,n})) \\ =  -\chi(\widetilde W_{\overline{Y}_i}(\overline{d}_i))+\chi(\widetilde W_{\overline{Y}_i}(\overline{d}_i))+\chi(\widetilde W_{\overline{Y}_i}(\overline{d}_i))
   =  \chi(\widetilde W_{\overline{Y}_i}(\overline{d}_i)),
 \end{gather*}
 where in the first equality we used Lemma \ref{lem:chiclosefiber}.

Second, if the divided block above the new salient block is a full 0/1 block, then for each relative of $\overline{Y}$ there is a new relative which contains the label $1$ half block (and possibly some other blocks above it).
Let us denote the relatives of $\overline{Y}$ as $\{\overline{Y}_0=\overline{Y},\overline{Y}_1,\overline{Y}_2,\dots \}$. Then the relatives of $Y$ are 
\[ \{Y_0=Y,Y_1,Y_2,\dots \}\cup \{Y_{0,1}, Y_{1,1},Y_{2,1},\dots \}. \]
We obtain these by performing the inverse of the move (R2) in the algorithm of Lemma \ref{lem:1red}. Schematically:
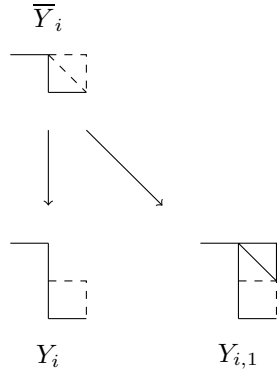
\begin{center}
\begin{tikzpicture}[scale=0.5]
\draw (1,6) -- (1,7);
\draw (1,7) -- (0,7);
\draw (1,6) -- (2,6);
\draw[dashed] (1,7) -- (2,7) --(2,6);
\draw[dashed] (1,7) -- (2,6);

\draw (1,0) -- (1,2);
\draw (1,2) -- (0,2);
\draw (1,0) -- (2,0);
\draw[dashed] (1,1) -- (2,1) --(2,0);


\draw (6,0) -- (6,2);
\draw (6,2) -- (5,2);
\draw (6,0) -- (7,0);
\draw (6,2) -- (7,1) -- (7,2) -- (6,2);
\draw[dashed] (6,1) -- (7,1) --(7,0);

\draw[->] (1,5) -- (1,3);
\draw[->] (2,5) -- (4,3);
\draw (1,8) node {$\overline{Y}_i$};
\draw (1,-1) node {$Y_i$};
\draw (6,-1) node {$Y_{i,1}$};
\end{tikzpicture}
\end{center}
In this case $Y_{i,0}$ cannot appear as a relative, since that would change the 0-weight.
Given a contributing closing datum $\overline{d}_i$ on $\overline{Y}_i$, shifting it appropriately gives a
unique contributing closing datum $d_{i,1}$ on $Y_{i,1}$. On $Y_{i}$, $\overline{d}_i$ induces 
two contributing closing data:
\begin{itemize}
\item $d_i$ is obtained by associating the support block in the bottom row to the salient half-block pair above it;
\item $d_i'$ is obtained by associating the support block in the bottom row to the salient half block of label 1 above it only.
\end{itemize}

Again, using Lemma \ref{lem:chiclosefiber} we get
\begin{eqnarray*}
\chi(\widetilde W_{Y_i}(d_i))+\chi(\widetilde W_{Y_i}(d'_i))+\chi(\widetilde W_{Y_{i,1}}(d_{i,1}))  & = & -\chi(\widetilde W_{\overline{Y}_i}(\overline{d}_i))+\chi(\widetilde W_{\overline{Y}_i}(\overline{d}_i))+\chi(\widetilde W_{\overline{Y}_i}(\overline{d}_i)) \\ & = & \chi(\widetilde W_{\overline{Y}_i}(\overline{d}_i)).
\end{eqnarray*}
Summing over these, we obtain in all cases 
 \[ \sum_{Y' \in \mathrm{Rel}(Y)}\sum_{d' \in \mathrm{cd}(Y')}  \chi(\widetilde W_{Y'}(d))=\sum_{\overline{Y}' \in \mathrm{Rel}(\overline{Y})}\sum_{\overline{d}' \in \mathrm{cd}(\overline{Y}')}  \chi(\widetilde W_{\overline{Y}'}(\overline{d})),\]
which proves the induction step.
\end{proof}

Putting everything together, we obtain the following result, the analogue of Corollary~\ref{cor:Acombinatorial} in type~$D$. 
\begin{theorem} Let $\Delta$ be of type $D_n$. Then the generating function of Euler characteristics of the coarse Hilbert schemes of points of the corresponding singular surface $\SC^2/G_\Delta$ is given by the  combinatorial generating series
\[\sum_{m=0}^\infty \chi\left(\mathrm{Hilb}^m(\SC^2/G_\Delta)\right)q^m = \sum_{Y\in {\mathcal Z}_\Delta^0} q^{{\rm wt}_0(Y)}.
\]
\label{thm:dnsingcells}
\end{theorem}
\begin{proof} We use the decomposition of Theorem~\ref{thm:assdiag}. 
For $Y \in{\mathcal Z}_\Delta'$, we have $\mathrm{Hilb}(\SC^2/G_\Delta)_Y^T=W_{Y}$ and thus 
$\chi(W_{Y})=\chi(\mathrm{Hilb}(\SC^2/G_\Delta)_Y)$. Now use Corollary~\ref{cor:relsum1} to sum the Euler characteristics 
of the strata $W_Y$ along the fibres of the combinatorial reduction map ${\rm red}$ of Lemma~\ref{lem:1red}.
\end{proof}

\begin{example}
\label{ex:singr2}
Let $n=4$ and let $Y\in{\mathcal Z}_\Delta^0$ be the distinguished 0-generated Young wall
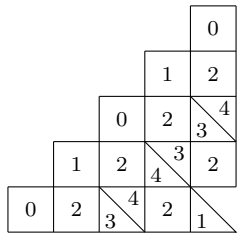
\begin{center}
\begin{tikzpicture}[scale=0.6, font=\footnotesize, fill=black!20]
 \draw (1, 0) -- (5,0);
  \foreach \x in {1,2,3,4}
    {
      \draw (\x, 0) -- (\x,\x);
    }
  
   \foreach \y in {1,2,3,4,5}
       {
            \draw (\y,\y) -- (6,\y);
       }
\draw (5, 0) -- (5,5);
 \draw (6, 1) -- (6,5);

\foreach \x in {0,1,2,3}
        {
        	\draw (\x+2.5, \x+0.5) node {2};
        }
  \foreach \x in {0,1}
       {
           \draw (\x+4.5, \x+0.5) node {2};
          }   
     
     \foreach \x in {0,2}
         {
                	\draw (\x+1.5, \x+0.5) node {0};
                	\draw (\x+2.5, \x+1.5) node {1};
                }
       \draw (5.5, 4.5) node {0};
 \foreach \x in {0,2}
           {
           \draw (\x+3.75,\x+0.75) node {4};
                        	\draw (\x+3.25,\x+0.25) node {3};
                        	\draw (\x+3,\x+1) -- (\x+4,\x+0);
           }
  \foreach \x in {1}
             {
             \draw (\x+3.75,\x+0.75) node {3};
                                     	\draw (\x+3.25,\x+0.25) node {4};
                                     	\draw (\x+3,\x+1) -- (\x+4,\x+0);
             }
     \draw (5,1)--(6,0)--(5,0);
     \draw (5.25,0.25) node {1};
             
\end{tikzpicture}
\end{center}
The parameter space $Z_Y$ of this Young wall $Y$ is isomorphic to that of Example \ref{ex:4sidepyr}, which in turn is isomorphic to that of Example \ref{ex:3sidepyr}. In particular, $Z_Y\cong \SA^1$. The difference is that in this case the salient blocks are the $0$-labelled blocks at $(0,4)$ and $(1,5)$. 

Denote by $Y_3$ and $Y_4$ the 0-generated non-distinguished Young walls
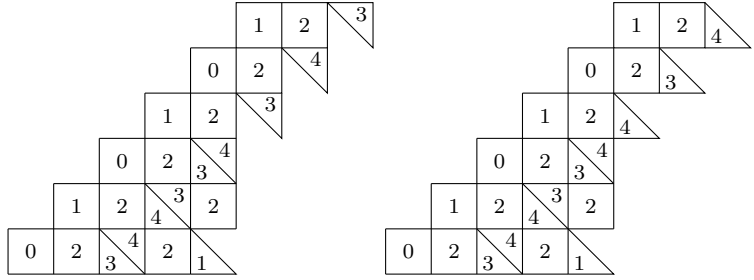
\begin{center}
\begin{tikzpicture}[scale=0.6, font=\footnotesize, fill=black!20]
 \draw (1, 0) -- (5,0);
  \foreach \x in {1,2,3,4}
    {
      \draw (\x, 0) -- (\x,\x);
    }
  
   \foreach \y in {1,2,3,4,5}
       {
            \draw (\y,\y) -- (6,\y);
       }
\draw (5, 0) -- (5,5);
 \draw (6, 1) -- (6,5);

\foreach \x in {0,1,2,3,4,5}
        {
        	\draw (\x+2.5, \x+0.5) node {2};
        }
  \foreach \x in {0,1}
       {
           \draw (\x+4.5, \x+0.5) node {2};
          }   
     
     \foreach \x in {0,2,4}
         {
                	\draw (\x+1.5, \x+0.5) node {0};
                	\draw (\x+2.5, \x+1.5) node {1};
                }
 \foreach \x in {0,2}
           {
           \draw (\x+3.75,\x+0.75) node {4};
                        	\draw (\x+3.25,\x+0.25) node {3};
                        	\draw (\x+3,\x+1) -- (\x+4,\x+0);
           }
  \foreach \x in {1}
             {
             \draw (\x+3.75,\x+0.75) node {3};
                                     	\draw (\x+3.25,\x+0.25) node {4};
                                     	\draw (\x+3,\x+1) -- (\x+4,\x+0);
             }
     \draw (5,1)--(6,0)--(5,0);
     \draw (5.25,0.25) node {1};
     
      \draw (6,4)--(7,3)--(7,4);
      \draw (6.75,3.75) node {3};
      \draw (7,5)--(8,4)--(8,5);
      \draw (7.75,4.75) node {4};
      \draw (8,6)--(9,5)--(9,6)--(8,6);
      \draw (8.75,5.75) node {3};
      \draw (6,4)--(7,4)--(7,6);
      \draw (6,5)--(8,5)--(8,6)--(6,6)--(6,5);
             
\end{tikzpicture}
\begin{tikzpicture}[scale=0.6, font=\footnotesize, fill=black!20]
 \draw (1, 0) -- (5,0);
  \foreach \x in {1,2,3,4}
    {
      \draw (\x, 0) -- (\x,\x);
    }
  
   \foreach \y in {1,2,3,4,5}
       {
            \draw (\y,\y) -- (6,\y);
       }
\draw (5, 0) -- (5,5);
 \draw (6, 1) -- (6,5);

\foreach \x in {0,1,2,3,4,5}
        {
        	\draw (\x+2.5, \x+0.5) node {2};
        }
  \foreach \x in {0,1}
       {
           \draw (\x+4.5, \x+0.5) node {2};
          }   
     
     \foreach \x in {0,2,4}
         {
                	\draw (\x+1.5, \x+0.5) node {0};
                	\draw (\x+2.5, \x+1.5) node {1};
                }
 \foreach \x in {0,2}
           {
           \draw (\x+3.75,\x+0.75) node {4};
                        	\draw (\x+3.25,\x+0.25) node {3};
                        	\draw (\x+3,\x+1) -- (\x+4,\x+0);
           }
  \foreach \x in {1}
             {
             \draw (\x+3.75,\x+0.75) node {3};
                                     	\draw (\x+3.25,\x+0.25) node {4};
                                     	\draw (\x+3,\x+1) -- (\x+4,\x+0);
             }
     \draw (5,1)--(6,0)--(5,0);
     \draw (5.25,0.25) node {1};
     
      \draw (6,4)--(7,3)--(6,3);
      \draw (6.25,3.25) node {4};
      \draw (7,5)--(8,4)--(7,4);
      \draw (7.25,4.25) node {3};
      \draw (8,6)--(9,5)--(8,5);
      \draw (8.25,5.25) node {4};
      \draw (6,4)--(7,4)--(7,6);
      \draw (6,5)--(8,5)--(8,6)--(6,6)--(6,5);
             
\end{tikzpicture}
\end{center}
We have $Y_3, Y_4\in {\mathcal Z}_\Delta'$ and are in fact $0$-generated, but both violate condition (R3) at their fourth row, so they are not distinguished. In the first step of the reduction algorithm of Lemma~\ref{lem:1red}, we remove the half block of label 3 (resp., 4) from $Y_3$ (resp, $Y_4$). Then we remove the blocks of label 2 and 4 (resp., 2 and 3) from the fifth row. Finally we remove the blocks of 1,2 and 3 (reps., 1,2 and 4) from the sixth row. This shows that both Young walls $Y_3, Y_4$ reduce to~$Y$. It is can be seen that these are in fact all the relatives of~$Y$. As explained in Example \ref{ex:4sidepyr}, when the two generators $(v_0,v_0') \in  V_{0,4,0} \times V_{1,5,0}$ are in a special position such that $L_1 v_0, v_0' \in (L_1 L_3)$, then the ideal $(v_0,v_0')$ has Young wall $Y_3$. Similarly, if $L_1 v_0, v_0' \in (L_1 L_4)$, then $(v_0,v_0')$ has Young wall $Y_4$. In fact we can think of the corresponding strata as gluing to one stratum inside the invariant Hilbert scheme $\mathrm{Hilb}^3(\SC^2/D_4)$, the strata of $Y_3$ and $Y_4$ ''patching in'' the gaps of the stratum of $Y$. At least on the level of Euler characteristics, this is what Proposition~\ref{prop:relsum1} shows in full generality.
\end{example}

\section{Type \texorpdfstring{$D_n$}{Dn}: abacus combinatorics}
\label{sect:dnabacus}

\subsection{Young walls and abacus of type  \texorpdfstring{$D_n$}{Dn}}
\label{subsec:Dabacus}
We continue to work with the root system~$\Delta$ of type~$D_n$, and introduce some associated combinatorics. From now on,
we return to the untransformed representation of the type~$D_n$ Young wall pattern introduced in~\ref{subsec:PYW}.

Recalling the Young wall rules (YW1)-(YW4), it is clear that every $Y\in \CZ_\Delta$ can be decomposed as
$Y=Y_1 \sqcup Y_2$, where $Y_1\in \CZ_\Delta$ has full columns only, and $Y_2\in \CZ_\Delta$ has all its columns
ending in a half-block. These conditions define two subsets $\CZ^f_\Delta, \CZ^h_\Delta\subset \CZ_\Delta$.
Because of the Young wall rules, the pair $(Y_1, Y_2)$ uniquely reconstructs $Y$, so we get a bijection
\begin{equation}\label{decomp_D_YW}\CZ_\Delta \longleftrightarrow \CZ^f_\Delta\times \CZ^h_\Delta.\end{equation}

Given a Young wall $Y \in \mathcal{Z}_\Delta$ of type $D_n$, let $\lambda_k$ denote the number of 
blocks (full or half blocks both contributing 1) in the $k$-th vertical 
column. By the rules of Young walls, the resulting positive integers
$\{\lambda_1,\dots,\lambda_r\}$ form a partition $\lambda=\lambda(Y)$ of weight equal to the total weight $|Y|$, 
with the additional property that 
its parts $\lambda_k$ are distinct except when $\lambda_k \equiv 0\;\mathrm{mod}\; (n-1)$. 
Corresponding to the decomposition~\eqref{decomp_D_YW}, we get a decomposition 
$\lambda(Y) = \mu(Y)\sqcup \nu(Y)$. In $\mu(Y)$, no part is congruent to $0$ modulo $(n-1)$, and there 
are no repetitions; all parts in $\nu(Y)$ are congruent to $0$ modulo $(n-1)$ and repetitions are allowed. 
Note that the pair $(\mu(Y), \nu(Y))$ does almost, but not quite, encode $Y$, because of the ambiguity 
in the labels of half-blocks on top of non-complete columns.

We now introduce the type\footnote{Once again, we should call it type $\tilde{D}_n^{(1)}$, but we simplify for 
ease of notation.} $D_n$ abacus, following~\cite{kang2004crystal, kwon2006affine}. 
This abacus is the arrangement of positive integers, called positions, in the following pattern.

\vspace{.1in}
\begin{center}
\begin{tabular}{c c c c c c c c}
1 & \dots & $n-2$ & $n-1$ & $n$ & \dots & $2n-3$ & $2n-2$  \\
$2n-1$ & \dots & $3n-4$ & $3n-3$ & $3n-2$ & \dots & $4n-5$ & $4n-4$\\
\vdots & &\vdots & \vdots & \vdots& &\vdots & \vdots
\end{tabular}
\end{center}
\vspace{.1in}

For any integer $1 \leq k \leq 2n-2$, the set of positions in the $k$-th column of the abacus is the $k$-th 
runner, denoted $R_k$. As in type $A$, positions on the runners are occupied by beads. For $k \not\equiv 0\; \mathrm{mod}\; (n-1)$, 
the runners $R_k$ can only contain normal (uncolored) beads,
with each position occupied 
by at most one bead. On the runners $R_{n-1}$ and $R_{2n-2}$, the beads are colored white and black.
An arbitrary number of white or black beads can be put on each such position, 
but each position can only contain beads of the same color.

Given a type $D_n$ Young wall $Y \in \mathcal{Z}_\Delta$, let $\lambda=\mu\sqcup \nu$ be the corresponding 
partition with its decomposition. For each nonzero part $\nu_k$ of $\nu$, set 
\[ n_k=\#\{1 \leq j \leq l(\mu)\;|\;\mu_j < \nu_k \}\]
to be the number of full columns shorter than a given non-full column. 
The abacus configuration of the Young wall $Y$ is defined to be the set of beads placed at positions
$\lambda_1,\dots,\lambda_r$. The beads at positions $\lambda_k=\mu_j$ are uncolored; the
color of the bead at position $\lambda_k=\nu_l$ corresponding to a column $C$ of $Y$ is
\[
\begin{cases}
\textrm{white,} & \textrm{if the block at the top of } C \textrm{ is } 
\begin{tikzpicture}
 \draw (0, 0) -- (0.3,0.3);
 \draw (0, 0) -- (0.3,0);
 \draw (0.3, 0) -- (0.3,0.3);
\end{tikzpicture}
\textrm{ and } n_l \textrm{ is even;} \\ 
\textrm{white,} & \textrm{if the block at the top of } C \textrm{ is } 
\begin{tikzpicture}
 \draw (0, 0) -- (0.3,0.3);
 \draw (0, 0) -- (0,0.3);
 \draw (0, 0.3) -- (0.3,0.3);
\end{tikzpicture}
\textrm{ and } n_l \textrm{ is odd;}\\
\textrm{black,} & \textrm{if the block at the top of } C \textrm{ is } 
\begin{tikzpicture}
 \draw (0, 0) -- (0.3,0.3);
 \draw (0, 0) -- (0,0.3);
 \draw (0, 0.3) -- (0.3,0.3);
\end{tikzpicture}
\textrm{ and } n_l \textrm{ is even;} \\ 
\textrm{black,} & \textrm{if the block at the top of } C \textrm{ is } 
\begin{tikzpicture}
 \draw (0, 0) -- (0.3,0.3);
 \draw (0, 0) -- (0.3,0);
 \draw (0.3, 0) -- (0.3,0.3);
\end{tikzpicture}
\textrm{ and } n_l \textrm{ is odd.}\\
\end{cases}
\]
One can check that the abacus rules are satisfied, that all abacus configurations satisfying the above rules, 
with finitely many uncolored, black and white beads, can arise, and that the Young wall $Y$ is uniquely determined 
by its abacus configuration.

\begin{example} \label{ex:abacusY}
The abacus configuration associated to the Young wall $Y_6$ of Example \ref{ex:singr1} is
\begin{center}
\begin{tikzpicture}

\draw (0,5) node {$R_1$};
\draw (1,5) node {$R_{2}$};
\draw (2,5) node {$R_{3}$};
\draw (3,5) node {$R_{4}$};
\draw (4,5) node {$R_{5}$};
\draw (5,5) node {$R_{6}$};

\draw (-0.5,4.8) -- (5.5,4.8);
\draw (2.5,5.3) -- (2.5,2.5);
\draw (1.5,5.3) -- (1.5,2.5);
\draw (4.5,5.3) -- (4.5,2.5);
\draw (5.5,5.3) -- (5.5,2.5);

\node at ( 0,4.5)  {1};
\node at ( 1,4.5)  {2};
\node at ( 2,4.5)  {3};
\node at ( 3,4.5)  {4};
\node at ( 4,4.5)  {5};
\node at ( 5,4.5) [circle,draw,inner sep=0pt,minimum size=12pt,fill=gray!50] {6};
\node at ( 0,4) [circle,draw,inner sep=0pt,minimum size=12pt] {7};
\node at ( 1,4) [circle,draw,inner sep=0pt,minimum size=12pt] {8};
\node at ( 2,4) {9};
\node at ( 3,4) [circle,draw,inner sep=0pt,minimum size=12pt] {10};
\node at ( 4,4) [circle,draw,inner sep=0pt,minimum size=12pt] {11};
\node at ( 5,4) [circle,draw,inner sep=0pt,minimum size=12pt] {12};
\node at ( 0,3.5)  {13};
\node at ( 1,3.5)  {14};
\node at ( 2,3.5)  {15};
\node at ( 3,3.5)  {16};
\node at ( 4,3.5)  {17};
\node at ( 5,3.5)  {18};
\node at ( 0.5,3) {\vdots};
\node at ( 3.5,3) {\vdots};
\end{tikzpicture}
\end{center}
\end{example}

\subsection{Core Young walls and their abacus representation} 
\label{subsect:coreabacus}

In parallel with the type $A$ story, we now introduce the combinatorics of core Young walls of type $D_n$, and 
the corresponding abacus moves. On the Young wall side, define a {\em bar} to be a connected set of blocks
and half-blocks, with each half-block occuring once and each block occuring twice. A Young wall 
$Y\in\mathcal{Z}_\Delta$ will be called a {\em core} Young wall, if no bar can be removed from it without 
violating the Young wall rules. For an example of bar removal, see \cite[Example 5.1(2)]{kang2004crystal}. Let ${\mathcal C}_\Delta\subset\mathcal{Z}_\Delta$ denote the set of all 
core Young walls. The following statement is the type $D$ analogue of the discussion of~\ref{subsec:Anabacus}. 

\begin{proposition}{\rm (\cite{kang2004crystal, kwon2006affine})} 
\label{prop:dncoredecomp}
Given a Young wall 
$Y\in \mathcal{Z}_\Delta$, any complete sequence of bar removals through Young walls results in the same
core $\mathrm{core}(Y)\in {\mathcal C}_\Delta$, defining a map of sets
\[ \mathrm{core}\colon \mathcal{Z}_\Delta\to {\mathcal C}_\Delta.\]
The process can be described on the abacus, respects the decomposition~\eqref{decomp_D_YW}, and results
in a bijection
\begin{equation} \CZ_\Delta  \longleftrightarrow  {\mathcal P}^{n+1}  \times  \CC_\Delta\label{Dpart_biject}\end{equation}
where ${\mathcal P}$ is the set of ordinary partitions.
Finally, there is also a bijection
\begin{equation}\label{Dcore_biject} \CC_\Delta\longleftrightarrow \SZ^n.\end{equation}
\end{proposition} 
\begin{proof} Decompose $Y$ into a pair of Young walls $(Y_1, Y_2)$ as above. Let us first consider $Y_1$. 
On the corresponding runners $R_k$, $k \not\equiv 0\; \textrm{ mod}\; (n-1)$,
the following steps correspond to bar removals \cite[Lemma 5.2]{kang2004crystal}.
\begin{enumerate}
 \item[(B1)] If $b$ is a bead at position $s>2n-2$, and there is no bead at position $(s-2n+2)$, then move $b$ 
one position up and switch the color of the beads at 
positions $k$ with $k \equiv 0\; \textrm{ mod}\; (n-1)$, $s-2n+2 < k < s$.
 \item[(B2)] If $b$ and $b'$ are beads at position $s$ and $2n-2-s$ ($1 \leq s \leq n-2$) respectively, then 
remove $b$ and $b'$ and switch the color of the beads 
at positions $k \equiv 0\; \textrm{ mod}\; (n-1), s< k < 2n-2-s$.
\end{enumerate}
Performing these steps as long as possible results in a configuration of beads on the runners $R_k$ with 
$k \not\equiv 0\; \textrm{ mod}\; (n-1)$ with no gaps from above, and for $1 \leq s \leq n-2$, beads on only one
of $R_s$, $R_{2n-2-s}$. This final configuration can be uniquely described by an ordered set of 
integers $\{z_1, \ldots, z_{n-2}\}$, $z_s$ being the number of beads on $R_s$ minus the number of beads 
on $R_{2n-2-s}$ \cite[Remark 3.10(2)]{kwon2006affine}. In the correspondence \eqref{Dcore_biject} this gives $\SZ^{n-2}$. It turns out that the reduction steps in this part of the algorithm can be encoded 
by an $(n-2)$-tuple of ordinary partitions, with the summed weight of these 
partitions equal to the number of bars removed \cite[Theorem 5.11(2)]{kang2004crystal}. 

Let us turn to $Y_2$, represented on the runners $R_k$, $k \equiv 0\; \textrm{ mod}\; (n-1)$. 
On these runners, the following steps correspond to bar removals \cite[Sections 3.2 and 3.3]{kwon2006affine}.
\begin{enumerate} 
\item[(B3)] Let $b$ be a bead at position $s\geq 2n-2$. If there is no bead at position $(s-n+1)$, and the beads at position $(s-2n+2)$ are of the same color as $b$, then shift $b$ up to position $(s-2n+2)$.
 \item[(B4)] If $b$ and $b'$ are beads at position $s\geq n-1$, then move them up to position $(s-n+1)$. If $s-n+1>0$ and this position already contains beads, then $b$ and $b'$ take that same color.
\end{enumerate}
During these steps, there is a boundary condition: there is an imaginary position $0$ in the rightmost column, 
which  is considered to contain unvisible white beads; placing a bead there means 
that this bead disappears from the abacus. 
It turns out that the reduction steps in this part of the algorithm can be described 
by a triple of ordinary partitions, again with the summed weight of these 
partitions equal to the number of bars removed \cite[Proposition 3.6]{kwon2006affine}. On the other hand, the final result can be encoded by a pair of 
$2$-core partitions (see \ref{subsec:Anabacus}).

The different bar removal steps (B1)-(B4) construct the map $c$ algorithmically and uniquely. The stated 
facts about parameterizing the steps prove the existence of the bijection~\eqref{Dpart_biject}. 
To complete the proof of~\eqref{Dcore_biject}, we only need to remark further that the set  
of $2$-core partitions, in our language $A_1$-core partitions, is in bijection with the set of integers
by bijection~\eqref{typeA-cores} in~\ref{subsec:Anabacus} (see also \cite[Remark 3.10(2)]{kwon2006affine}). This gives the remaining $\SZ^{2}$ factor in the bijection \eqref{Dcore_biject}.
\end{proof}

\begin{example}
The abacus configuration associated to the core of the Young wall $Y_6$ of Examples \ref{ex:singr1} and \ref{ex:abacusY} is
\begin{center}
\begin{tikzpicture}

\draw (0,5) node {$R_1$};
\draw (1,5) node {$R_{2}$};
\draw (2,5) node {$R_{3}$};
\draw (3,5) node {$R_{4}$};
\draw (4,5) node {$R_{5}$};
\draw (5,5) node {$R_{6}$};

\draw (-0.5,4.8) -- (5.5,4.8);
\draw (2.5,5.3) -- (2.5,3);
\draw (1.5,5.3) -- (1.5,3);
\draw (4.5,5.3) -- (4.5,3);
\draw (5.5,5.3) -- (5.5,3);

\node at ( 0,4.5)  {1};
\node at ( 1,4.5)  {2};
\node at ( 2,4.5)  {3};
\node at ( 3,4.5)  {4};
\node at ( 4,4.5)  {5};
\node at ( 5,4.5) [circle,draw,inner sep=0pt,minimum size=12pt,fill=gray!50] {6};
\node at ( 0,4) {7};
\node at ( 1,4) {8};
\node at ( 2,4) {9};
\node at ( 3,4) {10};
\node at ( 4,4) {11};
\node at ( 5,4) [circle,draw,inner sep=0pt,minimum size=12pt] {12};
\node at ( 0.5,3.5) {\vdots};
\node at ( 3.5,3.5) {\vdots};
\end{tikzpicture}
\end{center}
\end{example}

We next determine the multi-weight of a Young wall $Y$ in terms 
of the bijections \eqref{Dpart_biject}-\eqref{Dcore_biject}. 
The quotient part is easy: the multi-weight of each bar is $(1,1,2,\ldots, 2, 1,1)$, so in complete 
analogy with the type $A$ situation,
the contribution of the $(n+1)$-tuple of partitions to the multi-weight is easy to compute. Turning to 
cores, under the bijection $\CC_\Delta\leftrightarrow \SZ^n$, the total weight of a core Young wall $Y\in \CC_\Delta$ corresponding 
to $(z_1, \ldots, z_n)\in\SZ^n$ is calculated in \cite[Remark 3.10]{kwon2006affine}: 
\begin{equation} 
\label{eq:dncoreweight}
|Y|=\frac{1}{2}\sum_{i=1}^{n-2} \left((2n-2)z_i^2-(2n-2i-2)z_i\right)+(n-1)\sum_{i=n-1}^n\left(2z_{i}^2+z_{i}\right).
\end{equation}
A refinement of this formula calculates the multi-weight of $Y$.

\begin{theorem} Let $q=q_0q_1q_2^2\dots q_{n-2}^2q_{n-1}q_n$, corresponding to a single bar. 
\label{thm:orbiserdn}
\begin{enumerate}
\item Composing the bijection~\eqref{Dcore_biject} with an appropriate $\SZ$-change of coordinates in the lattice~$\SZ^n$, 
the multi-weight of a core Young wall $Y\in\CC_\Delta$ corresponding to an element $\overline{m}=(m_1,\dots,m_n) \in \SZ^n$
can be expressed as
\[q_1^{m_1}\cdot\dots\cdot q_n^{m_n}(q^{1/2})^{\overline{m}^\top \cdot C \cdot \overline{m}},\]
where $C$ is the Cartan matrix of type $D_n$.

\item The multi-weight generating series 
\[ Z_{\Delta}(q_0,\dots,q_n) = \sum_{Y\in \CZ_\Delta} \underline{q}^{\underline{wt}(Y)}\]
of Young walls for $\Delta$ of type $D_n$ can be written as
\begin{equation*} Z_{\Delta}(q_0,\dots,q_n)=\frac{\displaystyle\sum_{ \overline{m}=(m_1,\dots,m_n) \in \SZ^n }^\infty q_1^{m_1}\cdot\dots\cdot q_n^{m_n}(q^{1/2})^{\overline{m}^\top \cdot C \cdot \overline{m}}}{\displaystyle\prod_{m=1}^{\infty}(1-q^m)^{n+1}}.
\end{equation*}

\item The following identity is satisfied between the coordinates $(m_1,\dots,m_n)$ and $(z_1, \ldots, z_n)$ on~$\SZ^n$:
\begin{equation*} 
\sum_{i=1}^n m_i = -\sum_{i=1}^{n-2}(n-1-i)z_i-(n-1)c(z_{n-1}+z_n)-(n-1)b.
\end{equation*}
Here $z_1+\dots+z_{n-2}=2a-b$ for integers $a \in \SZ$, $b \in \{0,1\}$, and $c=2b-1 \in \{-1,1\}$.
\end{enumerate}
\end{theorem}
\begin{proof}
The coordinate change $(z_1, \ldots, z_n) \mapsto (m_1,\dots,m_n)$ and the multiweight formula of (1), as well as (3), 
follow from somewhat involved but routine calculations which we omit; details can be found in the thesis~\cite{gyenge2016phdthesis} of the first named author. (2) clearly follows from (1) and the preceeding discussion.  
\end{proof}

\subsection{0-generated Young walls and their abacus representations} 

In this section, we characterize the abacus configurations corresponding to Young walls in the 
sets~${\mathcal Z}_\Delta^0$ and~$\mathcal{Z}_{\Delta}^1$ defined in~\ref{subsect:dis0gen}.

Recall conditions (R1)--(R3) on Young walls $Y\in{\mathcal Z}_\Delta$ from~\ref{subsect:dis0gen}. Recall also 
that a Young wall corresponds uniquely to an abacus configuration where the beads are placed at the 
positions $\lambda_1, \dots, \lambda_r$. Finally recall that the quantity $n_k$
denotes the number of full columns shorter than a given non-full column of height $\lambda_k$. 

\begin{lemma} 
\label{lem:maxinvcells}
Conditions (R1)-(R2) on Young walls are equivalent to the following conditions for an abacus configuration.
\begin{enumerate}
\item[{\rm (D1)}] In each row, the rightmost bead is always on the $(2n-2)$-nd ruler, and either all the beads of 
the row are at this position, or the number of beads in this position is odd.
\item[{\rm (D2)}] If a row ends with a white (resp., black) bead corresponding to a column of height $\lambda_k$, then 
$k+n_k$ must be odd (resp.~even). If several beads are placed on this position, which is allowed since it 
is on the ruler $R_{2n-2}$, then this condition refers to the smallest possible $k$.
\item[{\rm (D3)}] The total number of beads in the whole abacus is even, or the total number of beads on the rulers $R_1,\ldots, R_{n-1}$ in the first row is $n-2$.
\item[{\rm (D4)}] The beads on the rulers $R_n\ldots, R_{2n-3}$ are pushed to the right as much as possible in each
row. In any given row, the positions on the rulers $R_1, \ldots, R_{n-1}$ are empty unless all the positions 
on the rulers $R_n, \ldots, R_{2n-2}$ are filled. 
\item[{\rm (D5)}]  The beads on the rulers $R_1 \ldots R_{n-1}$ in any given row 
are either all on the ruler $R_{n-1}$, or on the rules $R_1 \ldots R_{n-2}$, and pushed to the right as much as possible. 
\end{enumerate}
Condition (R3) is equivalent to the following condition. 
\begin{enumerate}
\item[{\rm (D6)}]  Let $s$ be the total number of beads on the rulers $R_1,\ldots, R_{n-1}$ in any given row. 
\begin{enumerate}
\item[(a)] If  $s>n-2$, then all these beads are on $R_{n-1}$.
\item[(b)] If $s\leq n-2$, then all these beads are on the rulers $R_1, \ldots, R_{n-2}$, pushed to the right.
\end{enumerate}
\end{enumerate}
Thus $0$-generated Young walls $Y\in \mathcal{Z}_{\Delta}^1$ correspond to abacus configurations satisfying (D1)-(D5), whereas distinguished $0$-generated Young walls $Y\in \mathcal{Z}_{\Delta}^0$  correspond to those satisfying (D6) also. 
\end{lemma}

\begin{proof}
Two kinds of salient blocks can appear in a Young wall that satisfies (R1)-(R2):
\begin{itemize}
\item label 0 half-blocks,
\item and salient block-pairs of label $n-1/n$.
\end{itemize}
As in the type $A$ case a salient block corresponds to the first bead in a consecutive series of beads in the abacus. More precisely, if there are several columns of height $\lambda_k$, or equivalently, if there are several beads placed on the position $\lambda_k$, then the salient block corresponds to that column of height $\lambda_k$ which has the smallest possible index $k$ among these.

The label 0 blocks always correspond to positions which are on the ruler $R_{2n-2}$. In the odd columns of the type $D$ pattern they are of the shape \begin{tikzpicture}
 \draw (0, 0) -- (0.3,0.3);
 \draw (0, 0) -- (0,0.3);
 \draw (0, 0.3) -- (0.3,0.3);
\end{tikzpicture}
while in the even columns they are of the shape \begin{tikzpicture}
 \draw (0, 0) -- (0.3,0.3);
 \draw (0, 0) -- (0.3,0);
 \draw (0.3, 0) -- (0.3,0.3);
\end{tikzpicture}.
Condition (D2) encodes the fact, that the salient blocks of label 0 are upper triangles in odd columns and lower triangles 
in even columns, and that the color of the beads corresponding to them is also affected by the parity of the 
appropriate $n_k$.

If there is a salient block of label 0, then some columns of the same height, let's say, $\lambda_k$, can follow it. If the first column after them has height $\lambda_k-1$ then on its top there is again a salient block of label 0. This block can only have the opposite orientation than the aforementioned salient block, hence the number of columns of height $\lambda_k$ in this case can only be odd. This gives condition (D1).

Condition (D3) follows again from the absence of label 1 salient blocks. To see this recall that in the bottom row of the type $D$ pattern there are  half blocks which have label 0 in the odd columns and have label 1 in the even columns. Since there are no salient blocks of label $1$ in $Y$, this total number of columns can only be odd if the last column has height 1, the column to the left of it has height 2, and so on. This is can only happen when in the bottom row $s = n-2$.

The fact that there is no salient block of label $2, \dots, n-2$ implies that for each bead on the rulers $R_1, \dots, R_{2n-1}$ there has to be a block placed on its right. The only exception is the ruler $R_{n-2}$. There cannot be any bead on this ruler, except when there is a salient block pair of label $n-1/n$ which corresponds to a hole on $R_{n-1}$. These considerations imply conditions (D4) and (D5).

Condition (D6) corresponds directly to property (R3).
\end{proof}

Given $Y\in{\mathcal Z}_\Delta^0$, let $t_i$ denote the total number of beads in the $i$-th row of its abacus 
representation, and $l_i$ the number of beads in the rightmost position of the $i$-th row. We obtain a sequence of 
pairs $(t_i,l_i)_{i \in \SZ_+}$, only finitely many of which do not equal $(0,0)$. 

\begin{corollary}
\label{lem:invcelldesc}
Given $Y\in{\mathcal Z}_\Delta^0$, the associated sequence of pairs $(t_i,l_i)_{i \in \SZ_+}$ satisfies the following 
conditions. 
\begin{enumerate}
\item For all $i$, $0 \leq l_i \leq t_i$.
\item For all $i$, if $t_i > 0$, then either $l_i=t_i$ is even, or $l_i$ is odd.
\item Either $\sum_i t_i$ is even, or $t_1-l_1=2n-4$.
\end{enumerate}
Conversely, any sequence  $(t_i,l_i)_{i \in \SZ_+}$ satisfying these conditions arises as a sequence associated to at 
least one Young wall 
$Y\in {\mathcal Z}_\Delta^0$. More precisely, the number of different Young walls $Y\in {\mathcal Z}_\Delta^0$ corresponding 
to any given sequence is $2^m$, where $m$ is the number of indices $i$ such that $t_i-l_i > 2n-2$. All Young walls~$Y$ 
corresponding to a single sequence have the same multi-weight, when the weights for labels $n-1$ and $n$ are counted 
together. 
\end{corollary}
\begin{proof}
Condition (1) is clear from the definition of $(t_i,l_i)$. Condition (2) follows from (D1) above. Condition (3) is equivalent to condition (D3).

Conversely, given a sequence $(t_i,l_i)_{i \in \SZ_+}$ satisfying conditions (1)-(3), we can reconstruct a
corresponding $Y\in {\mathcal Z}_\Delta^0$ in its abacus representation as follows. 
On the $i$-th row, $l_i$ beads have to be put on the last position; (D1) is satisfied because of (1). They are 
white if $\sum_{j<i} t_j + \sum_{j>i, t_j \not\equiv 0\;\mathrm{mod}\;n-1} t_j\equiv 1 \; \mathrm{mod}\; 2$, and black otherwise; 
this is just a reformulation of (D2). (D3) is satisfied because of (3). At 
most one bead can be put on each ruler between $R_n$ and $R_{2n-3}$, pushed to the right as much as possible; this is (D4). 
If $t_i-l_i \leq 2n-2$, then the rest of the beads fill up the rulers between $R_1$ and $R_{n-2}$, pushed to the right.
If $t_i-l_i > 2n-2$, then there are no beads in this row on the rulers between $R_1$ and $R_{n-2}$; instead, the
remaining beads are all placed on the $(n-1)$-st ruler, and they can be either white or black. These rules
reconstruct a configuration satisfying (D5)-(D6), and give the stated ambiguity in the reconstruction.
\end{proof}

\subsection{The generating series of distinguished 0-generated walls}

In light of Theorem~\ref{thm:dnsingcells}, in order to complete the proof of our main Theorem~\ref{thmsing} for
type $D$, we need the following combinatorial result, the precise analogue of 
Proposition~\ref{prop:ansubst} in type $A$. 

\begin{theorem} 
\label{thm:dnsubst} Let $\Delta$ be of type $D_n$, and let $\xi$ be a primitive $(2n-1)$-st root of unity. Then
the generating series of the set ${\mathcal Z}_\Delta^0$ of distinguished $0$-generated Young walls is given in 
terms of the generating function of all Young walls by the following substitution:
\[ \sum_{Y\in {\mathcal Z}_\Delta^0} q^{{\rm wt}_0(Y)} = Z_{\Delta}(q_0,\dots,q_n)\Big|_{q_0=\xi^2q, q_1=\dots=q_n=\xi}.
\]
\end{theorem}

Recall that by Lemma \ref{lem:yprimeunique} (1) for each $Y \in {\mathcal Z}_\Delta$ there is a unique Young wall $Y_0 \in {\mathcal Z}_\Delta^0$ which has exactly the same set of 
label $0$ blocks as $Y$. This induces a combinatorial map 
\[p\colon{\mathcal Z}_\Delta \to{\mathcal Z}^0_\Delta,\]
which, in analogy once again with the type $A$ proof, will be the key construction in our proof 
of Theorem \ref{thm:dnsubst}.

\begin{proposition} On the abacus representation of Young walls,
the map $p\colon{\mathcal Z}_\Delta \to{\mathcal Z}^0_\Delta$ corresponds to the following steps:
\begin{enumerate}
\item If a row ends with a white (resp., black) bead on $R_{2n-2}$ corresponding to a column of height $\lambda_k$, and 
$k+n_k$ is even (resp.~odd), then one bead should be moved to the next position, which is the leftmost in the next row. This is applied also on the zeroth position of the abacus, where we assume that there are infinitely many beads. This corresponds to the appearance of a new column in the Young wall represented by the abacus.
\item Every bead on the rulers $R_{1},\dots,R_{2n-4}$ is moved to right as much as possible according to the abacus rules.
\item If there is at least one bead on the rulers $R_1, \dots , R_{2n-3}$ after performing Step 1 on the previous row, and the number of beads on $R_{2n-2}$ is even, then one more bead is moved onto $R_{2n-2}$. If there were beads on $R_{2n-2}$ already, then this step does not change the parity of $k+n_k$ for the rightmost bead. If there were no beads on $R_{2n-2}$ before, then this beads should take the appropriate color and it is possible to see that it can not be moved further with Step 1.
\item Let $s$ be the total number of beads on the rulers $R_1,\dots,R_{n-1}$ after performing the Steps 1-3. If $s > n-2$, then move all these beads on $R_{n-1}$. In this case some of these beads were here previously, so the color of the whole group of beads is already determined. If $s \leq n-2$, then  move them onto the rulers $R_1,\dots,R_{n-2}$, each as right as possible.
\end{enumerate}
\end{proposition}
\begin{proof}
Step 1 enforces condition (D2). It also enforces condition (D3) when applied to the minus first row. Step 2 enforces conditions (D4) and (D5), Step 3 enforces condition (D1), and finally Step 4 enforces condition (D6).
\end{proof}

The fibers of the map $p$ can be described as follows. Given a Young wall $Y\in {\mathcal Z}^0_\Delta$,
we are allowed to move beads to the left and, occasionally, to the right, using the following rules. 
\begin{enumerate}
\item From the last position of the $i-th$ row only one (resp.~zero) bead can be moved to the left if $l_i$ is odd (resp., even).
\item Every other bead is allowed to moved to the left in its row if the result is a valid abacus configuration.
 \item The leftmost bead in a row can be moved to the last position of the previous row. There it will take the color white if $\sum_{j<i} t_j + \sum_{j>i, t_j \not\equiv 0\;\mathrm{mod}\;n-1} t_j\equiv 1 \; \mathrm{mod}\; 2$, and grey otherwise.
 \item If $t_i-l_i \leq 2n-2$, then the beads to the left from the $n-1$-st position are allowed to be moved to the right at most onto the ruler $R_{n-1}$. 
 \item If $t_i-l_i \leq 2n-2$, then any configuration, in which there is at least one bead at the $(n-1)$-st position, 
has to be counted with multiplicity two.
\end{enumerate}

Let us call the beads that can be moved according to these rules~\emph{movable}. For a row with data $(t,l)$, let us also introduce the number $c(t,l)$, which is signed sum of the distance of the beads from the $R_{n-1}$-st ruler, where the sum runs over the movable beads, and the beads to the left of the $R_{n-1}$-st ruler are counted with negative sign, and the beads to the right of it are counted with positive sign. These numbers are listed in the table below:

\[\def\arraystretch{1.5}
\begin{array}{r|c|c}
& l \equiv 0 \; \mathrm{mod}\; 2 & l \equiv 1 \; \mathrm{mod}\; 2 \\ 
\hline 
0 \leq t-l \leq n-2 &  \binom{n-1}{2}-\binom{n-1-t+l}{2}  &  \binom{n}{2}-\binom{n-1-t+l}{2} \\ 
\hline 
n-1 \leq t-l \leq 2n-3 & \binom{n-1}{2}-\binom{t-l-n+1}{2} &  \binom{n}{2}-\binom{t-l-n+1}{2} \\ 
\hline
2n-2 \leq t-l & \binom{n-1}{2} & \binom{n}{2} \\ 
\end{array}
\]

\begin{lemma} The contribution of a row with data $(t,l)$ to the total weight of the fiber is $\xi^{-c(t,l)}$.
\end{lemma}
\begin{proof}
If $l$ is even but nonzero, then according to Corollary \ref{lem:invcelldesc} $l=t$ and there isn't any movable bead.

If $l$ is odd, then there is one movable bead on the $R_{2n-2}$-nd ruler. Assume first that $ t \leq n-1$. Then the expression
\[ \sum_{n_1=0}^{2n-t+l-2}\sum_{n_2=0}^{n_1} \dots \sum_{n_{t-l+1}=0}^{n_{t-l}} (\xi^{-1})^{n_1+\dots+n_{t-l+1}}=\binom{2n-1}{t-l+1}_{\xi^{-1}}\]
counts once every preimage, in which there is at most one bead at the $n-1$-st position. Similarly
\[ (\xi^{-1})^{n-t+l-1} \sum_{n_2=0}^{2n-t} \sum_{n_3=0}^{n_2}\dots \sum_{n_{t-l+1}=0}^{n_{t-l}} (\xi^{-1})^{n_2+\dots+n_{t-l+1}}=(\xi^{-1})^{n-t+l-1} \binom{2n-1}{t-l}_{\xi^{-1}}\]
counts once every preimage, in which there is at least one and at most two beads at the $n-1$-st position, since we moved one bead from the leftmost occupied position to the $n-1$-st position, and we fixed it there. The next term is given by
\[ (\xi^{-1})^{(n-t+l-1)+(n-t+l)} \sum_{n_3=0}^{2n-t} \sum_{n_4=0}^{n_3}\dots \sum_{n_{t-l+1}=0}^{n_{t-l}} (\xi^{-1})^{n_3+\dots+n_{t-l+1}}\]
\[=(\xi^{-1})^{(n-t+l-1)+(n-t+l)} \binom{2n-1}{t-l-1}_{\xi^{-1}},\]
which counts once every preimage, in which there is at least two and at most three beads at the $n-1$-st ruler. Continuing in this fashion and summing up in the end we get
\[ \binom{2n+1}{t-l+1}_{\xi^{-1}}+\sum_{i=1}^{t-l+1} \xi^{-\sum_{j=n-t+l-1}^{n-t+l-2+i} j} \binom{2n-1}{t-l+1-i}_{\xi^{-1}}. \]
It can be checked that $\binom{2n-1}{k}_{\xi^{-1}}=0$ unless $k=0$, in which case it is equal to $1$. 
Therefore, the only surviving part is the one with $i=t-l+1$:
\[ \xi^{-\sum_{j=n-t+l-1}^{n-1} j}=\xi^{\binom{n}{2}-\binom{n-1-t+l}{2}}=\xi^{-c(t,l)}. \]

The proofs of the remaining two cases, when $l$ is odd, are very similar. The only difference in the case $2n-2 \leq t-l$ is that first the preimages with zero or one extra movable beads at $R_{n-1}$ have to be counted, then the preimages with two or three extra movable beads, etc. As in the case $ t \leq n-1$, the only nonzero term is the one where in the beginning all the movable beads have been shifted to the $R_{n-1}$-st ruler, and in this case the powers of $\xi^{-1}$ sum up to $\binom{n}{2}=c(t,l)$. 
\end{proof}
\begin{corollary}
\label{cor:Dn:subst}
Let $Y\in {\mathcal Z}^0_\Delta$ be a distinguished 0-generated Young wall described by the data $\{(t_i,l_i)_i\}$. Then
\[ 
\begin{aligned}
 \sum_{Y' \in p_{\ast}^{-1}(Y)}\underline{q}^{\underline{\mathrm{wt}}(Y')}\Big|_{q_1=\dots=q_n=\xi,q_0=\xi^2q}& =\underline{q}^{\underline{\mathrm{wt}}(Y)}\Big|_{q_1=\dots=q_n=\xi,q_0=\xi^{-(2n-3)}q} \cdot \xi^{-\sum_i c(t_i,l_i)}\\
& =q^{\mathrm{wt}_0(Y)} \xi^{\sum_{j\neq 0} (\mathrm{wt}_j(Y)-\dim\rho_j \cdot \mathrm{wt}_0(Y))-\sum_i c(t_i,l_i)} 
\end{aligned}
\]
\end{corollary}

\begin{lemma}
\label{lem:core0gen}
The core of a 0-generated Young wall is a 0-generated Young wall.
\end{lemma}
\begin{proof} With each reduction step (B1)-(B4) we always remove a bar. In the original Young wall, the salient blocks were only label 0 half blocks and salient block-pairs of label $n-1/n$. 

A similar reasoning as in the type $A$ case shows that no salient block of label $2,\dots,n-2$ can appear after we perform the step (B1) until possible. The same is true with (B2), since if we can perform it on a pair of beads in a row, then we can always perform it on the beads between them. More precisely, it can be seen that even label 1 salient blocks cannot appear during these two steps because the parity conditions in (D1) and (D2) is always maintained by the reduction steps.

The parity conditions in (D1) and (D2) are maintained by the step (B3) as well. If we perform (B4) on a pair of beads, then we can perform it on this pair as long as they disappear from the abacus. Hence, when performing (B4) until possible we also get back the good parities. After the reduction is completed there cannot be any bead on the rulers $R_1,\dots,R_{n-2}$, and all the beads on the rulers $R_{n},\dots,R_{2n-3}$ are right-adjusted. Therefore the conditions (D4) and (D5) are also satisfied. 

If the total number of beads was initially even, then since no label 1 salient block can appear, the total number of beads in the end must be even as well. So the final Young wall will satisfy (D3). 
If in the total number of beads in the initial abacus configuration is odd, then in the first row $t_1-l_1=2n-4$. Hence, the beads on ruler $R_{2n-2}$ in the first row are necessarily white and one of them can be taken away from the abacus with the step (B3), together with the beads on the other rulers using (B2). This is an odd number of beads removed from the abacus. After this the total number of beads is even, so we continue as in the even case.
\end{proof}

\begin{proof}[Proof of Theorem~\ref{thm:dnsubst}] In light of Corollary~\ref{cor:Dn:subst}, it remains to show that 
\[\xi^{\sum_{j\neq 0} (\mathrm{wt}_j(Y)-\dim\rho_j \cdot \mathrm{wt}_0(Y))-\sum_i c(t_i,l_i)}=1.\]
We follow in the line of the proof the $A_n$ case.

\textit{Step 1: Reduction to 0-generated cores.} According to Lemma \ref{lem:core0gen} the core of a 0-generated Young wall is a 0-generated core. It is immediate from the definition of $c(t,l)$ that the steps (B1), (B2) and (B4) leave the sum $\sum_i c(t_i,l_i)$ unchanged, while (B3) is a bit more complicated. Indeed,
\begin{itemize}
\item (B1) removes one movable bead from a row, and adds one to another on the same ruler;
\item (B2) removes two movable beads, but these two contribute with opposite signs into $c(t,l)$;
\item (B4) either moves non-movable beads from $R_{2n-2}$ onto $R_{n-1}$, or beads from $R_{n-1}$ into non-movable beads on $R_{2n-2}$.
\end{itemize}
Moving beads on $R_{n-1}$ according to (B3) does not affect the numbers $c(t,l)$. If (B3) moves a bead on the ruler $R_{2n-2}$ between rows having $l$ of different parity, then the sum of the movable beads at the $(2n-2)$-nd positions of the two rows is constant, so $\sum_i c(t_i,l_i)$ does not change after such a step. If (B3) were able to move a bead on $R_{2n-2}$ between rows both having odd $l$'s, then this can happen only if they have the same color and if there is no bead on $R_{n-1}$ between them. This means that either there must be an even number of beads between them and they should have same kind of top block, or odd number of beads and different kind of top blocks. Both cases are forbidden in 0-generated Young walls due to Lemma \ref{lem:maxinvcells}. For the same reasons (B3) cannot move beads on $R_{2n-2}$ between rows both having even $l$'s.

It follows from these considerations that for all $Y \in \CZ_{\Delta}^0$
\[ \xi^{\sum_{j\neq 0} (\mathrm{wt}_j(Y)-\dim\rho_j \cdot \mathrm{wt}_0(Y))-\sum_i c(t_i,l_i)}=\xi^{\sum_{j\neq 0} (\mathrm{wt}_j(\mathrm{core}(Y))-\dim \rho_j\cdot\mathrm{wt}_0(\mathrm{core}(Y))-\sum_i c(t_i,l_i)}. \]

\textit{Step 2: Reduction to distinguished 0-generated cores.}

As described above, for any 0-generated core $Y$ there is a decomposition as $\lambda(Y)=\mu(Y) \cup \nu(Y)$, where $\mu(Y) \in \mathcal{C}^1$, $\nu(Y) \in \mathcal{C}^2$, and we have $\nu(Y)=\nu^{(0)}(Y)+\nu^{(1)}(Y)$, where $\nu^{(0)}(Y)$ and $\nu^{(1)}(Y)$ are two-cores, and parts in $\nu^{(1)}(Y)$ have colors given by their parity. Since $Y$ is 0-generated, the largest part of $\nu(Y)$ has to be even, otherwise there is no bead at the rightmost position in the last row of the abacus. Therefore, the abacus represenation of $Y$ can be described as follows. 
\begin{enumerate}
\item There is no bead on the rulers $R_k$ for $1\leq k \leq n-2$;
\item On the ruler $R_{n-1}$ all the beads are of the same color, the beads are at the first, let's say, $m$ positions, exactly one bead at each.
\item The positions in the first $m$ rows of the rulers $R_k$ for $n\leq k \leq 2n-3$ are all filled up with beads, the other beads are pushed to the right as much as possible.
\item There is at least $m$ beads on the ruler $R_{2n-2}$, one at each of the first positions and there is no space between them. The first $m$ of these are all white. The total number of the rest is even, and half of them is white, half of them is black.
\end{enumerate}
The abacus of a typical 0-generated core looks like this:
\begin{center}
\begin{tikzpicture}

\draw (0,5) node {$R_1$};
\draw (1,5) node {\dots};
\draw (2,5) node {$R_{n-2}$};
\draw (3,5) node {$R_{n-1}$};
\draw (4,5) node {$R_{n}$};
\draw (5,5) node {\dots};
\draw (6,5) node {$R_{2n-3}$};
\draw (7,5) node {$R_{2n-2}$};

\draw (-0.5,4.8) -- (7.5,4.8);
\draw (3.5,5.3) -- (3.5,-0.3);
\draw (2.5,5.3) -- (2.5,-0.3);
\draw (6.5,5.3) -- (6.5,-0.3);
\draw (7.5,5.3) -- (7.5,-0.3);

\node at (3,4.5) [circle,draw,fill=gray!50] {};
\draw (5,4.5) node {\dots};
\node at ( 4,4.5) [circle,draw] {};
\node at ( 6,4.5) [circle,draw] {};
\node at ( 7,4.5) [circle,draw] {};

\draw (3,3.9) node {\vdots};
\draw (4,3.9) node {\vdots};
\draw (6,3.9) node {\vdots};
\draw (7,3.9) node {\vdots};

\node at (3,3) [circle,draw,fill=gray!50] {};
\draw (5,3) node {\dots};
\node at ( 4,3) [circle,draw] {};
\node at ( 6,3) [circle,draw] {};
\node at ( 7,3) [circle,draw] {};

\node at ( 4,2.5) [circle,draw] {};
\node at ( 5,2.5) [circle,draw] {};
\node at ( 6,2.5) [circle,draw] {};
\node at ( 7,2.5) [circle,draw,fill=gray!50] {};
\node at ( 5,2) [circle,draw] {};
\node at ( 6,2) [circle,draw] {};
\node at ( 7,2) [circle,draw] {};
\node at ( 6,1.5) [circle,draw] {};
\node at ( 7,1.5) [circle,draw,fill=gray!50] {};
\node at ( 6,1) [circle,draw] {};
\node at ( 7,1) [circle,draw] {};
\node at ( 6,0.5) [circle,draw] {};
\node at ( 7,0.5) [circle,draw,,fill=gray!50] {};
\node at ( 7,0) [circle,draw] {};
\end{tikzpicture}
\end{center}

It is not true that the core of a distinguished 0-generated Young wall is always distinguished. But we can reduce each non-distinguished core further using the reduction map $\mathrm{red}$ and then taking the core of the result using the steps (B1)-(B4) again. The result of this is a distinguished 0-generated core. The first step corresponds to shifting the beads on $R_{n-1}$ one step left. This decreases $\sum_i c(t_i,l_i)$ by $m$, and decreases $\sum_{j\neq 0} (\mathrm{wt}_j(Y)-\dim\rho_j \cdot \mathrm{wt}_0(Y))$ also by $m$. The second step does not change these numbers by the considerations of Step 1. So we are done once we know the statement for distinguished 0-generated cores. 

We remark here that during the first step the color of every second bead on the ruler $R_{2n-2}$ changes. In the end the color of the beads on this ruler will alternate.

\textit{Step 3: The case of distinguished 0-generated cores.}

Let $(z_1,\dots,z_n)$ be the integer tuple assigned to $Y$ as in \ref{subsect:coreabacus}. Then $z_{n-1}=z_n$ since there is no bead on $R_{n-1}$ and the color of the beads on $R_{2n-2}$ alternates. In particular, $z_{n-1}+z_{n}$ is an even number.

We claim that \[ \sum_{1=1}^{n}m_n=\sum_i c(t_i,l_i).\]
Indeed, by Theorem \ref{thm:orbiserdn} (3), $\sum_{1=1}^{n}m_n=-\sum_{i=1}^{n-2}(n-1-i)z_i-(n-1)c(z_{n-1}+z_n)-(n-1)b$.
The number of movable beads on $R_{2n-2-i}$ is $-z_i$ if $1 \leq i \leq n-1$, and it is $-c(z_{n-1}+z_n)-b$ if $i=0$. This proves the claim.

By Theorem \ref{thm:orbiserdn} (1), for a core
\[ \underline{q}^{\underline{\mathrm{wt}}(Y)}=q_1^{m_1}\cdot\dots\cdot q_n^{m_n}(q^{1/2})^{\overline{m}^\top \cdot C \cdot \overline{m}}. \]
As in the type $A$ case, on the right hand side of this expression $q_0$ only appears in the $q^{1/2}$-term. Hence
\[ \frac{1}{2}\overline{m}^\top \cdot C \cdot \overline{m}=\mathrm{wt}_0(Y),\]
and
\[ (q^{1/2})^{\overline{m}^\top \cdot C \cdot \overline{m}}\Big|_{q_1=\dots=q_n=\xi,q_0=\xi^{2}q}=q^{\mathrm{wt}_0(Y)}. \]
On the other hand,
\[ \underline{q}^{\underline{\mathrm{wt}}(Y)}\Big|_{q_1=\dots=q_n=\xi,q_0=\xi^{2}q}=q^{\mathrm{wt}_0(Y)}\xi^{\sum_{j\neq 0} (\mathrm{wt}_j(Y)-\dim \rho_j \cdot \mathrm{wt}_0(Y))}.\]
Therefore
\[\xi^{\sum_{j\neq 0} (\mathrm{wt}_j(Y)-\dim \rho_j \cdot \mathrm{wt}_0(Y))}=\xi^{\sum_{i=1}^{n}m_n}=\xi^{\sum_i c(t_i,l_i)},\]
and so indeed
\[\xi^{\sum_{j\neq 0} (\mathrm{wt}_j(Y)-\dim \rho_j \cdot \mathrm{wt}_0(Y))-\sum_i c(t_i,l_i)}=1.\]
\end{proof}


\appendix
\section{Background on representation theory}
\label{sect:apprepr}

This section plays no logical role in our paper, but it provides important background. For further discussion about the role of 
representation theory, see the announcement~\cite{gyenge2015announcement}.

\subsection{Affine Lie algebras and extended basic representations}
\label{extbasic}

Let $\Delta$ be an irreducible finite-dimensional root system, corresponding to a complex finite dimensional 
simple Lie algebra $\frakg$ of rank $n$. Attached to $\Delta$ is also an (untwisted) affine Lie 
algebra $\tilde\frakg$, 
but a slight variant will be more interesting for us, see e.g.~\cite[Sect 6]{etingof2012symplectic}. 
Denote by $\widetilde{\frakg\oplus\SC}$ the  Lie algebra that is 
the direct sum of the affine Lie algebra $\tilde\frakg$ and an infinite Heisenberg algebra $\frakheis$, 
with their centers identified.

Let $V_0$ be the basic representation of $\tilde\frakg$, the level-1 representation with highest 
weight $\omega_0$. Let $\CF$ be the standard
Fock space representation of $\frakheis$, having central charge 1. Then $V=V_0\otimes \CF$ is a representation of 
$\widetilde{\frakg\oplus\SC}$ that we may call the {\em extended basic representation}. By the Frenkel--Kac theorem \cite{frenkel1980basic}, in fact
\[
V \cong \CF^{n+1}\otimes \SC[Q_\Delta],
\]
where $Q_\Delta$ is the root lattice corresponding to the root system $\Delta$. Here, for $\beta\in \SC[Q_\Delta]$, 
$\CF^{n+1}\otimes e^\beta$ is the sum of weight subspaces of weight $\omega_0 - \left(m+\frac{\langle \beta,\beta\rangle}{2}\right)\delta + \beta$, $m \geq 0$.
Thus, we can write the character of this representation as
\begin{equation} 
\label{eq:extcharformula}
\mathrm{char}_V(q_0, \ldots, q_n) = e^{\omega_0} \left(\prod_{m>0}(1-q^m)^{-1}\right)^{n+1} \cdot \sum_{ \beta \in Q_\Delta } q_1^{\beta_1}\cdot\dots\cdot q_n^{\beta_n}(q^{1/2})^{\langle \beta,\beta\rangle}\;,
\end{equation}
where $q=e^{-\delta}$, and $\beta=(\beta_1, \ldots, \beta_n)$ is the expression of an element of the finite type root lattice in terms of the simple roots. 
 
\begin{example} For $\Delta$ of type $A_n$, we have $\frakg={\mathfrak{sl}}_{n+1}$, $\tilde\frakg=\widetilde{\mathfrak{sl}}_{n+1}$, $\widetilde{\frakg\oplus\SC}=\widetilde{\mathfrak{gl}}_{n+1}$.
In this case there is in fact a natural vector space isomorphism $V\cong \CF$ with Fock space itself, see e.g.~\cite[Section 3E]{tingley2011notes}.
\end{example} 

\subsection{Affine crystals}
\label{sect:affinecryst}
The basic representations $V_0, V$ of $\tilde\frakg$, $\widetilde{\frakg\oplus\SC}$ respectively
can be constructed on vector spaces spanned by explicit ``crystal'' bases. Crystal bases have many combinatorial
models; in types $A$ or $D$, the sets denoted with $\CZ_\Delta$ in the main part of our paper provide one 
possible combinatorial model~\cite{kang2004crystal} for the crystal basis for the basic representation. 
More precisely, given $\Delta$ of type $A$ or $D$, there is a combinatorial condition which singles out a 
subset $\mathcal{Y}_\Delta\subset \mathcal{Z}_\Delta$. 
The basic representation $V_0$ of $\tilde{\frakg}$ has a basis \cite{misra1990crystal,kang2003crystal}  
in bijection with elements of $\mathcal{Y}_\Delta$.
The extended basic representation $V$ of $\widetilde{\frakg\oplus\SC}$ has a basis in bijection with elements 
of $\mathcal{Z}_\Delta$. The canonical embedding $V_0\subset V$, defined by the vacuum vector inside 
Fock space $\CF$, is induced by the inclusion $\mathcal{Y}_\Delta\subset \mathcal{Z}_\Delta$.

\subsection{Affine Lie algebras and Hilbert schemes}
\label{sect:aff_lie_hilb}

As before, let $\Gamma<\mathrm{SL}(2,\SC)$ be a finite subgroup and let $\Delta\subset\widetilde\Delta$ be the corresponding finite and affine Dynkin diagrams. It is a well-known fact that the equivariant Hilbert schemes $\mathrm{Hilb}^{\rho}([\SC^2/\Gamma])$ for all finite dimensional representations $\rho$ of $G$ are Nakajima quiver varieties~\cite{nakajima2002geometric} associated to $\widetilde\Delta$, with dimension vector determined by $\rho$, and a specific stability condition (see \cite{fujii2005combinatorial, nagao2009quiver} for more details for type $A$).

Nakajima's general results on the relation between the cohomology of quiver varities and Kac-Moody algebras, specialized to this case, imply~\cite{nakajima2002geometric} that the direct sum of all cohomology groups $H^*(\mathrm{Hilb}^{\rho}([\SC^2/G]))$ is graded isomorphic to the extended basic representation~$V$ of the corresponding extended affine Lie algebra $\widetilde{\frakg\oplus\SC}$ defined in~\ref{extbasic} above. By \cite[Section 7]{nakajima2001quiver}, these quiver varieties have no odd cohomology. Thus the character formula~(\ref{eq:extcharformula}) implies Theorem~\ref{thmgenfunct} in all types~$A$,~$D$, and~$E$.

\section{Joins}
\label{sec:joins}

Recall~\cite{altman1975joins} that the {\em join} $J(X,Y)\subset\SP^n$ of two projective 
varieties $X, Y\subset\SP^n$ is the locus of 
points on all lines joining a point of $X$ to a point of $Y$ in the ambient projective space.   
One well-known example of this construction is the following. Let $L_1\cong\SP^k$ and $L_2\cong\SP^{n-k-1}$ 
be two disjoint projective linear subspaces of $\SP^n$. 

\begin{lemma}\label{lemma_linear_join} The join $J(L_1, L_2)\subset\SP^n$ equals $\SP^n$. Moreover, the 
locus $\SP^n\setminus (L_1\cap L_2)$ is covered by lines uniquely: for every 
$p\in\SP^n\setminus (L_1\cap L_2)$, there exists a unique line $\SP^1\cong p_1p_2\subset\SP^n$ with $p_i\in L_i$,
containing $p$. \end{lemma}

Let now $H\subset \SP^n$ be a hyperplane not containing the $L_i$, which we think of as the hyperplane
``at infinity''. Let $V=\SP^n\setminus H\cong\SA^n$. Let $\overline L_i=L_i\cap H$, and let $L_i^o=L_i\setminus \overline L_i=L_i \cap V$ be 
the affine linear subspaces in $V$ corresponding to $L_i$. Finally let
$X=J(\overline L_1, L_2)\cong\SP^{n-1}$ and $X^o=X\setminus (X\cap H)\cong \SA^{n-1}$. 

\begin{lemma} \label{lemma_linear_proj}
Projection away from $L_1$ defines a morphism $\phi\colon X\to L_2$, which is an affine fibration 
with fibres isomorphic to~$\SA^k$. $\phi$ restricts to a morphism $\phi^o\colon X^o\to L_2^o$, which is a 
trivial affine fibration over $L_2^o\cong \SA^{n-k-1}$. 
\end{lemma} 

In geometric terms, the map $\phi$ is defined on $X\setminus L_2$ as follows: take $p\in X\setminus L_2$,
find the unique line ${p_1p_2}$ passing through it, with $p_1\in \overline L_1$ and 
$p_2\in L_2$; then $\phi(p)=p_2$.

Let now $U$ be a projective subspace of $H$ which avoids $\overline L_2$. Let $U_1\subset U$ be a codimension one linear subspace, and $W=U\setminus U_1$ its affine complement. In the main text, we need the following statement. 
\begin{lemma} 
\label{lem:affjoinchar}
$\chi((J(L_2^o,W)\setminus L_2^o) \cap V)=0$.
\end{lemma}
\begin{proof}
With the same argument as in Lemma \ref{lemma_linear_proj}, $J(L_2^o,W) \cap V$ is a fibration over $L_2^o$ with fiber $\mathrm{Cone}(W)$, and $(J(L_2^o,W)\setminus L_2^o) \cap V$ is a fibration over $L_2^o$ with fiber $\mathrm{Cone}(W)\setminus \{\mathrm{vertex}\}$. Since $\mathrm{Cone}(W)\setminus \{\mathrm{vertex}\}=\SC^{\ast} \times W$, the projection from $(J(L_2^o,W)\setminus L_2^o) \cap V\cong L_2^o \times W \times \SC^{\ast}$ to $L_2^o \times W$ has fibers $\SC^{\ast}$. The lemma follows.
\end{proof}

We finally recall the base-change property of joins.

\begin{lemma} \cite[B1.2]{altman1975joins}  
\label{lem:joinbasechange}
Let $S$ be an arbitrary scheme. Then for schemes $X,Y\subset\SP^n_S$ and an $S$-scheme $T$, we have the 
following equality in $\SP^n_T$:
\[ J(X\times_S T, Y \times_S T) = J(X,Y)\times_S T.\]
\end{lemma}

\bibliographystyle{amsplain}
\bibliography{simplehilb}

\end{document}